\documentclass[final,onefignum,onetabnum,pagebackref]{siamart190516}

\usepackage{lipsum}
\usepackage{amsfonts}
\usepackage{graphicx}
\usepackage{epstopdf}
\usepackage{mleftright}
\usepackage{mathtools}
\usepackage{soul}
\usepackage{stmaryrd}
\usepackage{xspace}
\usepackage{algorithm, algpseudocode, algorithmicx}
\usepackage{subcaption}
\usepackage{cleveref}

\ifpdf
  \DeclareGraphicsExtensions{.eps,.pdf,.png,.jpg}
\else
  \DeclareGraphicsExtensions{.eps}
\fi

\newsiamremark{remark}{Remark}
\newsiamremark{hypothesis}{Hypothesis}
\crefname{hypothesis}{Hypothesis}{Hypotheses}
\newsiamthm{claim}{Claim}
\newsiamthm{fact}{Fact}

\headers{Error and Variance Estimation}{E. N. Epperly and J. A. Tropp}

\title{Efficient Error and Variance Estimation \\ for Randomized Matrix Computations\thanks{This is a preprint of \emph{Efficient error and variance estimation for randomized matrix computations} (\url{https://doi.org/10.1137/23M1558537}), which appeared in the SIAM Journal on Scientific Computing on February 8, 2024. An earlier preprint version of this paper, with less material, was titled \emph{Jackknife variability estimation For randomized matrix computations}.
\funding{This material is based upon work supported by the U.S. Department of Energy, Office of Science, Office of Advanced Scientific Computing Research, Department of Energy Computational Science Graduate Fellowship under Award Number DE-SC0021110.  JAT was supported in part by ONR Awards N00014-17-1-2146 and N00014-18-1-2363 and NSF FRG Award 1952777.}}}

\author{Ethan N. Epperly\thanks{Department of Computing and Mathematical Sciences, California Institute of Technology, Pasadena, CA 91125 USA 
    (\email{eepperly@caltech.edu}, \email{jtropp@cms.caltech.edu}).}
\and Joel A. Tropp\footnotemark[2]
}

\usepackage{amsopn}

\newcommand{\real}{\mathbb{R}}
\newcommand{\complex}{\mathbb{C}}
\newcommand{\field}{\mathbb{K}}

\DeclareMathOperator{\tr}{tr}
\DeclareMathOperator{\diag}{diag}
\newcommand{\mat}[1]{\boldsymbol{#1}}
\renewcommand{\vec}[1]{\boldsymbol{#1}}
\newcommand{\lowrank}[1]{\mleft\llbracket #1 \mright\rrbracket}
\newcommand{\norm}[1]{\mleft\| #1 \mright\|}

\newcommand{\expmat}[1]{\begin{bmatrix} #1 \end{bmatrix}}
\newcommand{\twobytwo}[4]{\expmat{#1 & #2 \\ #3 & #4}}
\newcommand{\twobyone}[2]{\expmat{#1 \\ #2}}
\newcommand{\onebytwo}[2]{\expmat{#1 & #2}}

\DeclareMathOperator{\Var}{Var}
\DeclareMathOperator{\expect}{\mathbb{E}}

\DeclarePairedDelimiterX\condexphelper[2]{[}{]}{#1 \,\delimsize\vert\, #2}

\newcommand{\order}{\mathcal{O}}

\newcommand{\set}[1]{\mathsf{#1}}

\renewcommand{\hat}[1]{\widehat{#1}}

\newcommand{\matabssq}[1]{\mleft| #1 \mright|^2} 
\newcommand{\QR}{\textsf{QR}\xspace}

\newcommand{\uinorm}[1]{{\mleft\vert\kern-0.25ex\mleft\vert\kern-0.25ex\mleft\vert #1 
        \mright\vert\kern-0.25ex\mright\vert\kern-0.25ex\mright\vert}}
\DeclareMathOperator{\Err}{Err}

\DeclareMathOperator{\Std}{Std}
\DeclareMathOperator{\Jack}{Jack}
\newcommand{\schatten}[1]{{\left\vert\kern-0.25ex\left\vert\kern-0.25ex\left\vert #1 
        \right\vert\kern-0.25ex\right\vert\kern-0.25ex\right\vert}}
 \newcommand{\mycaption}[2]{\caption[#1]{\textbf{#1.}\ {#2}}}
\newcommand{\Id}{\mathbf{I}}
\newcommand{\evec}{\mathbf{e}}
\newcommand{\actionbox}[1]{\begin{center} \vspace{0.5pc}
\fbox{ \begin{minipage}{0.9\textwidth}
\begin{center}#1\end{center}
\end{minipage}
} \vspace{0.5pc}
\end{center}}
\usepackage{listings}
\usepackage{color} 
\definecolor{mygreen}{RGB}{28,172,0} 
\definecolor{mylilas}{RGB}{170,55,241}
\crefname{lstlisting}{Program}{Programs}
\Crefname{lstlisting}{Program}{Programs}

\usepackage{amsmath}

\ifpdf
\hypersetup{
  pdftitle={Efficient Error and Variance Estimation for Randomized Matrix Computations},
  pdfauthor={E. N. Epperly and J. A. Tropp}
}
\fi

\begin{document}

\maketitle

\begin{abstract}
  Randomized matrix algorithms have become workhorse tools in scientific computing and machine learning.
  To use these algorithms safely in applications, they should be coupled with posterior error estimates to assess the quality of the output.
  To meet this need, this paper proposes two diagnostics: a leave-one-out error estimator for randomized low-rank approximations and a jackknife resampling method to estimate the variance of the output of a randomized matrix computation.
  Both of these diagnostics are rapid to compute for randomized low-rank approximation algorithms such as the randomized SVD and randomized Nystr\"om approximation, and they provide useful information that can be used to assess the quality of the computed output and guide algorithmic parameter choices.
\end{abstract}



\section{Introduction}

In recent years, randomness has become an essential tool in the design of matrix algorithms \cite{DM16,HMT11,MT20,MDM+23,Woo14}, with randomized algorithms proving especially effective for matrix low-rank approximation.
To use these algorithms safely in practice, they should be supported by \emph{posterior error estimates} and other quality metrics that inform the user about the accuracy of the computational output.

This paper presents two diagnostic tools for randomized matrix computations:
\begin{itemize}
\item First, we provide a \emph{leave-one-out posterior estimate} for the error $\norm{\mat{A} - \mat{X}}_{\rm F}$ for a low-rank approximation $\mat{X}$ to a matrix $\mat{A}$ produced by randomized algorithms such as the randomized SVD or randomized Nystr\"om aproximation. 
\item Second, we present a jackknife method for estimating the \emph{variance} $\Var(\mat{X}) \coloneqq \expect \norm{\mat{X} - \expect \mat{X}}_{\rm F}^2$ of the matrix output $\mat{X}$ of a randomized algorithm.
The variance is a useful diagnostic: If the computation is sensitive to the randomness used by the algorithm, the computed output should be treated with suspicion.
\end{itemize}
By using novel downdating formulas (see \cref{eq:nys_update,eq:rsvd_update} below), we can rapidly compute both of these estimators for widely used low-rank approximation methods such as the randomized SVD and randomized Nystr\"om approximation.
Our diagnostics are also \emph{sample-efficient} in the sense that they require no additional matrix--vector products or other information beyond that used in the original algorithm.
The speed and efficiency of these diagnostics make them compelling additions to workflows involving randomized matrix computation.

\subsection{Leave-one-out error estimation} \label{sec:loo_intro}

We begin by motivating our first diagnostic, a leave-one-out estimator for the error of a low-rank matrix approximation.
For concreteness, we introduce this estimate in the context of Nystr\"om approximation of positive-semidefinite (psd) matrices; see \cref{sec:loo} for a more general setting 

Suppose we want to approximate a psd matrix $\mat{A} \in \real^{d\times d}$, which we can only access by the matrix--vector product operation $\vec{\omega} \mapsto \mat{A}\vec{\omega}$.
Using the matrix--vector product operation, we can compute the matrix--matrix product $\mat{A}\mat{\Phi}$ with a (random) matrix $\mat{\Phi} \in \real^{d\times s}$ with $s$ columns and form the Nystr\"om approximation \cite{GM13,TYUC17b,WS00}
\begin{equation} \label{eq:nystrom}
    \mat{A}\langle \mat{\Phi}\rangle \coloneqq \mat{A}\mat{\Phi} \mleft(\mat{\Phi}^*\mat{A}\mat{\Phi}\mright)^\dagger (\mat{A}\mat{\Phi})^*.
\end{equation}
The result is a psd, rank-$s$ approximation $\mat{A}\langle \mat{\Phi}\rangle$ to the matrix $\mat{A}$.
We will generate $\mat{\Phi}$ by applying $q\ge 0$ steps of subspace iteration to a random test matrix $\mat{\Omega}$
\begin{equation} \label{eq:subspace_iteration}
    \mat{\Phi} = \mat{A}^q \mat{\Omega},
\end{equation}
where $\mat{\Omega}$ is populated with statistically independent standard Gaussian entries.
The quality of the Nystr\"om approximation improves with a higher approximation rank $s$ and number of subspace iteration steps $q$.
Using the forthcoming \cref{alg:nystrom}, we can compute $\mat{A}\langle \mat{\Phi}\rangle$ in the form of a compact eigenvalue decomposition:
\begin{equation} \label{eq:nystrom_eigendecomposition}
    \mat{A}\langle \mat{\Phi}\rangle = \mat{V} \mat{\Lambda}\mat{ V}^*,
\end{equation}
where $\mat{V} \in \real^{d\times s}$ has orthonormal columns and $\mat{\Lambda} \in \real_+^{s\times s}$ is diagonal.
The cost of this procedure is $\order(qs)$ matrix--vector products with $\mat{A}$ and $\order(ds^2)$ additional operations.

To use the Nystr\"om approximation with confidence in applications and to guide the choice of parameters $s$ and $q$, we need to understand the accuracy of the approximation $\mat{A} \approx \mat{A}\langle \mat{\Phi}\rangle$.
This motivates our question:
\actionbox{What is the most efficient way to estimate the error $\norm{\mat{A} - \mat{A}\langle \mat{\Phi}\rangle}_{\rm F}$?}
Inspired by this question, this article proposes the \emph{leave-one-out error estimator}, which provides an estimate of $\norm{\mat{A} - \mat{A}\langle \mat{\Phi}\rangle}_{\rm F}$ using only the information already collected from $\mat{A}$ to form the Nystr\"om approximation.

The leave-one-out estimator is built by recomputing the Nystr\"om approximation using subsamples of the columns of the matrix $\mat{\Omega}$.
We can regard the Nystr\"om approximation as a function of the test matrix $\mat{\Omega}$:
\begin{equation*}
    \mat{\Omega} \mapsto \mat{X} = \mat{X}(\mat{\Omega}) \coloneqq \mat{A}\langle \mat{A}^q \mat{\Omega}\rangle.
\end{equation*}
Let $\mat{\Omega}^{(j)}$ denote $\mat{\Omega}$ without its $j$th column.
Introduce \emph{replicates} $\mat{X}^{(1)},\ldots,\mat{X}^{(s)}$ by recomputing $\mat{X}$ with each column of $\mat{\Omega}$ left out in turn:
\begin{equation*}
    \mat{X}^{(j)} = \mat{X} \big(\mat{\Omega}^{(j)}\big) \quad \text{for $j=1,2,\ldots,s$}.
\end{equation*}
Letting $\vec{\omega}_j$ denote the $j$th column of $\mat{\Omega}$ and $\norm{\cdot}$ denote the Euclidean norm, we define the \emph{leave-one-out error estimator}
\begin{equation*}
    \hat{\Err}^2(\mat{X},\mat{A}) \coloneqq \frac{1}{s} \sum_{i=1}^s \big\|\big(\mat{A} - \mat{X}^{(j)}\big)\vec{\omega}_j\big\|^2.
\end{equation*}
As we show in \cref{thm:loo}, this error estimator is an \emph{unbiased estimator for the mean-square error of the rank-$(s-1)$ Nystr\"om approximation}:
\begin{equation*}
    \expect \hat{\Err}^2(\mat{X},\mat{A}) = \expect \big\|\mat{A} - \mat{X}(\mat{\Omega}^{(s)})\big\|_{\rm F}^2.
\end{equation*}
Once the Nystr\"om approximation has been computed, $\hat{\Err}(\mat{X},\mat{A})$ is cheap to form, requiring at most $\order(ds^2)$ additional operations (and only $\order(s^3)$ operations if $q = 0$).
See \cref{sec:nystrom} for details.
The error estimate $\hat{\Err}(\mat{X},\mat{A})$ requires no information about $\mat{A}$ beyond what is required to form the approximation $\mat{X}$.

A good point of comparison for the leave-one-out error estimator is provided by the Girard--Hutchinson norm estimate \cite[\S4.8]{MT20}
\begin{equation} \label{eq:gh}
    \hat{\Err}^2_{\rm GH}(\mat{X},\mat{A}) = \frac{1}{t} \sum_{i=1}^t \norm{ (\mat{A}-\mat{X}) \vec{\nu}_i }^2 \approx \norm{\mat{A}-\mat{X}}_{\rm F}^2.   
\end{equation}
Here, $\vec{\nu}_1,\ldots,\vec{\nu}_t$ are independent standard Gaussian test vectors.
This estimator requires $t$ matrix--vector products with $\mat{A}$ (beyond those used to form the approximation $\mat{X}$), and the norm estimate becomes more accurate with larger $t$. 
The Girard--Hutchinson norm estimator is equivalent to the Girard--Hutchinson trace estimator \cite[\S4.2]{MT20} applied to $(\mat{A}-\mat{X})^*(\mat{A}-\mat{X})$ and served as our inspiration for the leave-one-out error estimator.
The leave-one-out estimator improves on the Girard--Hutchinson estimator as it requires no additional matrix--vector products with $\mat{A}$ to compute.
In addition, the quality of the leave-one-out estimator automatically improves when the approximation rank $s$ increases, whereas the Girard--Hutchinson estimate only improves by using more matrix--vector products.

\begin{figure}[t]
  \centering
  \includegraphics[width=0.45\textwidth]{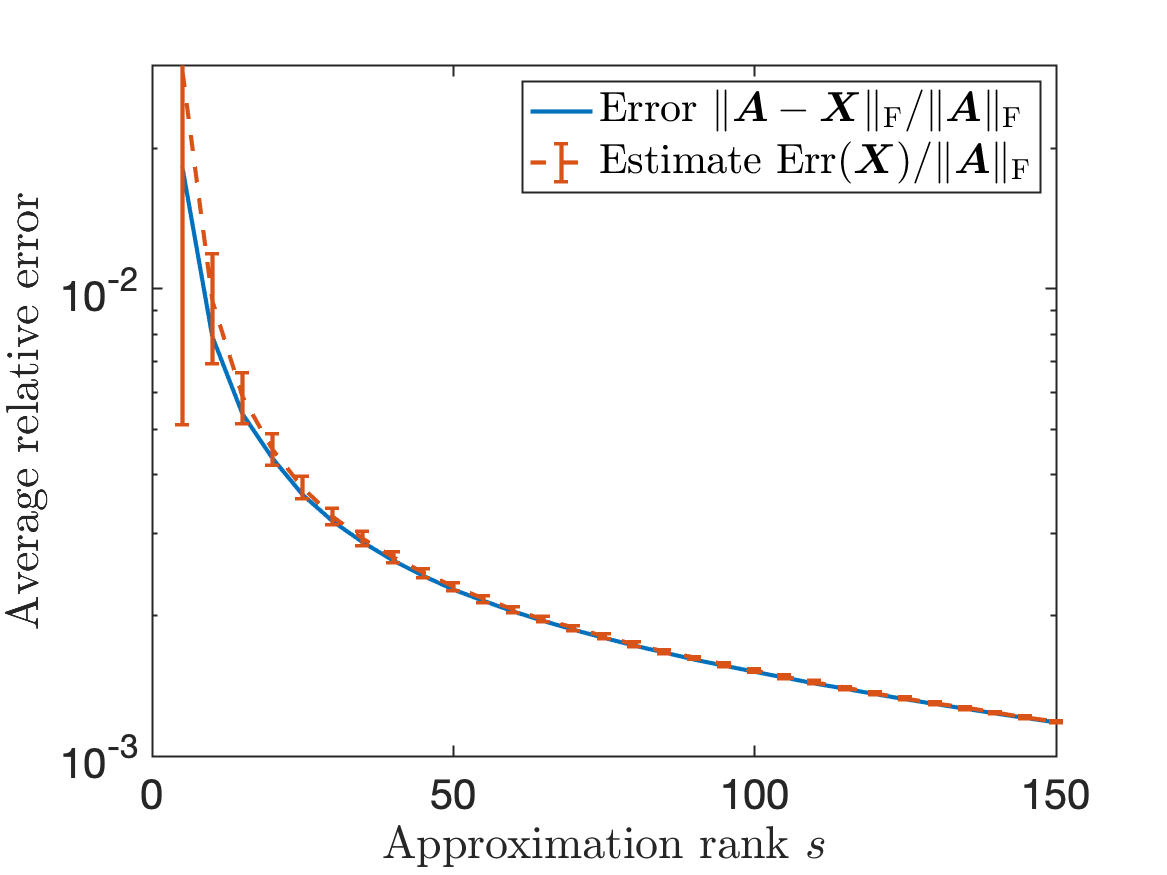}\hfill \includegraphics[width=0.45\textwidth]{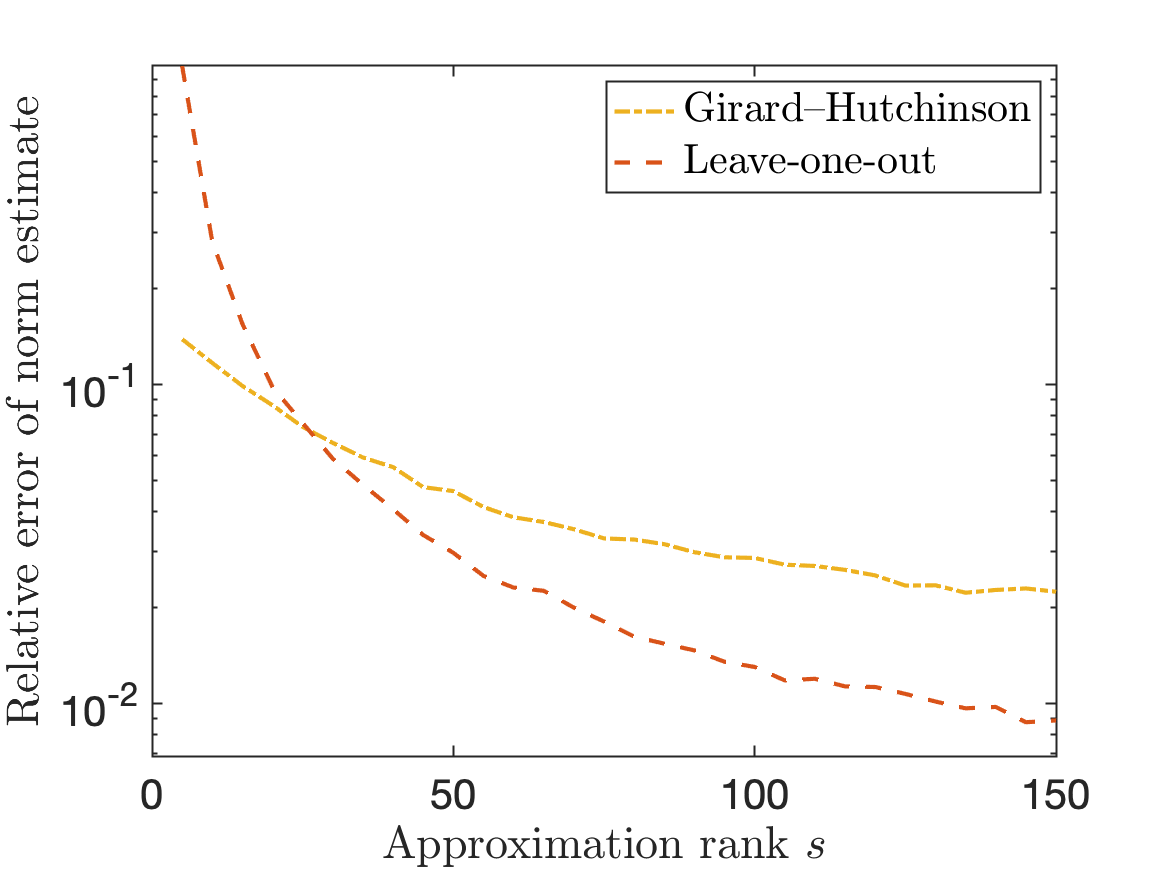}
  
  \mycaption{Leave-one-out error estimator}{Approximation error $\Err(\mat{X},\mat{A}) / \norm{\mat{A}}_{\rm F}$ and normalized error estimate $\hat{\Err}(\mat{X},\mat{A}) / \norm{\mat{A}}_{\rm F}$ (\emph{left}) and relative error for the leave-one-out and Girard--Hutchinson error estimators (\emph{right}) for Nystr\"om approximation $\mat{X}$ to a psd kernel matrix $\mat{A} \in \real^{10^4\times 10^4}$.
  Error and error estimate were computed by an average of $1000$ trials.
  Error bars show one standard deviation.}
  \label{fig:loo}
\end{figure}

\Cref{fig:loo} demonstrates the leave-one-out error estimator.
In this figure, we apply single-pass Nystr\"om approximation ($q=0$) to approximate psd kernel matrix $\mat{A} \in \real^{10^4\times 10^4}$ formed from a random subsample of $10^4$ points from the QM9 dataset \cite{RDRv14,RvBR12} using approximation ranks $5\le s \le 150$.
In the left panel, we plot the mean error
\begin{equation} \label{eq:mean_error}
    \Err(\mat{X},\mat{A}) \coloneqq \expect \norm{\mat{A} - \mat{X}}_{\rm F}
\end{equation}
and leave-one-out error estimate $\hat{\Err}(\mat{X},\mat{A})$, estimated using $1000$ trials.
We see that the estimate $\hat{\Err}(\mat{X},\mat{A})$ closely tracks the true error.
Moreover, the error estimate $\hat{\Err}(\mat{X},\mat{A})$ is fast to compute, with the error estimate taking less than 1\% of the total runtime to form.
The right panel compares the mean relative error
\begin{equation*}
    \text{mean relative error} = \expect \left[ \frac{|\norm{\mat{A}-\mat{X}}_{\rm F} - \mathrm{Est}|}{\norm{\mat{A}-\mat{X}}_{\rm F}}\right], \quad \mathrm{Est} \in \left\{ \hat{\Err}(\mat{X},\mat{A}), \hat{\Err}_{\rm GH}(\mat{X},\mat{A})\right\}
\end{equation*}
for both the leave-one-out and Girard--Hutchinson error estimators.
Following \cite[\S7.9]{TYUC19}, we use $t=10$ matrix--vector products for the Girard--Hutchinson estimator.
For $s\ge 25$, the leave-one-out estimator is more accurate than Girard--Hutchinson and, for all values of $s$, the leave-one-out estimator is cheaper to form than Girard--Hutchinson, requiring just $\order(s^3)$ operations and no additional matrix--vector products.

\subsection{Matrix jackknife motivating example: spectral clustering}
\label{sec:spectral_clustering}

Randomized low-rank approximations can also be used for \emph{spectral computations} (i.e., to approximate eigenvalues, eigenvectors, singular values, etc.).
In this case, the low-rank approximation error $\norm{\mat{A} - \mat{X}}_{\rm F}$ may only provide indirect information about the accuracy of the computation.
For situations such as this, we propose a \emph{matrix jackknife variance estimate} as a diagnostic tool.
In this section, we illustrate the value of this matrix jackknife approach in a spectral clustering application, before introducing the method in generality in \cref{sec:matrix-jackknife}.

\subsubsection{Nystr\"om-accelerated spectral clustering}
Spectral clustering \cite{Von07} is an algorithm that uses eigenvectors to assign data points, say $\vec{c}_1,\ldots,\vec{c}_d \in \real^m$, into groups.
To measure similarity between points, we employ a nonnegative, positive definite kernel function $\kappa : \real^m \times \real^m \to \real_+$.
One popular choice is the square-exponential kernel
\begin{equation} \label{eq:square_exponential}
    \kappa\big(\vec{c},\vec{c}'\big) = \exp \mleft( - \frac{\norm{\vec{c} - \vec{c}'}^2}{2\sigma^2}\mright).
\end{equation}
To cluster the data points into groups, we perform the following steps
\begin{enumerate}
    \item Form the kernel matrix $\mat{K}$ with entries $k_{ij} = \kappa(\vec{c}_i,\vec{c}_j)$.
    \item Form the diagonal matrix $\mat{D} = \diag\big( \sum_{j=1}^d k_{ij} : i=1,\ldots,d \big)$.
    \item Compute the $n_{\rm dim}$ dominant eigenvectors $\mat{U}$ of $\mat{A} \coloneqq \mat{D}^{-1/2}\mat{K}\mat{D}^{-1/2}$.
    \item Set $\mat{W} \coloneqq \mat{D}^{-1/2} \mat{U}$.
    \item Apply a general-purpose clustering algorithm, such as k-means \cite{AV07} with $n_{\rm cen}$ centers, to the rows of $\mat{W}$.
\end{enumerate}
Parameters $n_{\rm dim}$ and $n_{\rm cen}$ set the clustering space dimension and number of clusters.

If one uses direct methods for the eigenvalue problem, the cost of spectral clustering is dominated by the $\order(d^3)$ cost for the eigenvector calculation in step 3.
We can accelerate spectral clustering by using Nystr\"om approximation \cref{eq:nystrom}.
The modification is simple: Use the $n_{\rm dim}$ dominant eigenvectors of the Nystr\"om approximation, accessible from the eigendecomposition \cref{eq:nystrom_eigendecomposition}, in place of the eigenvectors of $\mat{A}$.

\subsubsection{Variance estimation for spectral clustering}

As we refine the approximation by increasing $s$, the approximate eigenvectors $\mat{\hat{U}}$ will converge to the true eigenvectors $\mat{U}$ (provided $\mat{U}$ is unique).
But how do we know when we have taken $s$ large enough?
Our guiding principle is:
\actionbox{In order to trust the answer provided by a randomized algorithm, the output should be insensitive to the randomness used by the algorithm.}
The \emph{variance} of the matrix output $\mat{X}$ of a randomized algorithm, defined as
\begin{equation} \label{eq:variance}
    \Var(\mat{X}) \coloneqq \expect \norm{\mat{X} - \expect \mat{X}}_{\rm F}^2,
\end{equation}
provides a quantitative measurement of the sensitivity of the algorithmic output to randomness used by the algorithm.
In the context of spectral clustering, we can use a variance estimate to guide our choice of the rank $s$.

To understand the sensitivity of Nystr\"om-accelerated spectral clustering to randomness in the algorithm, we need to specify a target matrix $\mat{X}$ for variance estimation.
The input to k-means clustering are the \emph{coordinates}
\begin{equation*}
    \mat{\hat{W}} \coloneqq \mat{D}^{-1/2} \mat{\hat{U}}.
\end{equation*}
To respect the invariance of k-means to scaling and rotation of the coordinates, our target for variance estimation will be
\begin{equation*}
    \mat{X} = \mat{\hat{W}} \mat{\hat{W}}^* / \big\|\mat{\hat{W}}\mat{\hat{W}}^*\big\|_{\rm F}.
\end{equation*}

\subsubsection{Jackknife variance estimation}

Jackknife variance estimation is similar to the leave-one-out error estimator in that we use replicates $\mat{X}^{(1)},\ldots,\mat{X}^{(s)}$ recomputed by successively leaving out columns of the random test matrix $\mat{\Omega}$.
As before, we view the target $\mat{X}$ as a function
\begin{equation*}
    \mat{X} = \mat{X}(\mat{\Omega})
\end{equation*}
of the test matrix $\mat{\Omega}$ defining the Nystr\"om approximation by \cref{eq:nystrom} and \cref{eq:subspace_iteration}.
Form \emph{jackknife replicates} $\mat{X}^{(j)}$ and their average $\mat{X}^{(\cdot)}$ via
\begin{equation*}
    \mat{X}^{(j)} = \mat{X}\mleft( \mat{\Omega}^{(j)} \mright) \quad \text{for $j = 1,\ldots,s$} \quad \text{and} \quad \mat{X}^{(\cdot)} \coloneqq \frac{1}{s} \sum_{j=1}^s \mat{X}^{(j)}
\end{equation*}
where $\mat{\Omega}^{(j)}$ again denotes $\mat{\Omega}$ without its $j$th column.
The jackknife estimate for $\Var(\mat{X})$ is
\begin{equation*}
    \Jack^2(\mat{X}) \coloneqq \sum_{j=1}^s \norm{\mat{X}^{(j)} - \mat{X}^{(\cdot)}}_{\rm F}^2.
\end{equation*}
Guarantees for this estimator are provided in \cref{thm:jackknife}.
\Cref{alg:spectral_clustering} and \cref{list:spectral_clustering} provide pseudocode and a MATLAB implementation of Nystr\"om-accelerated spectral clustering with the jackknife estimate $\Jack(\mat{X})$.
The cost of forming the jackknife estimate $\Jack(\mat{X})$ is $\order(s^2d)$, much faster than the $\order(qsd^2)$ total cost of Nystr\"om-accelerated spectral clustering.

\subsubsection{Numerical example} \label{sec:spectral_clustering_numerics}
\begin{figure}
    \centering

    \begin{subfigure}[b]{0.46\textwidth} \centering
    \includegraphics[width=\textwidth]{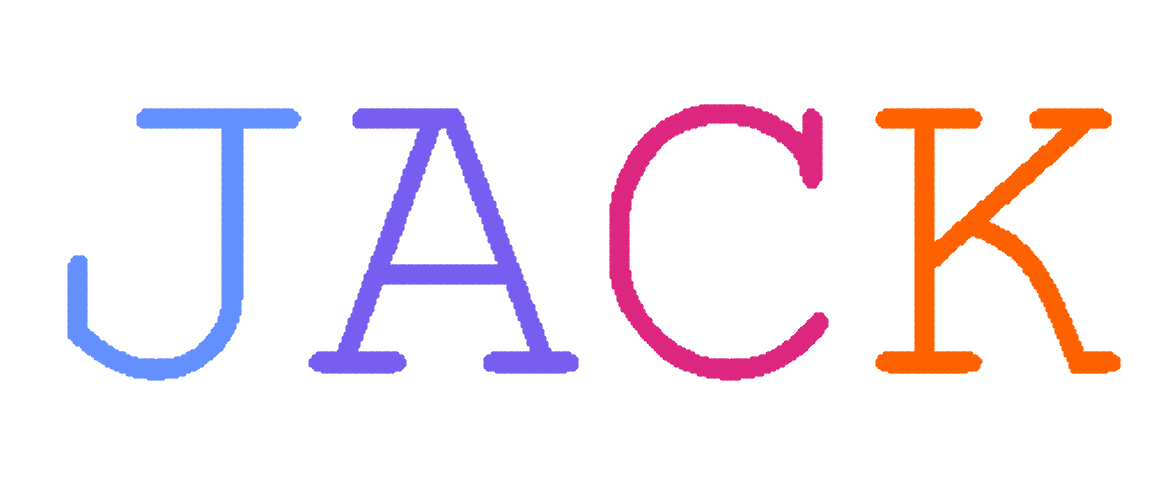}
        \caption*{Correct clustering}
      \end{subfigure}
      \begin{subfigure}[b]{0.46\textwidth} \centering
    \includegraphics[width=\textwidth]{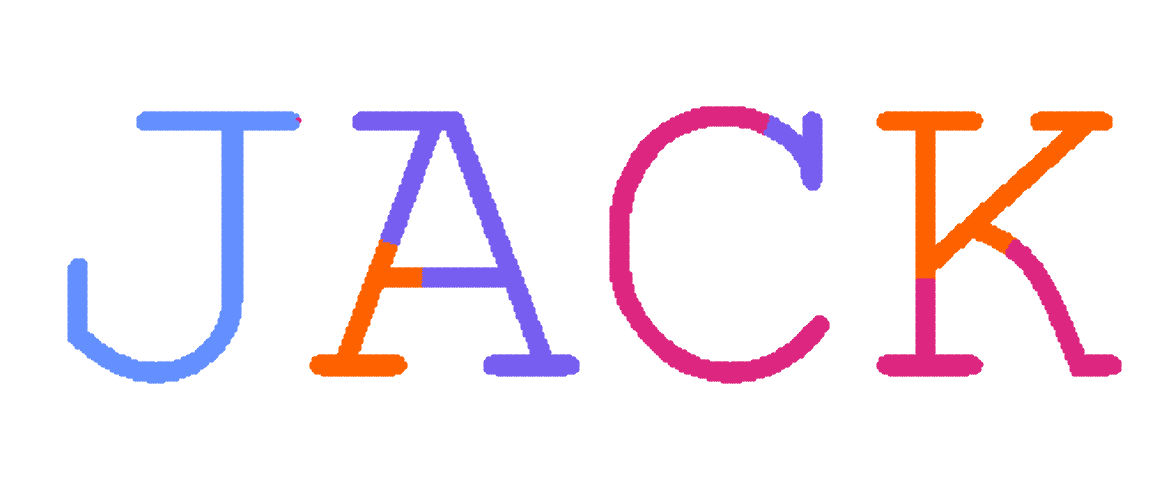}
        \caption*{Incorrect clustering}
      \end{subfigure}
    
    \includegraphics[width=0.7\textwidth]{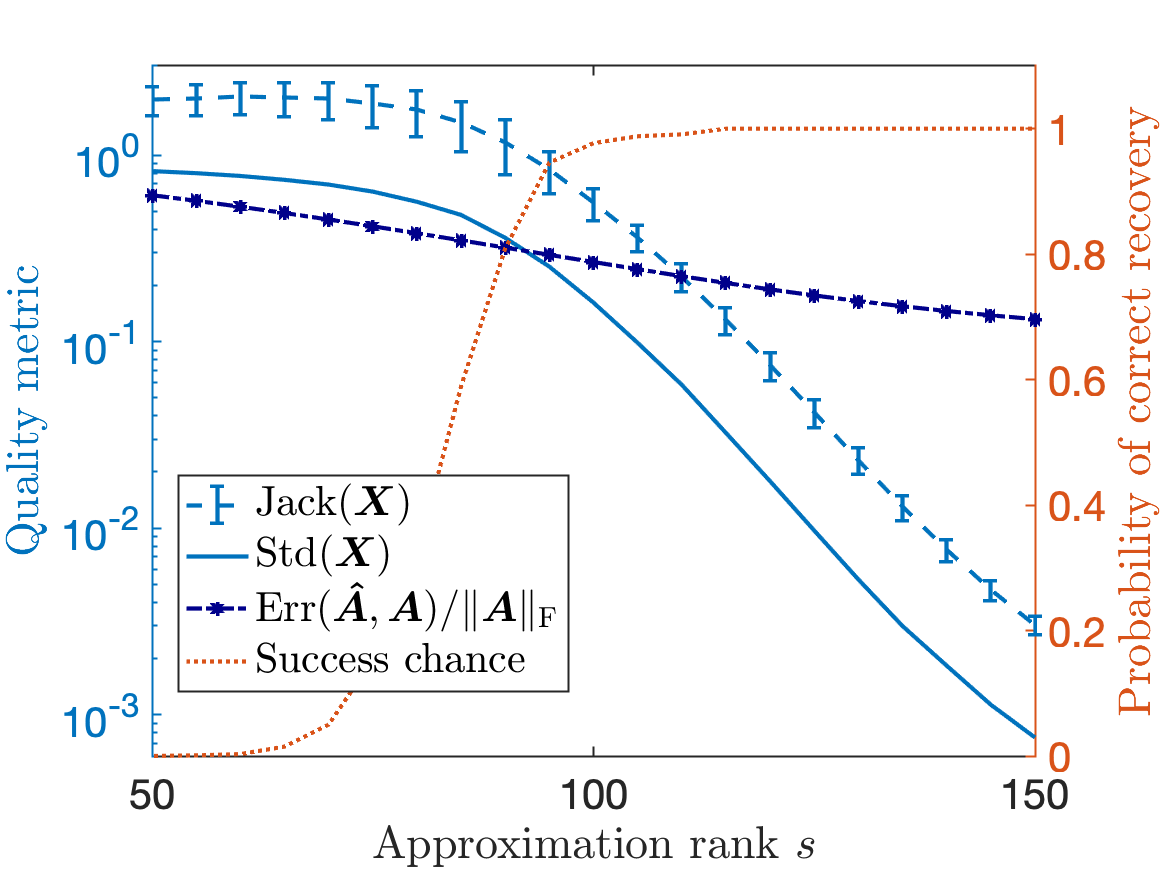}
    \mycaption{Matrix jackknife for spectral clustering}{\emph{Top}: Example of correct (\emph{left}) and incorrect (\emph{right}) clusterings. \emph{Bottom}: Standard deviation $\operatorname{Std}(\mat{X})$ and its jackknife estimate $\Jack(\mat{X})$ (left axis) and success probability of spectral clustering (right axis) versus Nystr\"om approximation rank $50\le s\le 150$.}
    \label{fig:spectral_clustering}
\end{figure}
To demonstrate the effectiveness of jackknife variance estimation for spectral clustering, we use the following experimental setup: Consider the task of separating the four letters \texttt{JACK} from a point cloud $\vec{c}_1,\ldots,\vec{c}_{9426} \in \real^2$.
For spectral clustering, use the square-exponential kernel \cref{eq:square_exponential} and parameters $n_{\rm dim} = n_{\rm cen} = 4$. For Nystr\"om, use $q = 3$ steps of subspace iteration and test a range of approximation ranks $50 \le s \le 150$.
For each value of $s$, we run $1000$ trials and report the natural Monte Carlo estimate of the standard deviation
\begin{equation} \label{eq:std}
    \operatorname{Std}(\mat{X}) = \mleft( \expect \norm{\mat{X} - \expect \mat{X}}_{\rm F}^2 \mright)^{1/2},
\end{equation}
the mean and standard deviation for the \emph{standard deviation} estimate $\Jack(\mat{X})$, the relative error $\Err(\mat{\hat{A}},\mat{A})/\norm{\mat{A}}_{\rm F}$ for the Nystr\"om approximation $\mat{\hat{A}}$, and the empirical success probability for spectral clustering.

\Cref{fig:spectral_clustering} shows the results.
The estimate $\Jack(\mat{X})$ overestimates the standard deviation by a modest amount (a factor of five at most).
The jackknife is not quantitatively sharp, but it is a reliable indicator of whether the variance is high or low.

The jackknife estimate $\Jack(\mat{X})$ can be used to determine whether the Nystr\"om-based approximations to the eigenvectors are accurate enough for clustering task.
At rank $s = 50$, clustering is never performed correctly, and the jackknife estimate $\Jack(\mat{X}) \approx 2$ is high.
As the approximation rank $s$ is increased, clustering begins to succeed with higher and higher probability and the jackknife estimate $\Jack(\mat{X})$ decreases, indicating reduced variability.

The success of the jackknife estimate $\Jack(\mat{X})$ should be contrasted with the failure of the Nystr\"om error \smash{$\Err(\mat{\hat{A}},\mat{A})$} as a useful diagnostic for spectral clustering correctness.
The Nystr\"om error \smash{$\Err(\mat{\hat{A}},\mat{A})$} decreases at a steady rate as $s$ is increased---unlike the jackknife estimate, the plot of \smash{$\Err(\mat{\hat{A}},\mat{A})$} does not show an inflection point indicating the transition between clustering failure and success.

On the basis of these experiments, we propose two possible uses for the jackknife estimate in a spectral clustering workflow:
\begin{itemize}
    \item \textbf{User warning.} If the jackknife variance estimate is high, provide a warning to the user.
    This gives the user to determine and fix the problem for themselves by changing the Nystr\"om parameters $s,q$ or the spectral clustering parameters $n_{\rm cen},n_{\rm dim}$.
    \item \textbf{Adaptive stopping.} Choose parameters $s$ or $q$ for the Nystr\"om approximation adaptively at runtime by increasing these parameters until $\Jack(\mat{X})$ falls below a tolerance (e.g., $0.1$).
\end{itemize}
These uses demonstrate the potential for jackknife variance estimation to be helpful in incorporating randomized low-rank approximation into general-purpose software.

\subsubsection{Benefits of matrix jackknife variance estimation}

The spectral clustering example demonstrates a number of virtues for matrix jackknife variance estimation:
\begin{itemize}
    \item \textbf{Flexibility.} Matrix jackknife variance estimation can be applied to a very general target function $\mat{X}(\mat{\Omega})$ depending on a random test matrix $\mat{\Omega}$.
    This allows the jackknife to be applied to a wide array of randomized low-rank approximation algorithms and allows the user to design the variance estimation target for their application.
    \item \textbf{Efficiency.} By using optimized algorithms (\cref{sec:computations,sec:efficient_spectral_clustering}), computation of the jackknife variance estimate can be very fast.
    For instance, for the clustering problem, the $\order(s^2d)$ cost of the jackknife estimate is dwarfed by the $\order(sd^2)$ cost of the clustering procedure.
    For $s = 150$, computing the jackknife variance estimate amounts to less than 3\% of the total runtime.
\end{itemize}

\subsection{Outline}
\label{sec:outline}

Having introduced our two diagnostics, we present each in more generality; \cref{sec:loo} discusses leave-one-out error estimation and \cref{sec:matrix-jackknife} discusses the matrix jackknife.
\Cref{sec:computations} discusses efficient computations for both of these diagnostics applied to the randomized SVD and Nystr\"om approximation.
\Cref{sec:numer-exper} contains numerical experiments, and \cref{sec:high-schatt-norms} extends the the matrix jackknife to higher Schatten norms ($p > 2$).

\subsection{Notation}
\label{sec:notation}

We work over the field $\field=\real$ or $\field = \complex$.
Notations ${}^*$, ${}^\dagger$, and $\norm{\cdot}_{\rm F}$ denote the conjugate transpose, Moore--Penrose pseudoinverse, and Frobenius norm.
The expectation of a random variable $X$ is denoted $\expect X$, and its variance is defined as $\Var(X) \coloneqq \expect |X - \expect X|^2$.
We adopt the convention that nonlinear operators bind before the expectation; for example, $\expect X^2 \coloneqq \expect (X^2)$.
The variance of a random matrix is given by \cref{eq:variance}.

\section{Leave-one-out error estimation for low-rank approximation} \label{sec:loo}

In this section, we present the leave-one-out error estimation technique introduced in \cref{sec:loo_intro} for more general randomized matrix approximations.

\subsection{The estimator}

Let $\mat{A} \in \field^{d_1\times d_2}$ be a matrix we seek to approximate by a randomized approximation $\mat{X}$.
We are interested in a general class of algorithms which collect information about the matrix $\mat{A}$ by matrix--vector products
\begin{equation*}
    \mat{A}\vec{\omega}_1,\ldots,\mat{A}\vec{\omega}_s
\end{equation*}
with random test vectors $\vec{\omega}_1,\ldots,\vec{\omega}_s$.
Many algorithms are defined for an arbitrary number of test vectors $s$, allowing us to construct error estimates by leaving out a test vector, resulting in an approximation $\mat{X}_{s-1}$ defined using only $s-1$ vectors. 
This motivates the following abstract setup:
\begin{itemize}
\item Let $\vec{\omega}_1,\ldots,\vec{\omega}_s$ be independent and identically distributed (iid) random vectors in $\field^{d_2}$ that are \emph{isotropic}: $\expect \big[ \vec{\omega}_j^{\vphantom{*}} \vec{\omega}_j^* \big] = \Id$.
\item Let $\mat{X}$ denote one of two matrix estimators defined for $s$ or $s-1$ inputs:
  \begin{equation*}
    \mat{X} : (\field^{d_2})^s \to \field^{d_1\times d_2} \quad \textrm{or} \quad \mat{X} : (\field^{d_2})^{s-1} \to \field^{d_1\times d_2}.
  \end{equation*}
\item Define estimates $\mat{X}_s \coloneqq \mat{X}(\vec{\omega}_1,\ldots,\vec{\omega}_s)$ and $\mat{X}_{s-1} \coloneqq \mat{X}(\vec{\omega}_1,\ldots,\vec{\omega}_{s-1})$.
\end{itemize}
Examples of estimators which fit this description include randomized Nystr\"om approximation, the randomized SVD \cite{HMT11}, and randomized block Krylov iteration \cite{MM15,TW23}.

We seek to approximate the mean-square error
\begin{equation*}
    \operatorname{MSE}(\mat{X}_{s-1},\mat{A}) = \expect \norm{\mat{A} - \mat{X}_{s-1}}_{\rm F}^2
\end{equation*}
of the $(s-1)$-sample approximation $\mat{X}_{s-1}$ as a proxy for the mean-square error $\operatorname{MSE}(\mat{X}_s)$ of the $s$-sample approximation $\mat{X}_s$.
Define the \emph{leave-one-out mean-square error estimate}
\begin{equation} \label{eq:loo_general}
    \hat{\Err}^2(\mat{X}_{s-1},\mat{A}) \coloneqq \frac{1}{s} \sum_{j=1}^s\norm{ \big(\mat{A} - \mat{X}^{(j)}\big)\vec{\omega}_j}^2,
\end{equation}
where the replicates $\mat{X}^{(j)}$ are 
\begin{equation*}
    \mat{X}^{(j)} = \mat{X}(\vec{\omega}_1,\ldots,\vec{\omega}_{j-1},\vec{\omega}_{j+1},\ldots,\vec{\omega}_s) \quad \text{for $j = 1,\ldots,s$}.
\end{equation*}
This estimator is an unbiased estimator for $\operatorname{MSE}(\mat{X}_{s-1},\mat{A})$.
\begin{theorem}[Leave-one-out error estimator] \label{thm:loo}
    With the prevailing notation,
    \begin{equation*}
        \operatorname{MSE}(\mat{X}_{s-1},\mat{A}) = \expect\, \hat{\Err}^2(\mat{X}_{s-1},\mat{A}).
    \end{equation*}
\end{theorem}

\begin{proof}
    For each $j$, $\mat{X}^{(j)}$ and $\omega_j$ are independent. 
    Consequently, letting $\expect_j$ denote an expectation over the randomness in $\vec{\omega}_j$ alone, we compute
    \begin{align*}
        \expect_j \big\|(\mat{A}-\mat{X}^{(j)})\vec{\omega}_j\big\|^2 &= \expect_j \big[ \vec{\omega}_j^*(\mat{A}-\mat{X}^{(j)})^*(\mat{A}-\mat{X}^{(j)})\vec{\omega}_j^{\vphantom{*}}\big] \\
        &= \expect_j \tr \big[(\mat{A}-\mat{X}^{(j)})^*(\mat{A}-\mat{X}^{(j)})\vec{\omega}_j^{\vphantom{*}}\vec{\omega}_j^*\big] \\
        &= \tr \big((\mat{A}-\mat{X}^{(j)})^*(\mat{A}-\mat{X}^{(j)})\expect\left[\vec{\omega}_j^{\vphantom{*}}\vec{\omega}_j^*\right]\big) \\
        &= \tr \big((\mat{A}-\mat{X}^{(j)})^*(\mat{A}-\mat{X}^{(j)})\big) = \big\|\mat{A} - \mat{X}^{(j)}\big\|_{\rm F}^2.
    \end{align*}
    The first line is an identity for the Frobenius norm, the second line is the cyclic property of the trace, the third line is the independence of $\mat{X}^{(j)}$ and $\vec{\omega}_j$, and the fourth line is the isotropic property of $\vec{\omega}_j$.
    Thus, by the tower property of conditional expectation, we conclude
    \begin{align*}
        \expect\, \hat{\Err}^2(\mat{X}_{s-1},\mat{A}) &= \frac{1}{s} \sum_{j=1}^s\expect\, \big[\expect_j \, \big\| \big(\mat{A} - \mat{X}^{(j)}\big)\vec{\omega}_j\big\|^2\big] \\&= \frac{1}{s} \sum_{j=1}^s\expect\, \big\| \big(\mat{A} - \mat{X}^{(j)}\big)\vec{\omega}_j\big\|^2_{\rm F} = \operatorname{MSE}(\mat{X}_{s-1},\mat{A}).
    \end{align*}
    This confirms the theorem.
\end{proof}

Numerical evidence for the quality of this error estimate is provided in \cref{fig:loo,fig:rsvd_loo}.
With efficient algorithms (\cref{sec:computations}), the leave-one-out error estimator is rapid to compute for the randomized SVD and Nystr\"om approximation.

\subsection{Alternatives}

Two alternatives to the leave-one-out error estimator are worth mentioning.
First, in many situations, it may be possible and computationally cheap to simply compute the error $\norm{\mat{A} - \mat{X}}_{\rm F}$ directly.
For instance, if $\mat{X}$ is the approximation produced by the randomized SVD \cite{HMT11}, then
\begin{equation*}
    \norm{\mat{A} - \mat{X}}_{\rm F}^2 = \norm{\mat{A}}_{\rm F}^2 - \norm{\mat{X}}_{\rm F}^2,
\end{equation*}
which facilitates fast computation of the error if $\mat{A}$ is a dense or sparse matrix stored in memory.
The leave-one-out error estimator should only be used if direct computation of the error is not possible or too expensive.
This is the case, for example, in the black-box setting where one has access to $\mat{A}$ only through the matrix--vector product $\vec{\omega} \mapsto \mat{A}\vec{\omega}$ and adjoint--vector product $\vec{\omega} \mapsto \mat{A}^*\vec{\omega}$ operations.

A second alternative is the Girard--Hutchinson norm estimator \cref{eq:gh}, discussed in \cref{sec:loo_intro}.
The leave-one-out estimator improves on the Girard--Hutchinson estimator as the leave-one-out estimate does not require any additional matrix--vector products and automatically improves in quality as the number of vectors $s$ is increased.

\section{Matrix jackknife variance estimation}
\label{sec:matrix-jackknife}

This section outlines our proposal for \emph{matrix jackknife variance estimation} for more general randomized matrix algorithms.
\Cref{sec:efron-stein-steele} reviews jackknife variance estimation for scalar quantities.
We introduce and analyze the matrix jackknife variance estimator in \cref{sec:matr-jackkn-estim}.
\Cref{sec:uses-jackkn-vari,sec:matr-jackkn-vers,sec:bootstrap} discuss potential applications of the matrix jackknife and complementary topics.

\subsection{Tukey's jackknife variance estimator and the Efron--Stein--Steele inequality}
\label{sec:efron-stein-steele}

To motivate our matrix jackknife proposal, we begin by presenting the jackknife variance estimator \cite{Tuk58} for scalar estimators due to Tukey in \cref{sec:tukeys-jackkn-vari}.
In \cref{sec:efron-stein-steele-1}, we discuss the Efron--Stein--Steele inequality used in its analysis.

\subsubsection{Tukey's jackknife variance estimator}
\label{sec:tukeys-jackkn-vari}

Consider the problem of estimating the variance of a statistical estimator computed from $s$ random samples.
We assume it makes sense to evaluate the estimator with fewer than $s$ samples, as is the case for many classical estimators like the sample mean and variance.
This motivates the following setup:
\begin{itemize}
\item Let $\omega_1,\ldots,\omega_s$ be independent and identically distributed random elements taking values in a measurable space $\Omega$.
\item Let $f$ denote either one of two estimators, defined for $s$ or $s-1$ arguments:
  \begin{equation*}
    f : \Omega^s \to \field \quad \textrm{or} \quad f : \Omega^{s-1} \to \field.
  \end{equation*}
\item Assume that $f$ is invariant to a reordering of its inputs:
  \begin{equation*}
    f(\omega_1,\ldots,\omega_s) = f(\omega_{\pi(1)},\ldots,\omega_{\pi(s)}) \quad \textrm{for any permutation } \pi.
  \end{equation*}
\item Define estimates $E_{s-1} \coloneqq f(\omega_1,\ldots,\omega_{s-1})$ and $E_s \coloneqq f(\omega_1,\ldots,\omega_s)$.
\end{itemize}
We think of $E_s$ as a statistic computed from a collection of samples $\omega_1,\ldots,\omega_s$.
We can also evaluate the statistic with only $s-1$ samples, resulting in $E_{s-1}$.

Tukey's jackknife variance estimator provides an estimate for $\Var(E_{s-1})$, which serves as a proxy for the variance of the $s$-sample estimator $E_s$.
Define jackknife replicates and mean
\begin{equation*} \label{eq:scalar_replicate}
  E^{(j)} \coloneqq f(\omega_1,\ldots,\omega_{j-1},\omega_{j+1},\ldots,\omega_s) \quad \mbox{for each } j = 1,2,\ldots,s; \quad E^{(\cdot)} \coloneqq \frac{1}{s} \sum_{j=1}^s E^{(j)}.
\end{equation*}
The quantities $E^{(1)},\ldots,E^{(s)}$ represent the statistic recomputed with each of the samples $\omega_1,\ldots,\omega_s$ left out in turn.

Tukey's estimator for $\Var(E_{s-1})$ is given by
\begin{equation} \label{eq:tukey}
  \widehat{\Var}(E_{s-1}) \coloneqq \sum_{j=1}^{s} \mleft| E^{(j)} - E^{(\cdot)} \mright|^2.
\end{equation}
Observe that Tukey's estimator \cref{eq:tukey} is the sample variance of the jackknife replicates $E^{(j)}$ up to a normalizing constant.
The form of Tukey's estimator suggests that the distribution of the jackknife replicates somehow approximates the distribution of the estimator.
This intuition can be formalized using the Efron--Stein--Steele inequality.

\subsubsection{Efron--Stein--Steele inequality}
\label{sec:efron-stein-steele-1}

To analyze Tukey's estimator, we rely on an inequality of Efron and Stein \cite{ES81}, which was improved by Steele \cite{Ste86}:
\begin{fact}[Efron--Stein--Steele inequality] \label{fact:efron_stein_steele}
  Let $\omega_1,\ldots,\omega_s \in \Omega$ be independent elements in a measurable space $\Omega$, and let $f : \Omega^s \to \field$ be measurable.
  Let $(\omega_j' : j = 1,\dots, s)$ be an independent copy of $(\omega_j : j = 1, \dots, s)$.
  Then
  \begin{equation} \label{eq:efron_stein_steele}
    \Var\big(f\big(\omega_1^{\vphantom{'}},\ldots,\omega_s^{\vphantom{'}}\big)\big) \le \frac{1}{2} \sum_{i=1}^s \expect \Big| f\big(\omega_1^{\vphantom{'}},\ldots,\omega_s^{\vphantom{'}}\big) - f\big( \omega_1^{\vphantom{'}},\ldots,\omega_{j-1}^{\vphantom{'}},\omega_j',\omega_{j+1}^{\vphantom{'}},\ldots,\omega_s^{\vphantom{'}} \big) \Big|^2. 
  \end{equation}
\end{fact}
The complex-valued version of the inequality presented here follows from the more standard version for real values by treating the real and imaginary parts separately.

In the setting of Tukey's estimator \cref{eq:tukey}, the samples $\omega_1,\ldots,\omega_{s-1}$ are identically distributed and the function $f$ depends symmetrically on its arguments.
Therefore, the last sample $\omega_s$ can be used to fill the role of each $\omega_j'$ in the right-hand side of \cref{eq:efron_stein_steele}.
As a consequence, the Efron--Stein--Steele inequality shows that
\begin{equation} \label{eq:tukey_overestimate}
\begin{aligned}
  \Var(E_{s-1}) &\le \frac{1}{2} \sum_{i=1}^{s-1} \expect\,\mleft( E^{(j)} - E^{(s)} \mright)^2 = \frac{1}{2s}\sum_{i,j=1}^s \expect\,\mleft( E^{(i)} - E^{(j)} \mright)^2 \\ &= \sum_{j=1}^s \expect\,\mleft( E^{(j)} - E^{(\cdot)} \mright)^2 = \expect\, \widehat{\Var}(E_{s-1}).
  \end{aligned}
\end{equation}
To move from the first line to the second, we expand the square, use the definition of the mean $E^{(\cdot)} = s^{-1} \sum_{j=1}^s E^{(j)}$, and regroup terms.
This computations shows that Tukey's variance estimator \cref{eq:tukey} \emph{overestimates} the true variance on average.

\subsection{The matrix jackknife estimator of variance}
\label{sec:matr-jackkn-estim}

Suppose we are interested in the variance of the output $\mat{X}\in \field^{d_1\times d_2}$ to a randomized matrix algorithm.
Similar to the scalar setting, we assume that $\mat{X}$ is a function of independent samples $\omega_1,\ldots,\omega_s$ and that it makes sense to evaluate $\mat{X}$ with fewer than $s$ samples.
The formal setup is as follows:
\begin{itemize}
\item Let $\omega_1,\ldots,\omega_s$ be independent and identically distributed random elements in a measurable space $\Omega$.
\item Let $\mat{X}$ denote one of two matrix estimators defined for $s$ or $s-1$ inputs:
  \begin{equation*}
    \mat{X} : \Omega^s \to \field^{d_1\times d_2} \quad \textrm{or} \quad \mat{X} : \Omega^{s-1} \to \field^{d_1\times d_2}.
  \end{equation*}
\item Assume $\mat{X}$ is invariant to reordering of its inputs: $$\mat{X}(\omega_1,\ldots,\omega_s) = \mat{X}(\omega_{\pi(1)},\ldots,\omega_{\pi(s)})\quad \text{for any permutation $\pi$.}$$
\item Define estimates $\mat{X}_s \coloneqq \mat{X}(\omega_1,\ldots,\omega_s)$ and $\mat{X}_{s-1} \coloneqq \mat{X}(\omega_1,\ldots,\omega_{s-1})$.
\end{itemize}
For a randomized low-rank approximation algorithm, the samples $\omega_1,\ldots,\omega_s$ might represent the columns of a test matrix $\mat{\Omega}$, as they did in \cref{sec:spectral_clustering}.

We are interested in estimating the variance of $\mat{X}_{s-1}$ as a proxy for the variance of $\mat{X}_s$.
We expect that adding additional samples will refine the approximation and thus reduce its variance.
Define jackknife replicates $\mat{X}^{(j)}$ and their average $\mat{X}^{(\cdot)}$
\begin{equation*} 
  \mat{X}^{(j)} = \mat{X}(\omega_1,\ldots,\omega_{j-1},\omega_{j+1},\ldots,\omega_s) \quad \mbox{for each } j = 1,\ldots,s; \quad \mat{X}^{(\cdot)} \coloneqq \frac{1}{s} \sum_{j=1}^s \mat{X}^{(j)}.
\end{equation*}
We propose the matrix jackknife estimate
\begin{equation} \label{eq:jackknife_approximation}
  \Jack^2\big(\mat{X}_{s-1}\big) \coloneqq \sum_{j=1}^{s} \norm{\mat{X}^{(j)} - \mat{X}^{(\cdot)}}_{\rm F}^2
\end{equation}
for the variance $\Var(\mat{X}_{s-1})$.
The estimator $\Jack(\mat{X}_{s-1})$ can be efficiently computed for several randomized low-rank approximation, as we shall demonstrate in \cref{sec:computations}.
Similar to the classic jackknife variance estimator, we can use the Efron--Stein--Steele inequality to show that this variance estimate is an overestimate on average.

\begin{theorem} [Matrix jackknife] \label{thm:jackknife}
  With the prevailing notation,
  \begin{equation} \label{eq:jackknife_bound}
    \Var(\mat{X}_{s-1}) \le \expect \, \Jack^2\big(\mat{X}_{s-1}\big).
  \end{equation}
\end{theorem}

\begin{proof}
  Fix a pair of indices $1\le m\le d_1$ and $1\le n\le d_2$.
  Applying \cref{eq:tukey_overestimate} to the $(m,n)$-matrix entry $\mleft(\mat{X}_{s-1}\mright)_{mn}$, we observe 
  \begin{equation*}
    \expect \mleft| \mleft(\mat{X}_{s-1}\mright)_{mn} - \expect \mleft(\mat{X}_{s-1}\mright)_{mn} \mright|^2 \le \expect \sum_{j=1}^s \mleft| \mat{X}^{(j)}_{mn} - \mat{X}^{(\cdot)}_{mn} \mright|^2
  \end{equation*}
  Summing this equation over all $1\le m\le d_1$ and $1\le n\le d_2$ yields the stated result.
\end{proof}

When the jackknife variance estimate is small, \cref{thm:jackknife} shows the variance of the approximation is also small.
Empirical evidence (\cref{sec:numer-exper}) suggests that $\Var(\mat{X}_{s-1})$ and $\expect \, \Jack^2(\mat{X}_{s-1})$ tend to be within an order of magnitude for the algorithms we considered.

It is natural to ask whether we can develop and theoretically jackknife estimates of the bias of randomized matrix algorithms to complement our variance estimate, perhaps using the natural analog of Quenouille's scalar jackknife bias estimate \cite{Que49}.
This is an interesting question for future work.
As we have demonstrated in \cref{sec:spectral_clustering} and will further demonstrate in \cref{sec:numer-exper}, our variance estimate already provides useful and actionable information for randomized matrix algorithms.

\subsection{Uses for matrix jackknife variance estimator}
\label{sec:uses-jackkn-vari}

Variance is a useful diagnostic for randomized matrix approximations.
When the variance is large, it suggests that one of two issues has arisen:
\begin{enumerate}
    \item More samples are needed to refine the approximation.
    \item The underlying approximation problem is badly conditioned
\end{enumerate}
In either case, the jackknife variance estimate can provide evidence that the computed output should not be trusted.

We anticipate the primary use case for matrix jackknife variance estimation will be for computations using eigenvectors or singular vectors computed by randomized low-rank approximation algorithms such as the randomized SVD and Nystr\"om approximation.
Spectral computations with randomized algorithms currently lack effective posterior estimates, making jackknife variance estimation one of the only available tools to assess the quality of the outputs of such computations at runtime.
In the context of spectral computations, matrix jackknife variance estimation can be used to adaptively determine the approximation rank $s$ needed to achieve outputs of sufficiently high quality.
This was demonstrated in \cref{sec:spectral_clustering}, where we used jackknife variance estimation to determine how large to pick $s$ in a spectral clustering context.
As we will later demonstrate in \cref{sec:numer-exper}, we can also use the jackknife variance estimation to detect ill-disposed eigenvectors and to get \emph{coordinate-wise} variance estimates for singular vector computations.

\subsection{Matrix jackknife versus scalar jackknife}
\label{sec:matr-jackkn-vers}

Sometimes, we are only interested in scalar outputs of a randomized matrix computation, such as eigenvalues or singular values, entries of eigenvectors or singular vectors, or the trace.
In these cases, it might be more efficient to directly apply Tukey's variance estimator \cref{eq:tukey} to assess the variance of these scalars.
The matrix jackknife may still be a useful tool because it gives \emph{simultaneous} variance estimates over many scalar quantities.
As examples, the matrix jackknife estimates the maximum variance over all linear functionals:
\begin{equation*}
  \expect \max_{\norm{\mat{C}}_{\rm F} \le 1} \mleft| \tr(\mat{C}\mat{X}_{s-1}) - \tr(\mat{C} \, \expect\mat{X}_{s-1}) \mright|^2 = \expect \norm{\mat{X}_{s-1} - \expect \mat{X}_{s-1}}_{\rm F}^2 \le \expect \Jack^2(\mat{X}_{s-1}),
\end{equation*}
The matrix jackknife gives the following variance estimate for the singular values:
\begin{equation*}
  \expect \sum_{j=1}^{\min(d_1,d_2)} \mleft| \sigma_j(\mat{X}_{s-1}) - \sigma_j(\expect \mat{X}_{s-1}) \mright|^2 = \expect \norm{\mat{X}_{s-1} - \expect \mat{X}_{s-1}}_{\rm F}^2 \le \expect \Jack^2(\mat{X}_{s-1}).
\end{equation*}
Thus, the matrix jackknife is appealing even when one is interested in multiple scalar-valued functions of the matrix approximation $\mat{X}_{s-1}$.
In addition, efficient algorithms for matrix jackknife variance  estimation, as detailed in \cref{sec:computations}, are useful for scalar jackknife variance estimation of functionals of a randomized matrix approximation.

\subsection{Related work: bootstrap for randomized matrix algorithms}
\label{sec:bootstrap}

Bootstrap resampling \cite[\S5]{Efr82}, a close relative of the jackknife, has seen several applications to matrix computations.
An early use case was to provide confidence intervals for eigenvalues and eigenvectors of sample covariance matrices \cite[\S7.2]{TE93}.

A more recent line of work, led by Lopes and collaborators, applies bootstrap resampling to randomized matrix algorithms \cite{LEM20,LEM23,LWM19,YEL23}. 
The work closest to ours \cite{LEM20} uses the bootstrap to provide asymptotically sharp estimates of the error quantiles with regards to general error metrics for each singular value and singular vector computed by a \emph{Monte Carlo-type} sketched SVD algorithm.  In this special case, the bootstrap provides more fine-grained information than 
the matrix jackknife.  Unfortunately, this sketched SVD is a poor computational method because its error decays at the Monte Carlo rate.

The main benefit of our matrix jackknife approach is that it is effective for very general matrix algorithms, such as the randomized SVD and Nystr\"om approximation.
These algorithms are used far more widely in practice because they achieve spectral accuracy, producing errors comparable with the best low-rank approximation~\cite{HMT11}.
As we demonstrate in \cref{sec:bootstrap-failure}, a straightforward application of the bootstrap variance of Efron \cite[\S5.1]{Efr82} to the randomized SVD can produce standard deviation estimates which are incorrect by over four orders of magnitude.

\section{Case studies in low-rank approximation}
\label{sec:computations}

In this section, we develop \emph{efficient} computational procedures to compute the jackknife variance estimate and leave-one-out error estimate for two randomized low-rank approximations, randomized Nystr\"om approximation (\cref{sec:nystrom}) and the randomized SVD (\cref{sec:rsvd}). 

\subsection{Nystr\"om approximation} \label{sec:nystrom}

Given a test matrix $\mat{\Omega} \in \field^{d\times s}$, consider again the Nystr\"om approximation \cref{eq:nystrom} with  $q$ steps of subspace iteration \cref{eq:subspace_iteration} applied to a psd matrix $\mat{A} \in \field^{d\times d}$:
\begin{equation*}
    \mat{X} = \mat{X}(\mat{\Omega}) \coloneqq \mat{A}\langle \mat{\Phi}\rangle = \mat{A}\mat{\Phi} (\mat{\Phi}^*\mat{A}\mat{\Phi})^\dagger (\mat{A}\mat{\Phi})^* \quad \text{where} \quad \mat{\Phi} = \mat{A}^q \mat{\Phi}.
\end{equation*}
The Nystr\"om approximation $\mat{X}$ is the best psd approximation to $\mat{A}$ spanned by $\mat{A}\mat{\Phi}$ with a psd residual.
We focus on the case where $\mat{\Omega}$ is populated with iid, isotropic columns, such as when $\mat{\Omega}$ is a standard Gaussian matrix.

We work with the Nystr\"om approximation in eigenvalue decomposition form:
\begin{equation*}
    \mat{X} = \mat{V}\mat{\Lambda}\mat{V}^*
\end{equation*}
where $\mat{V} \in\field^{d\times s}$ has orthonormal columns and $\mat{\Lambda} \in \real^{s\times s}_+$ is diagonal.
To facilitate efficient computations of our diagnostics, we can compute the Nystr\"om approximation in eigendecomposition form as follows:
\begin{enumerate}
\item Draw a test matrix $\mat{\Omega} \in\field^{n\times s}$ with iid isotropic columns.
\item Apply subspace iteration $\mat{\Phi} = \mat{A}^q \mat{\Omega}$.
\item Compute the product $\mat{Y} = \mat{A}\mat{\Omega}$.
\item Orthonormalize $\mat{Q} = \operatorname{orth}(\mat{Y})$ using economy \QR factorization $\mat{Y} = \mat{Q}\mat{R}$.
\item Compute $\mat{H} = \mat{\Omega}^*\mat{Y}$ and Cholesky factorize $\mat{H} = \mat{C}^*\mat{C}$.
\item Obtain a singular value decomposition $\mat{R}\mat{C}^{-1} = \mat{U}\mat{\Sigma}\mat{Z}^*$.
\item Set $\mat{\Lambda} \coloneqq \mat{\Sigma}^2$ and $\mat{V} \coloneqq \mat{Q}\mat{U}$.
\end{enumerate}
\cref{alg:nystrom} provides an implementation with $q = 0$ with tricks to improve its numerical stability adapted from \cite{LLS+17,TYUC17b}.
For $q > 0$, it may be necessary to introduce additional orthogonalization steps for reasons of numerical stability \cite[Alg.~5.2]{Saa92}.

Treating the Nystr\"om approximation $\mat{X}$ as a symmetric function of the iid, isotropic columns $\vec{\omega}_1,\ldots,\vec{\omega}_s$ of the test matrix $\mat{\Omega}$,
\begin{equation*}
    \mat{X} = \mat{X}(\vec{\omega}_1,\ldots,\vec{\omega}_s),
\end{equation*}
we can apply both the leave-one-out error estimator and jackknife variance estimation to $\mat{X}$.
Define replicates
\begin{equation*}
    \mat{X}^{(j)} = \mat{X}(\vec{\omega}_1,\ldots,\vec{\omega}_{j-1},\vec{\omega}_{j+1},\ldots,\vec{\omega}_s).
\end{equation*}
To compute the replicates efficiently, we use the update formula \cite[eq.~(2.4)]{ETW24}
\begin{subequations} \label{eq:nys_update}
\begin{equation} 
    \mat{X}^{(j)} = \mat{V} \mleft( \mat{\Lambda} - \vec{t}_j^{\vphantom{*}} \vec{t}_j^*\mright) \mat{V}^*
\end{equation}
where $\vec{t}_1,\ldots,\vec{t}_s$ are the columns of the matrix
\begin{equation} \label{eq:nys_update_t}
    \mat{T} = \mat{U}^*\mat{R}\mat{H}^{-1} \cdot \diag \mleft( (\mat{H}^{-1})_{ii}^{-1/2} : i=1,2,\ldots,s \mright).
\end{equation}
\end{subequations}
A derivation of this formula is provided in \cref{app:nystrom_update}.

The update formula facilitates efficient algorithms for the leave-one-out error estimator and jackknife variance estimates for the Nystr\"om approximation and derived quantities like spectral projectors and truncation of $\mat{X}$ to rank $r<s$.
To not belabor the point by presenting all possible variations, \cref{alg:nystrom} presents an implementation of Nystr\"om approximation without subspace iteration (i.e., $q = 0$) with the leave-one-out error estimate $\hat{\Err}(\mat{X},\mat{A})$.
The computation of the error estimate is a simple addition to the algorithm, requiring just a single line and taking only $\order(s^3)$ operations, independent of the size $d$ of the input matrix.
Further variants are discussed in \cref{sec:spectral_nystrom} and a MATLAB implementation is provided in \cref{list:nystrom}.

\begin{algorithm}[t]
  \caption{Nystr\"om approximation ($q=0$) with leave-one-out error estimate} \label{alg:nystrom}
  \textbf{Input:} $\mat{A}\in\field^{d\times d}$ to be approximated and approximation rank $s$ \\
  \textbf{Output:} Factors $\mat{V}\in\field^{d\times s}$ and $\mat{\Lambda}\in\real^{s\times s}$ defining a rank-$s$ approximation $\mat{X} = \mat{V}\mat{\Lambda}\mat{V}^*$ and leave-one-out error estimate $\hat{\Err}$
  \begin{algorithmic}[1]
    \State $\mat{\Omega} \leftarrow \mathtt{randn}(d,s)$
    \State $\mat{Y} \leftarrow \mat{A}\mat{\Omega}$
    \State $\nu \leftarrow \epsilon_{\rm mach} \norm{\mat{Y}}$ and  $\mat{Y} \leftarrow \mat{Y} + \nu \mat{\Omega}$ \Comment{Shift for numerical stability}
    \State $(\mat{Q},\mat{R}) \leftarrow \mathtt{qr}(\mat{Y}, \texttt{'econ'})$ \Comment{Economy \QR Factorization}
    \State $\mat{H} \leftarrow \mat{\Omega}^*\mat{Y}$
    \State $\mat{C} \leftarrow \mathtt{chol}((\mat{H}+\mat{H}^*)/2)$ \Comment{Upper triangular Cholesky decomposition $\mat{H} = \mat{C}^*\mat{C}$}
    \State $(\mat{U},\mat{\Sigma},\sim) \leftarrow \mathtt{svd}(\mat{R}\mat{C}^{-1})$ \label{line:svd_nystrom} \Comment{Triangular solve}
    \State $\mat{\Lambda} \leftarrow \max(\mat{\Sigma}^2-\nu\mathbf{I},0)$ \Comment{Entrywise maximum, shift back for numerical stability} \label{line:lambda}
    \State $\mat{V} \leftarrow \mat{Q}\mat{U}$
    \State $\hat{\Err} \leftarrow \norm{(\mat{R}\mat{C}^{-1})\mat{C}^{-*} \cdot \diag\{(\mat{H}^{-1}_{ii})^{-1} : i = 1,\ldots,s)\}}_{\rm F} / \sqrt{s}$
  \end{algorithmic}
\end{algorithm}

\subsection{Randomized SVD} \label{sec:rsvd}

The randomized SVD computes a rank-$s$ approximation to $\mat{A} \in \field^{d_1\times d_2}$ formed as an economy SVD $\mat{X} = \mat{U}\mat{\Sigma}\mat{V}^*$, where $\mat{U}\in\field^{d_1\times s}$ and $\mat{V} \in \field^{d_2\times s}$ have orthonormal columns and $\mat{\Sigma} \in \real^{s\times s}_+$ is diagonal.
With $q\ge 0$ steps of subspace iteration, the algorithm proceeds as follows:
\begin{enumerate}
\item Draw a \emph{test matrix} $\mat{\Omega}\in\complex^{d_2\times s}$ with iid isotropic columns.
\item Compute the product $\mat{Y} = (\mat{A}\mat{A}^*)^q \mat{A} \mat{\Omega}$. 
\item Orthonormalize $\mat{Q} = \operatorname{orth}(\mat{Y})$ using economy \QR factorization $\mat{Y} = \mat{Q}\mat{R}$.
\item Form the matrix $\mat{C} = \mat{Q}^*\mat{A}$.
\item Compute an economy SVD $\mat{C} = \mat{W}\mat{\Sigma}\mat{V}^*$.
\item Set $\mat{U} = \mat{Q}\mat{W}$.
\end{enumerate}
The output $\mat{X}$ is a symmetric function of the iid, isotropic columns $\vec{\omega}_1,\ldots,\vec{\omega}_s$ of $\mat{\Omega}$,
\begin{equation*}
    \mat{X} = \mat{X}(\vec{\omega}_1,\ldots,\vec{\omega}_s),
\end{equation*}
making it a candidate for leave-one-out error estimation and jackknife variance estimation.

To compute the replicates efficiently, we will use the following update formula for the $\mat{Q}$ matrix in the randomized SVD \cite[eq.~(2.1)]{ETW24}:
\begin{equation} \label{eq:rsvd_Q_update}
    \mat{Q}^{(j)}\mleft(\mat{Q}^{(j)}\mright)^* = \mat{Q}\mleft(\Id - \vec{t}_j^{\vphantom{*}}\vec{t}_j^*\mright)\mat{Q}^*
\end{equation}
where $\mat{Q}^{(j)}$ denotes the $\mat{Q}$ matrix produced by the randomized SVD algorithm executed without the $j$th column of $\mat{\Omega}$ and the vectors $\vec{t}_1,\ldots,\vec{t}_s$ are the normalized columns of $(\mat{R}^*)^{-1}$.
With this formula, the replicates are easily computed
\begin{equation} \label{eq:rsvd_update}
    \mat{X}^{(j)} = \mat{Q}^{(j)}\mleft(\mat{Q}^{(j)}\mright)^*\mat{A} = \mat{Q}\mleft(\Id - \vec{t}_j^{\vphantom{*}}\vec{t}_j^*\mright)\mat{Q}^*\mat{A} = \mat{U}(\Id - \mat{W}^*\vec{t}_j^{\vphantom{*}}\vec{t}_j^*\mat{W})\mat{\Sigma}\mat{V}^*.
\end{equation}

The update formula enables efficient algorithms for the leave-one-out error estimator and jackknife variance estimator for the approximation $\mat{X}$ and derived quantities like projectors onto singular subspaces and truncation of $\mat{X}$ to rank $r<s$.
As one example, \cref{alg:rsvd} gives an implementation the randomized SVD with no subspace iteration ($q = 0$) with the leave-one-out error estimator $\hat{\Err}(\mat{X},\mat{A})$.
The leave-one-out error estimator requires just two lines and runs in $\order(s^3)$ operations; see \cref{sec:loo-derivation} for a derivation of these two lines.
Further variants are discussed in \cref{sec:rsvd_extra} and a MATLAB implementation is provided in \cref{list:randsvd}.

\begin{algorithm}[t]
  \caption{Randomized SVD ($q = 0$) with jackknife variance estimate} \label{alg:rsvd}
  \textbf{Input:} $\mat{A}\in\field^{d_1\times d_2}$ to be approximated and approximation rank $s$ \\
  \textbf{Output:} Factors $\mat{U}\in\field^{d_1\times s}$, $\mat{\Sigma}\in\real^{s\times s}$, and $\mat{V}\in\field^{d_2\times s}$ defining a rank-$s$ approximation $\mat{X} = \mat{U}\mat{\Sigma}\mat{V}^*$, leave-one-out error estimate $\hat{\Err} = \hat{\Err}(\mat{X},\mat{A})$
  \begin{algorithmic}[1]
    \State $\mat{\Omega} \leftarrow \mathtt{randn}(d_2,s)$
    \State $\mat{Y} \leftarrow \mat{A}\mat{\Omega}$
    \State $(\mat{Q},\mat{R}) \leftarrow \mathtt{qr}(\mat{Y}, \texttt{'econ'})$ 
    \Comment{Economy \QR Factorization}
    \State $\mat{C} \leftarrow \mat{Q}^*\mat{A}$
    \State $(\mat{W}, \mat{\Sigma}, \mat{V}) \leftarrow \mathtt{svd}(
    \mat{C}, \texttt{'econ'})$
    \State $\mat{U} \leftarrow \mat{Q}\mat{W}$
    \State $\mat{G} \leftarrow (\mat{R}^*)^{-1}$
    \State $\hat{\Err} \leftarrow \mleft( s^{-1} \sum_{j=1}^s \norm{\mat{G}(:,j)}^{-2} \mright)^{1/2}$ \Comment{See \cref{sec:loo-derivation} for a derivation}
  \end{algorithmic}
\end{algorithm}

\section{Numerical experiments}
\label{sec:numer-exper}

In this section, we showcase numerical examples that demonstrate the effectiveness of the matrix jackknife variance estimate and leave-one-our error estimate for the Nystr\"om approximation and randomized SVD.
All numerical experiments work over the real numbers, $\field = \real$.

\subsection{Experimental setup}

To evaluate our diagnostics for matrices with different spectral characteristics, we consider synthetic test matrices from \cite[\S5]{TYUC17b}:
\begin{align}
  \mat{A} &= \diag( \underbrace{1,\ldots,1}_{R \textrm{ times}}, 0\ldots,0) + \xi d^{-1} \mat{G}\mat{G}^* \in \real^{d\times d}. \tag{NoisyLR} \label{eq:low_rank_plus_noise}\\
  \mat{A} &= \diag( \underbrace{1,\ldots,1}_{R \textrm{ times}},10^{-q},10^{-2q},\ldots,10^{-(d-R)q}) \tag{ExpDecay}\in \real^{d\times d}. \label{eq:exp_decay}
\end{align}
Here, $\xi,q\in\real$ are parameters, and $\mat{G} \in \real^{d\times d}$ is a standard Gaussian matrix. 
Using diagonal test matrices is justified by the observation that the randomized SVD, Nystr\"om approximation, and our diagnostics are orthogonally invariant when $\mat{\Omega}$ is a standard Gaussian matrix, which we use.
We also consider matrices from applications:
\begin{itemize}
\item \textbf{Velocity.} We consider a matrix $\mat{A} \in \real^{25096\times 1000}$ whose columns are snapshots of the velocity and pressure from simulations of a fluid flow past a cylinder.
  We thank Beverley McKeon and Sean Symon for this data.
\item \textbf{Spectral clustering.} The matrix $\mat{A}$ matrix from the spectral clustering example \cref{sec:spectral_clustering_numerics}.
\end{itemize}

We apply the jackknife variance estimate and leave-one-out error estimate to the randomized Nystr\"om approximation and randomized SVD with a standard Gaussian test matrix $\mat{\Omega}$ for a range of values for the approximation rank $s$.
For each value of $s$, we estimate the mean error $\Err$ \cref{eq:mean_error}, standard deviation $\Std$ \cref{eq:std}, mean jackknife estimate $\Jack$, or mean leave-one-our error estimator $\hat{\Err}$ using $1000$ independent trials.
Error bars on all figures show one standard deviation.

\subsection{Leave-one-out error estimator for randomized SVD}

First, we apply leave-one-out error estimation to estimate the error for the randomized SVD
\begin{equation}
    \mat{X} = \mat{U}\mat{\Sigma}\mat{V}^*.
\end{equation}
We set $q = 0$.
\Cref{fig:rsvd_loo} shows the results for three examples in the previous section.
In all cases, error estimate $\hat{\Err}(\mat{X},\mat{A})$ tracks the true error $\Err(\mat{X},\mat{A})$ closely.
Additional examples and analogous plots for randomized Nystr\"om approximation are provided in \cref{sec:addit-numer-exper}.

\begin{figure}[t]
  \centering
  \begin{subfigure}{0.32\textwidth}
      \includegraphics[width=0.99\textwidth]{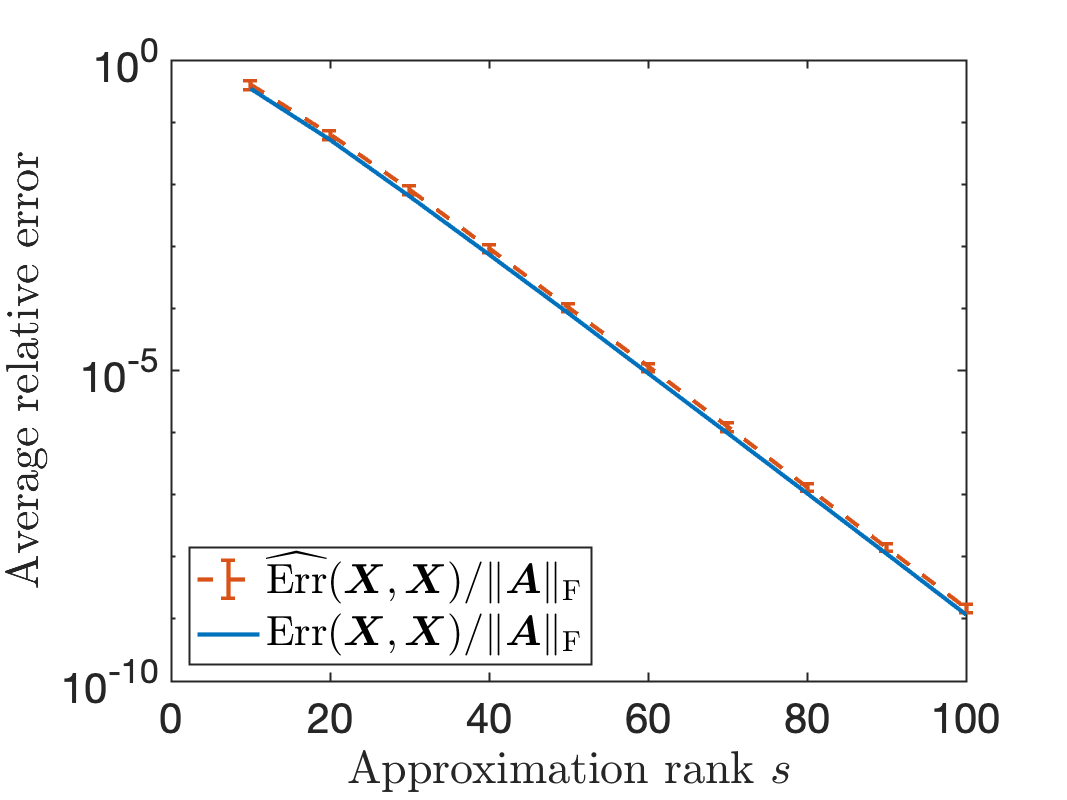}
      \caption*{\ref{eq:exp_decay}$(q \!=\! 0.1,R\!=\!5)$}
  \end{subfigure}
  \begin{subfigure}{0.32\textwidth}
      \includegraphics[width=0.99\textwidth]{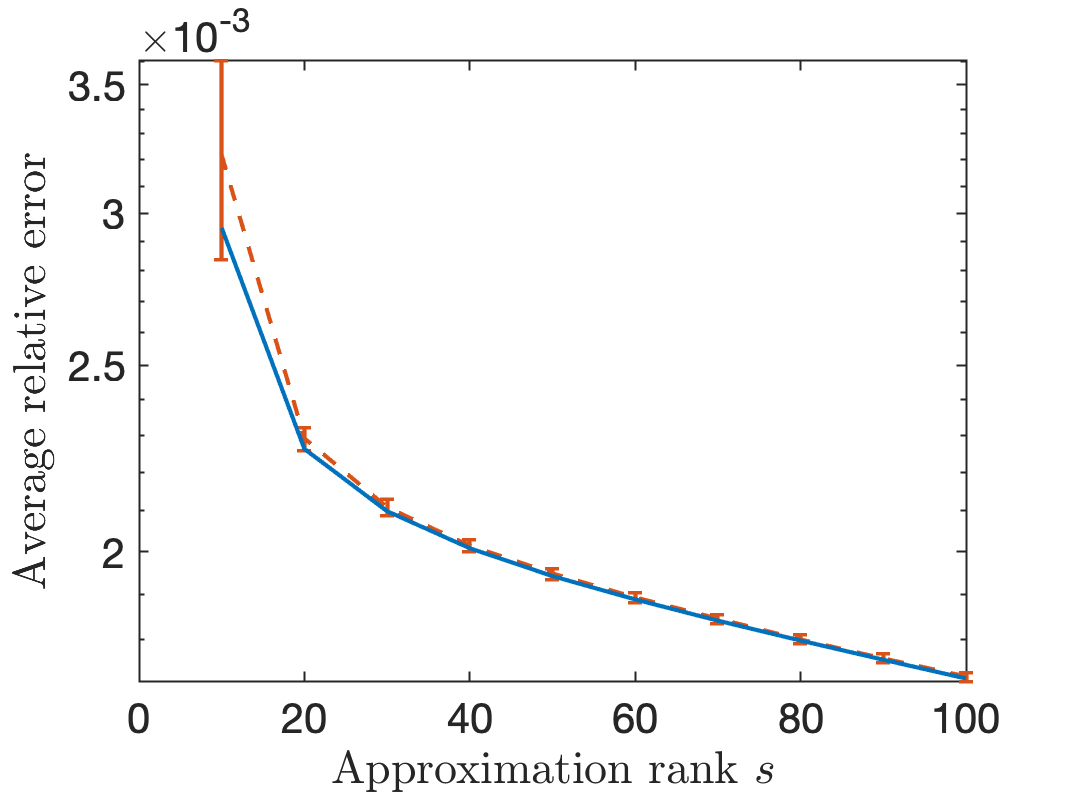}
      \caption*{\ref{eq:low_rank_plus_noise}$(\xi \!=\! 10^{-4},R\!=\!5)$}
  \end{subfigure}
  \begin{subfigure}{0.32\textwidth}
      \includegraphics[width=0.99\textwidth]{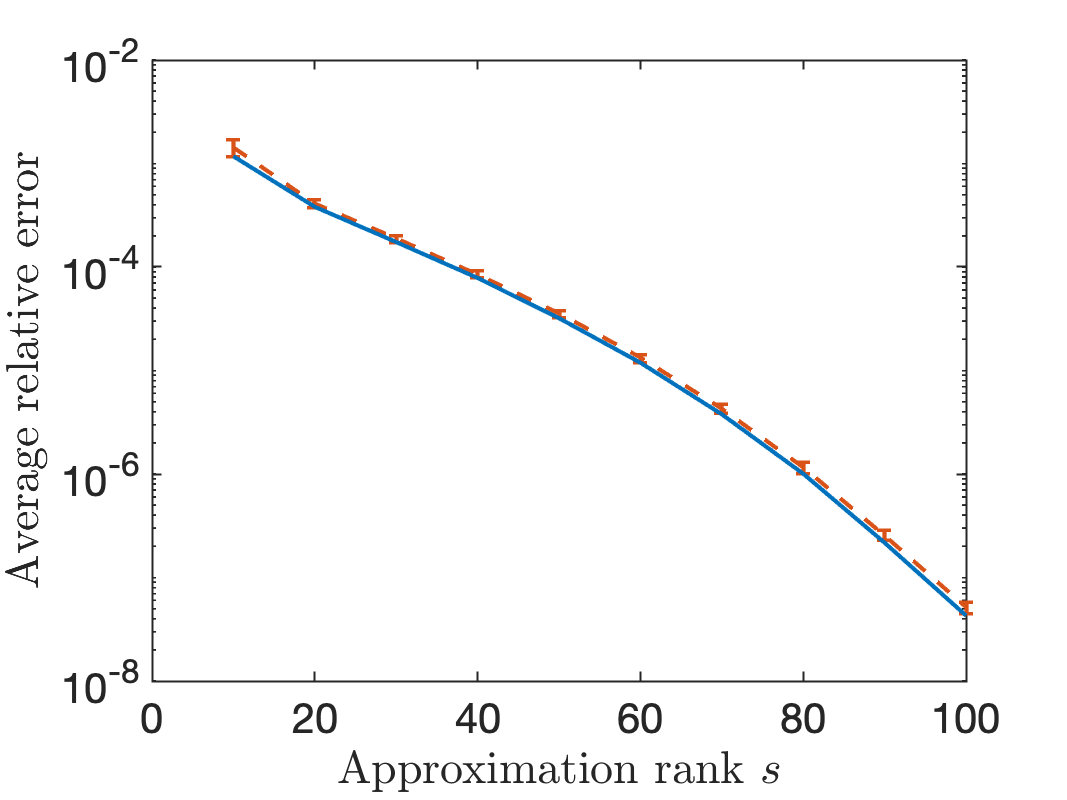}
      \caption*{Velocity}
  \end{subfigure}
  
  \mycaption{Leave-one-out error estimator for randomized SVD}{Error and error estimate for randomized SVD ($q=0$) approximation for three different matrices.}
  \label{fig:rsvd_loo}
\end{figure}

\subsection{Matrix jackknife for projectors onto singular subspaces}

Consider the task of computing the projector $\mat{\Pi}$ onto the dominant $5$-dimensional right singular subspace to a matrix $\mat{A}$.
The randomized SVD $\mat{A} \approx \mat{U}\mat{\Sigma}\mat{V}^*$ yields the approximation 
\begin{equation} \label{eq:singular_projector}
    \mat{X} = \mat{V}(:,1:5)\mat{V}(:,1:5)^*.
\end{equation}
\Cref{fig:rsvd_jack} shows the mean error $\Err(\mat{X},\mat{\Pi})$, standard deviation $\Std(\mat{X})$, and the jackknife estimate $\Jack(\mat{X})$ for the same three test matrices as previous.
We again set $q = 0$.
Consistent with \cref{thm:jackknife}, the jackknife estimate $\Jack(\mat{X})$ is an \emph{overestimate} of $\Std(\mat{X})$ by a factor of $2\times$ to $8\times$.
While the jackknife is not quantitatively sharp, it provides an order of magnitude estimate of the standard deviation and is a useful diagnostic for the quality of the computed output.
Additional examples and plots for randomized Nystr\"om approximations of projectors onto invariant subspaces are provided in \cref{sec:addit-numer-exper}.

\begin{figure}[t]
  \centering
  \begin{subfigure}{0.32\textwidth}
      \includegraphics[width=0.99\textwidth]{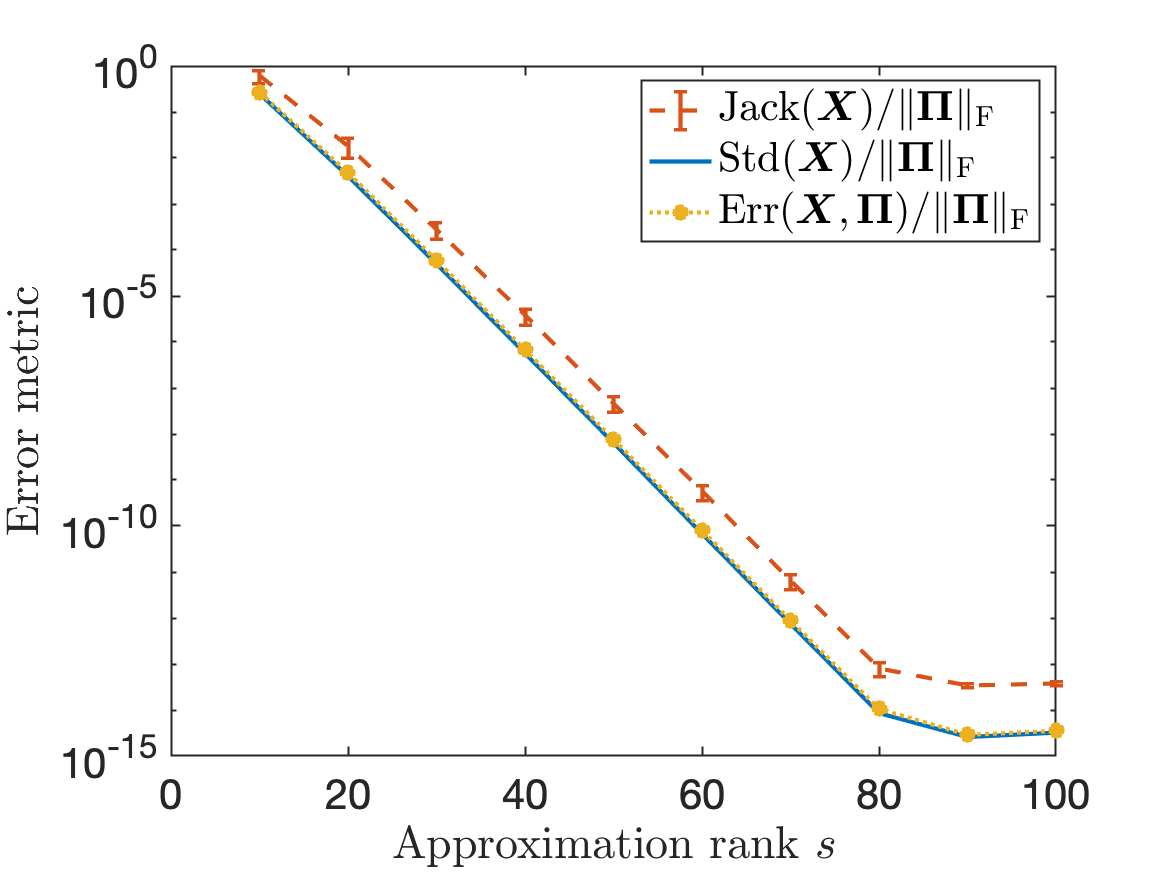}
      \caption*{\ref{eq:exp_decay}$(q \!=\! 0.1,R\!=\!5)$}
  \end{subfigure}
  \begin{subfigure}{0.32\textwidth}
      \includegraphics[width=0.99\textwidth]{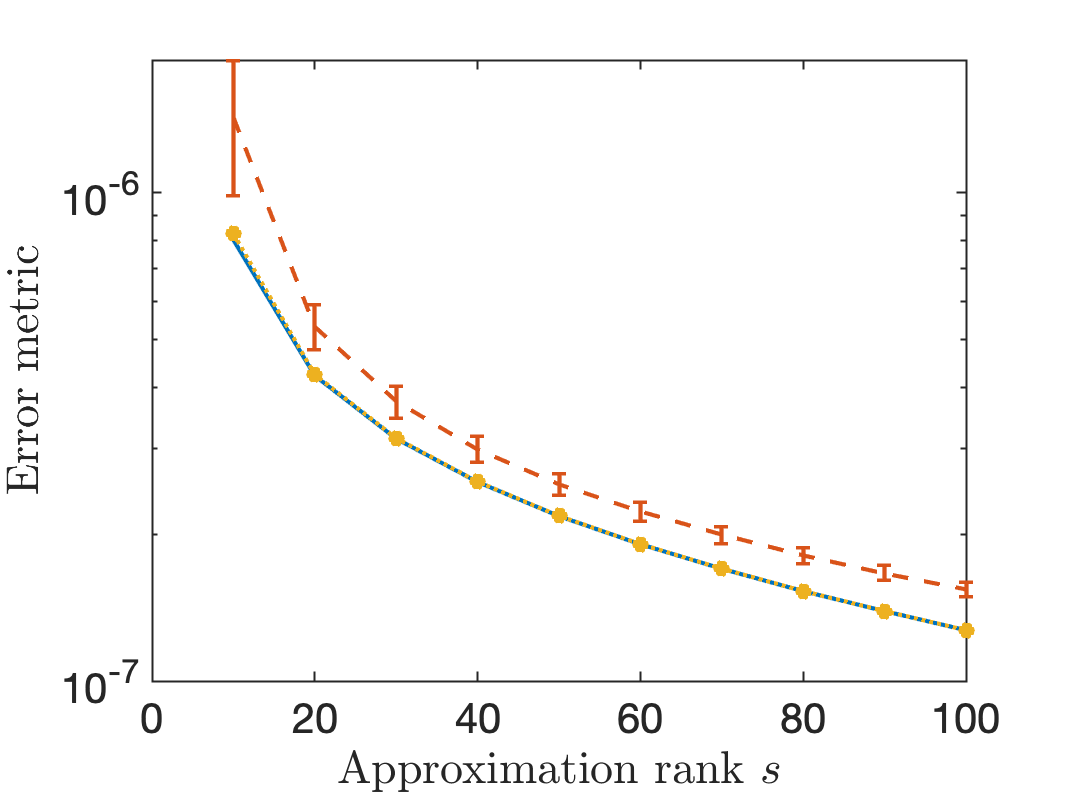}
      \caption*{\ref{eq:low_rank_plus_noise}$(\xi \!=\! 10^{-4},R\!=\!5)$}
  \end{subfigure}
  \begin{subfigure}{0.32\textwidth}
      \includegraphics[width=0.99\textwidth]{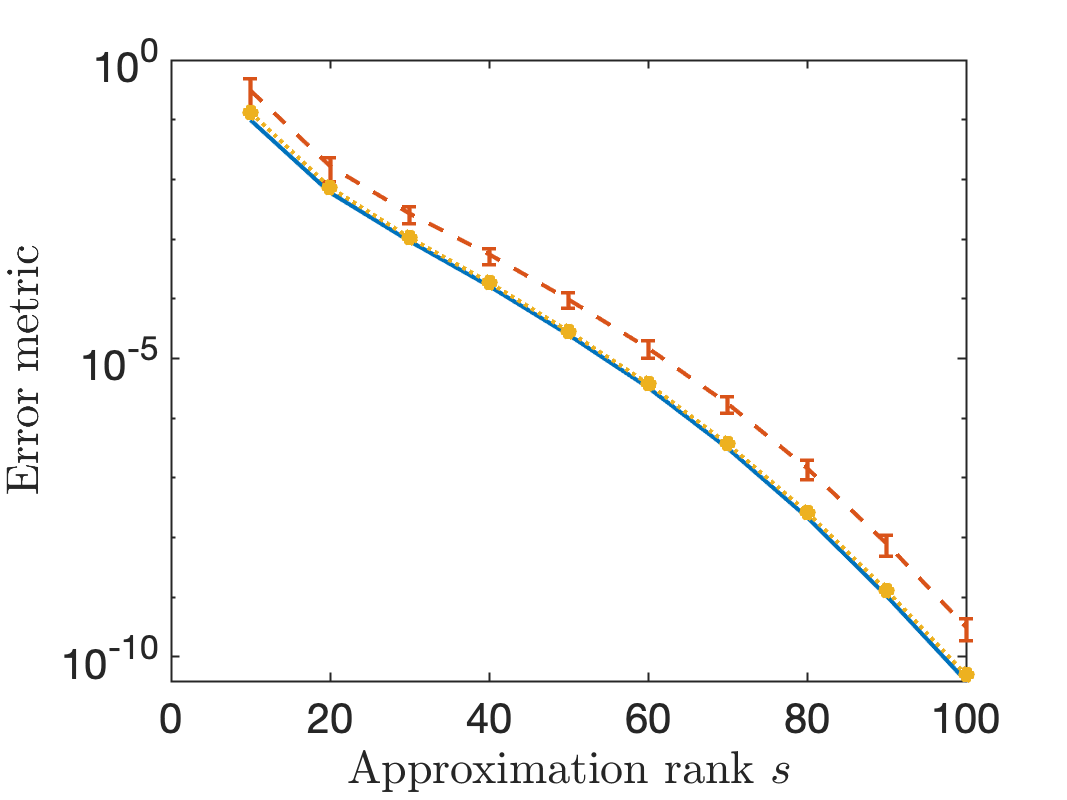}
      \caption*{Velocity}
  \end{subfigure}
  
  \mycaption{Matrix jackknife for projectors onto singular subspaces}{Error, standard deviation, and jackknife estimate for randomized SVD ($q=0$) approximation $\mat{X}$ \cref{eq:singular_projector} to the projector $\mat{\Pi}$ onto the span of the five dominant right singular vectors for three different matrices.}
  \label{fig:rsvd_jack}
\end{figure}

\subsection{Application: Diagnosing ill-conditioning for spectral clustering}

Jackknife variance estimates can be used to identify situations when a computational task is ill-posed or ill-conditioned.
In such cases, refining the approximation (i.e., increasing $s$ or $q$) may be of little help to improve the quality of the computation.

As an example, consider the spectral clustering application from \cref{sec:spectral_clustering}.
In Nystr\"om-accelerated spectral clustering, we use the dominant $n_{\rm dim}$ eigenvectors of a Nystr\"om approximation $\mat{V}\mat{\Lambda}\mat{V}^*$ as coordinates for k-means clustering. 
For spectral clustering to be reliable, we should pick the parameter $n_{\rm dim}$ such that these coordinates are well-conditioned; that is, they should not be highly sensitive to small changes in the normalized kernel matrix $\mat{A}$.
For the example in \cref{sec:spectral_clustering_numerics}, the five largest eigenvalues of $\mat{A}$ are
\begin{align*}
    \lambda_1 &= 0.999999999999999 \\
    \lambda_2 &= 0.999999998639842 \\
    \lambda_3 &= 0.999999940523446 \\
    \lambda_4 &= 0.999999931126177 \\ 
    \lambda_5 &= 0.99\mathbf{7}867975285136 
\end{align*}
The first four eigenvalues agree up to eight digits of accuracy, with the fifth eigenvalue separated by $\lambda_4 - \lambda_5 \approx 2\times 10^{-3}$.
Based on these values, the natural parameter setting would be $n_{\rm dim} = 4$, as the first four eigenvalues are nearly indistinguishable but are well-separated from the fifth.

\begin{figure}[t]
  \centering
  \includegraphics[width=0.7\textwidth]{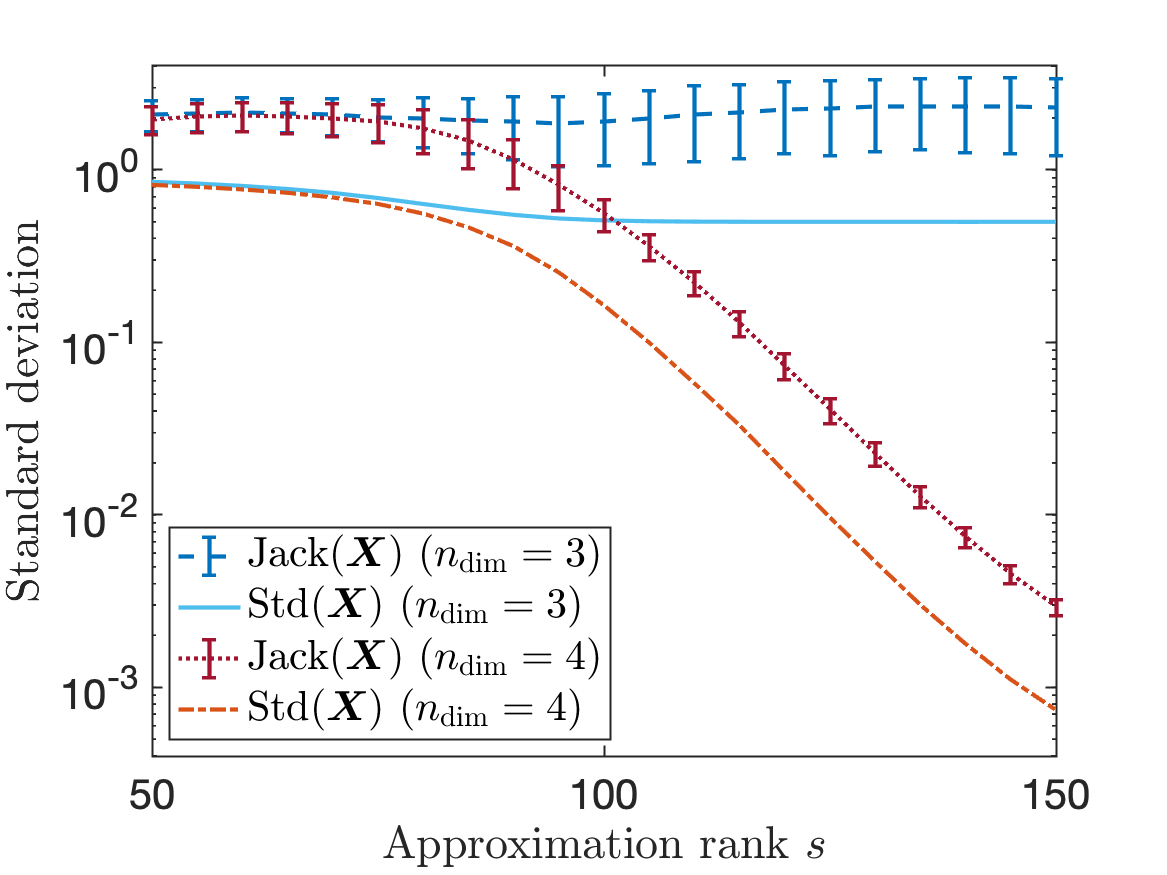}
  
  \mycaption{Detecting ill-conditioning}{Standard deviation and jackknife estimate for the Nystr\"om spectral projector \cref{eq:spectral_projector} associated with the dominant $n_{\rm dim}$-dimensional invariant subspace for $n_{\rm dim} \in \{3,4\}$ for the spectral clustering matrix.}
  \label{fig:illconditioned}
\end{figure}

When we use Nystr\"om-accelerated spectral clustering, we do not have access to the true eigenvalues of the matrix $\mat{A}$ and thus can have difficulties selecting the parameter $n_{\rm dim}$ appropriately.
Fortunately, the jackknife variance estimate can help warn the user of a poor choice for $n_{\rm dim}$.
Let 
\begin{equation} \label{eq:spectral_projector}
    \mat{X} = \mat{V}(:,1:n_{\rm dim})\mat{V}(:,1:n_{\rm dim})^*
\end{equation}
be the orthoprojector onto the dominant $n_{\rm dim}$-dimensional invariant subspace of the Nystr\"om approximation $\mat{V}\mat{\Lambda}\mat{V}^*$.
\Cref{fig:illconditioned} shows the standard deviation $\Std(\mat{X})$ and its jackknife estimate $\Jack(\mat{X})$ for both $n_{\rm dim} = 3$ and $n_{\rm dim} = 4$.
As in \cref{sec:spectral_clustering_numerics}, use $q = 3$ and $50 \le s \le 150$.
For the good parameter setting $n_{\rm dim} = 4$, the variance decreases sharply as $s$ is increased.
For the bad choice $n_{\rm dim} = 3$, the variance remains persistently high, even as the approximation is refined.
This provides evidence to the user that $\mat{X}$ is poorly conditioned and allows the user to fix this by changing the parameter $n_{\rm dim}$.

\subsection{Application: POD modes}

\begin{figure}[t]
  \centering
  \begin{subfigure}[b]{0.48\textwidth}
    \centering
    \includegraphics[width=\textwidth]{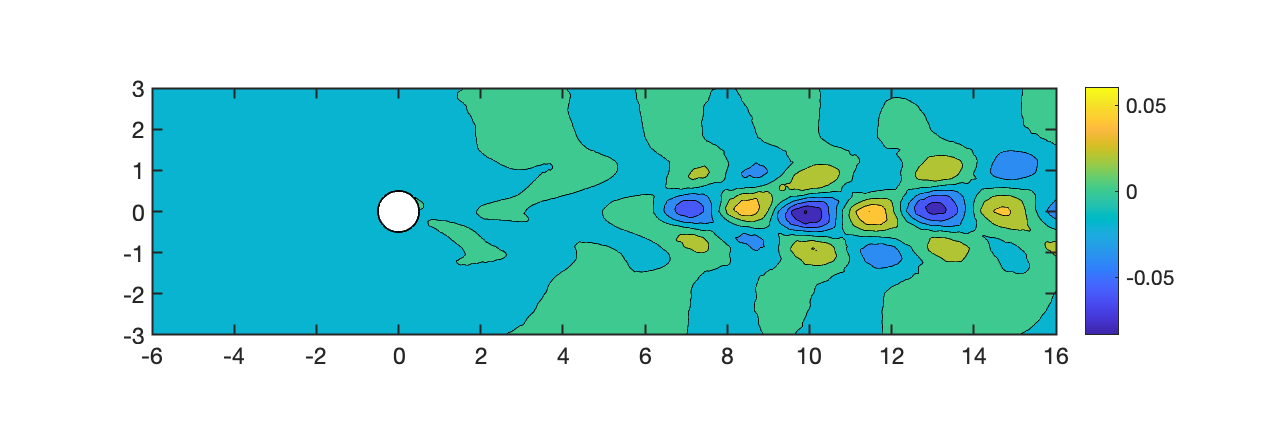}
    \caption{Exact} \label{fig:vel_exact}
  \end{subfigure}
  ~
  \begin{subfigure}[b]{0.48\textwidth}
    \centering
    \includegraphics[width=\textwidth]{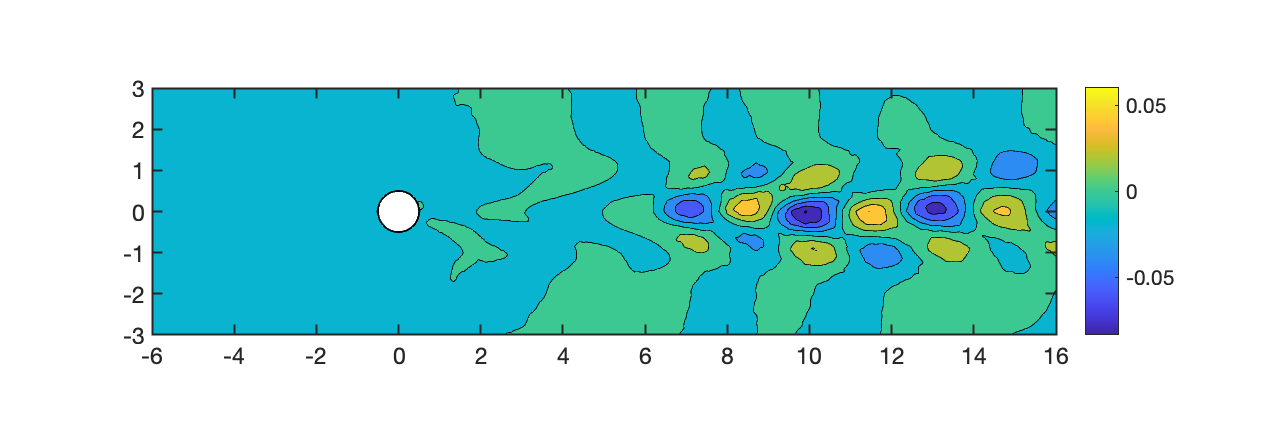}
    \caption{Randomized SVD ($s=20$, $q=0$)} \label{fig:vel_approx}
  \end{subfigure}

  \begin{subfigure}[b]{0.48\textwidth}
    \centering
    \includegraphics[width=\textwidth]{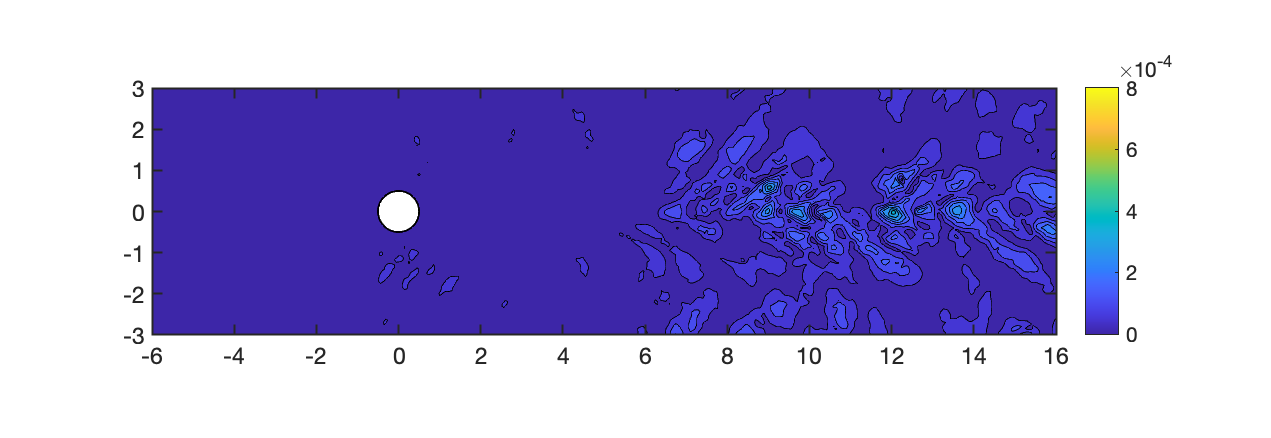}
    \caption{Error} \label{fig:vel_diff}
  \end{subfigure}
  \begin{subfigure}[b]{0.48\textwidth}
    \centering
    \includegraphics[width=\textwidth]{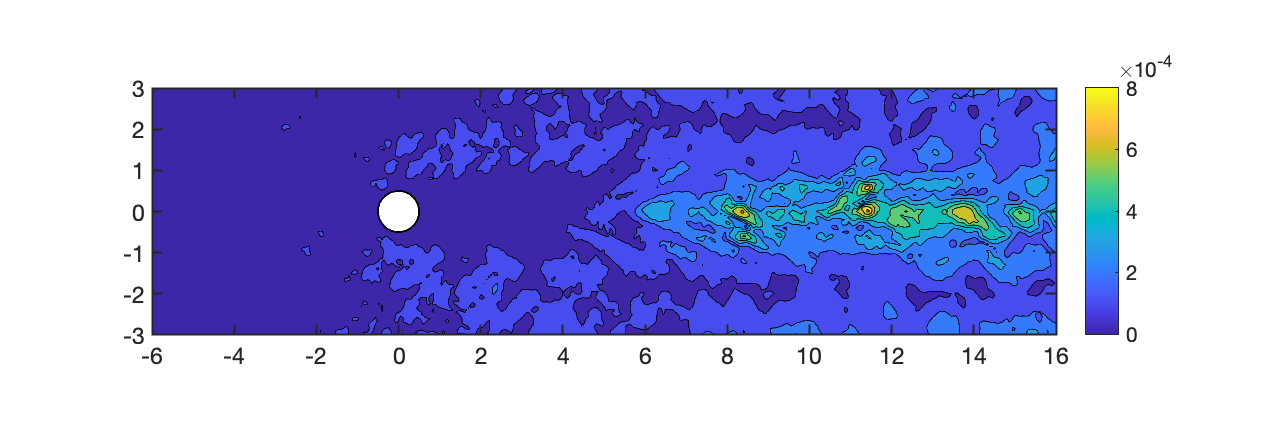}
    \caption{Jackknife standard deviation estimate} \label{fig:vel_jack}
  \end{subfigure}
  
  \mycaption{Jackknife for singular vector}{These panels assess the entrywise errors in the streamwise velocity from the fifth left singular vector of the velocity test matrix.
    Panel (\subref{fig:vel_exact}) shows the exact answer, and panel (\subref{fig:vel_approx}) shows the estimate produced by the randomized SVD.
    Panels (\subref{fig:vel_diff}) and (\subref{fig:vel_jack}) display the error and the the Tukey jacknife standard deviation estimate.
  The jackknife estimate presents a descriptive portrait of where the error is localized.}
  \label{fig:vel}
\end{figure}

Jackknife variance estimation can be used to give more fine-grained information about a randomized matrix computation.
In \cref{fig:vel}, we compute a (scalar) jackknife variance estimate for the absolute value of each entry of the fifth left singular vector of the velocity matrix computed using the randomized SVD with $s = 20$ and $q=0$.
(The absolute value is introduced to avoid sign ambiguities.)

Left singular vectors of a matrix of simulation data are known as POD modes and are useful for data visualization and model reduction.
Since variance is a lower bound on the mean-square error, the coordinate-wise jackknife variance estimates can be used a diagnostic to help identify regions of a POD mode that have high error.
This is demonstrated in \cref{fig:vel}; the jackknife estimate is not quantitatively sharp but paints a descriptive portrait of where the errors are localized.

\section{Extension: Variance estimates for higher Schatten norms}
\label{sec:high-schatt-norms}

The variance estimate $\Jack(\mat{X})$ serves as an estimate for the \emph{Frobenius-norm} variance
\begin{equation*}
  \Jack^2(\mat{X}) \approx \Var(\mat{X}) = \expect \norm{\mat{X} - \expect \mat{X}}^2_{\rm F}.
\end{equation*}
Often, it is more desirable to have error or variance estimates for Schatten norms $\schatten{\cdot}_p$ with $p > 2$, defined as the $\ell_p$ norm of the singular values:
\begin{equation*}
  \schatten{\mat{B}}_p^p \coloneqq \sum_{j=1}^{\min(d_1,d_2)} \sigma_j^p (\mat{B}).
\end{equation*}
One can also construct jackknife estimates for the variance in higher Schatten norms, although the estimates take more intricate forms.
For this section, fix an even number $p\ge 2$ and assume the same setup as \cref{sec:matr-jackkn-estim} with the additional stipulation that the samples $\omega_1,\ldots,\omega_s$ take values in a Polish space $\Omega$.

The jackknife variance estimate is defined as follows.
Consider matrix-valued jackknife \emph{variance proxies}:
\begin{equation*}
  \mat{\widehat{\Var}}_1(\mat{X}_{s-1}) \coloneqq \frac{1}{2} \sum_{j=1}^{s-1} \matabssq{\mat{X}^{(s)} - \mat{X}^{(j)}} \quad \textrm{and} \quad \mat{\widehat{\Var}}_2(\mat{X}_{s-1}) \coloneqq \frac{1}{2} \sum_{j=1}^{s-1} \matabssq{\left(\mat{X}^{(s)} - \mat{X}^{(j)}\right)^*}.
\end{equation*}
Here, $|\cdot|$ denotes the matrix modulus $|\mat{B}| \coloneqq (\mat{B}^*\mat{B})^{1/2}$.
Define the Schatten $p$-norm variance estimate
\begin{equation*}
  \Jack_p(\mat{X}_{s-1}) \coloneqq 2^{-\tfrac{1}{p}}\sqrt{2(p-1)} \mleft( \schatten{\mat{\widehat{\Var}}_1}_{p/2}^{p/2} + \schatten{\mat{\widehat{\Var}}_2}_{p/2}^{p/2} \mright)^{1/p}.
\end{equation*}
This quantity seeks to approximate
\begin{equation*}
  \Jack_p^p(\mat{X}_{s-1}) \approx \expect \schatten{\mat{X}_{s-1} - \expect\mat{X}_{s-1}}_p^p.
\end{equation*}
A matrix generalization of the Efron--Stein--Steele inequality \cite[Thm.~4.2]{PMT16} shows this jackknife variance estimate overestimates the Schatten $p$-norm variance in the sense
\begin{equation*}
  \expect \schatten{\mat{X}_{s-1} - \expect\mat{X}_{s-1}}_p^p \le \expect \, \Jack_p^p(\mat{X}_{s-1}).
\end{equation*}
The techniques we introduce in \cref{sec:computations} can be extended in a natural way to compute $\Jack_p(\mat{X}_{s-1})$ efficiently for the randomized SVD and Nystr\"om approximation.

\appendix 

\section{Derivation of update formulas}

In this section, we provide proofs of the update formulas \cref{eq:nys_update} and \cref{eq:rsvd_Q_update}.

\subsection{Proof of \cref{eq:nys_update}}
\label{app:nystrom_update}

Assume without loss of generality that $j = s$, instate the notation of \cref{sec:nystrom}, and assume $\mat{H}$ is invertible.
The key ingredient is the following consequence of the Banachiewicz inversion formula \cite[eq.~(0.7.2)]{PS05}:
\begin{equation*}
    \twobytwo{\big(\mat{H}^{(s)}\big)^{-1}}{\vec{0}}{\vec{0}^*}{0} = \mat{H}^{-1} - \frac{\mat{H}^{-1}\evec_s^{\vphantom{*}}\evec_s^* \mat{H}^{-1}}{\evec_s^*\mat{H}^{-1}\evec_s^{\vphantom{*}}}.
\end{equation*}
Using this formula and letting $\mat{R}_{-s}$ denote $\mat{R}$ without its $s$th column, we compute
\begin{equation*}
    \mat{X}^{(s)} = \mat{Q}\mat{R}_{-s}^{\vphantom{*}} \big(\mat{H}^{(s)}\big)^{-1}\mat{R}_{-s}^*\mat{Q}^* = \mat{X} - \frac{\mat{Q}\mat{R}\mat{H}^{-1}\evec_s^{\vphantom{*}}\evec_s^*\mat{H}^{-1}\mat{R}^*\mat{Q}^*}{\evec_s^*\mat{H}^{-1}\evec_s^{\vphantom{*}}}.
\end{equation*}
Since $\mat{Q} = \mat{V}\mat{U}^*$ and $\mat{X} = \mat{V}\mat{\Lambda}\mat{V}^*$, we thus have
\begin{equation*}
    \mat{X}^{(s)} = \mat{V} \mleft( \mat{\Lambda} - \frac{\mat{U}^*\mat{R}\mat{H}^{-1}\evec_s^{\vphantom{*}}\evec_s^*\mat{H}^{-1}\mat{R}^*\mat{Q}^*\mat{U}}{\evec_s^*\mat{H}^{-1}\evec_s^{\vphantom{*}}} \mright) \mat{V}^* = \mat{V} \mleft(\mat{\Lambda} - \vec{t}_s^{\vphantom{*}} \vec{t}_s^* \mright) \mat{V}^*,
\end{equation*}
where $\vec{t}_s$ is defined in \cref{eq:nys_update_t}.
The formula is established. \hfill $\proofbox$

\subsection{Proof of \cref{eq:rsvd_Q_update}}
\label{app:rsvd_update}
Fix $j \in \{1,\ldots,s\}$, instate the notation of \cref{sec:rsvd}, and assume $\mat{R}$ is invertible.
First observe that
\begin{equation*}
    \mat{Y}^{(j)} = \mat{Q} \mat{R}_{-j},
\end{equation*}
where $\mat{R}_{-j}$ is $\mat{R}$ without its $j$th column.
To compute $\mat{Q}^{(j)}$, we need an economy \QR factorization of $\mat{Y}^{(j)} = \mat{Q}\mat{R}_{-j}$.
To this end, compute a (full) \QR decomposition of $\mat{R}_{-j}$:
\begin{equation} \label{eq:Rj_QR}
    \mat{R}_{-j} = \onebytwo{\mat{Q}'}{\vec{t}_j} \twobyone{\mat{R}^{(j)}}{\vec{0}^*},
\end{equation}
where $\mat{Q}' \in \field^{s\times (s-1)}$, $\mat{R}^{(j)} \in \field^{(s-1)\times (s-1)}$, and $\vec{t}_j \in \field^s$.
Then
\begin{equation*}
    \mat{Y}^{(j)} = \mat{Q}^{(j)} \mat{R}^{(j)} \quad \text{for} \quad \mat{Q}^{(j)} = \mat{Q}\mat{Q}'
\end{equation*}
is an economy \QR decomposition of $\mat{Y}^{(j)}$.
Since $\onebytwo{\mat{Q}'}{\vec{t}_j}$ is orthogonal, we have
\begin{equation*}
    \Id = \onebytwo{\mat{Q}'}{\vec{t}_j} \onebytwo{\mat{Q}'}{\vec{t}_j}^* = \mat{Q}'(\mat{Q}')^* + \vec{t}_j^{\vphantom{*}} \vec{t}_j^* \implies \mat{Q}'(\mat{Q}')^* = \Id - \vec{t}_j^{\vphantom{*}} \vec{t}_j^*.
\end{equation*}
Thus,
\begin{equation*}
    \mat{Q}^{(j)} \bigl( \mat{Q}^{(j)}\bigr)^* = \mat{Q}\mat{Q}'(\mat{Q}')^* \mat{Q}^* = \mat{Q} \mleft( \Id - \vec{t}_j^{\vphantom{*}}\vec{t}_j^*\mright) \mat{Q}^*.
\end{equation*}
Finally, observe that \cref{eq:Rj_QR} implies that $\vec{t}_j$ is orthogonal to the columns of $\mat{R}_{-j}$ so 
\begin{equation*}
    \vec{t}_j^* \mat{R} \quad \text{is a nonzero multiple of $\mathbf{e}_j^*$}.
\end{equation*}
Therefore, $\vec{t}_j$ is proportional to the $j$th column of $(\mat{R}^*)^{-1}$.
Since $\vec{t}_j$ is a column of an orthogonal matrix, it is thus a normlized version of the $j$th column of $(\mat{R}^*)^{-1}$. \hfill $\proofbox$

\section*{Acknowledgements}

We thank Per-Gunnar Martinsson, Eitan Levin, and Robert Webber for their advice and feedback.

\section*{Disclaimer}

This report was prepared as an account of work sponsored by an agency of the United States Government. Neither the United States Government nor any agency thereof, nor any of their employees, makes any warranty, express or implied, or assumes any legal liability or responsibility for the accuracy, completeness, or usefulness of any information, apparatus, product, or process disclosed, or represents that its use would not infringe privately owned rights. Reference herein to any specific commercial product, process, or service by trade name, trademark, manufacturer, or otherwise does not necessarily constitute or imply its endorsement, recommendation, or favoring by the United States Government or any agency thereof. The views and opinions of authors expressed herein do not necessarily state or reflect those of the United States Government or any agency thereof.

\bibliographystyle{siamplain}
\bibliography{references}

\newpage
\begin{center} \textbf{SUPPLEMENTARY MATERIAL}
\end{center}

\setcounter{section}{0} \renewcommand{\thesection}{SM\arabic{section}}
\setcounter{figure}{0} \renewcommand{\thefigure}{SM\arabic{figure}}
\renewcommand\appendixname{Section}

\section{More on efficient jackknife algorithms}
In this section, we discuss additional implementation details for fast computation of the jackknife variance estimate.
MATLAB R2022b implementations are provided in \cref{sec:matlab}.

\subsection{Spectral computations with Nystr\"om approximation}
\label{sec:spectral_nystrom}
One of the great advantages of jackknife variance estimation is can be applied to a wide variety of objects such as eigenvalues, eigenvectors, projectors onto invariant subspaces, and truncations of a matrix to fixed rank $r < s$.
Computing jackknife variance estimates \emph{efficiently} requires care. 
The essential ingredient is $\order(s^2)$ algorithms for computing the eigendecomposition of a rank-one modification of a diagonal matrix \cite[\S5.3.3]{Dem97}.

As an example, consider the jackknife variance estimate for the projector onto the dominant eigenspace
\begin{equation*}
    \mat{X} = \vec{v}\vec{v}^* \quad \text{where} \quad \vec{v} = \mat{V}(:,1)
\end{equation*}
of the Nystr\"om approximation.
Using the update formula \cref{eq:nys_update}, the replicates take the form
\begin{equation*}
    \mat{X}^{(j)} = \vec{v}^{(j)}\mleft(\vec{v}^{(j)}\mright)^* \quad \text{where $\vec{v}^{(j)}$ is the dominant eigenvector of} \quad \mat{V}\mleft(\mat{\Lambda} - \vec{t}_j^{\vphantom{*}}\vec{t}_j^*\mright)\mat{V}^*.
\end{equation*}
To compute the replicates efficiently, we make use of the fact that the eigendecomposition
\begin{equation} \label{eq:rank_one_modification}
    \mat{\Lambda} - \vec{t}_j^{\vphantom{*}}\vec{t}_j^* = \mat{W}_j^{\vphantom{*}}\mat{\Lambda}_j \mat{W}_j^*
\end{equation}
of a rank-one modification of a diagonal matrix can be computed in $\order(s^2)$ operations; see \cite[\S5.3.3]{Dem97}.
By computing the eigendecomposition \cref{eq:rank_one_modification} for each $j = 1,\ldots,s$, we have a representation of the replicates
\begin{equation*}
    \mat{X}^{(j)} = \mat{V}\mat{W}_j(:,1)\mat{W}_j(:,1)^*\mat{V}^*.
\end{equation*}
Therefore, the mean of the replicates is
\begin{equation*}
    \mat{X}^{(\cdot)} = \mat{V}\mat{M}\mat{V}^* \quad \text{where} \quad \mat{M} = \frac{1}{s} \sum_{j=1}^s \mat{W}_j(:,1)\mat{W}_j(:,1)^*
\end{equation*}
and the jackknife variance estimate is
\begin{align*}
    \Jack^2(\mat{X}) = \sum_{j=1}^s \norm{\mat{V}\mleft(\mat{W}_j(:,1)\mat{W}_j(:,1)^* - \mat{M}\mright)\mat{V}^*}_{\rm F}^2 = \sum_{j=1}^s \norm{\mat{W}_j(:,1)\mat{W}_j(:,1)^* - \mat{M}}_{\rm F}^2.
\end{align*}
Using this formula, we can compute $\Jack(\mat{X})$ in $\order(s^3)$ operations, independent of the dimension $d$ of the input matrix.

\subsection{Efficient jackknife procedures for spectral clustering} \label{sec:efficient_spectral_clustering}
As a more elaborate example, we sketch the $\order(ds^2)$ algorithm used to produce the jackknife variance estimates from \cref{sec:spectral_clustering}.
Instate the notation of \cref{sec:spectral_clustering}.
We begin, as in the previous section, by computing the eigendecomposition \cref{eq:rank_one_modification} for $j = 1,\ldots,s$, requiring $\order(s^3)$ operations in total.
Let $\mat{Q}_j$ denote the first $n_{\rm dim}$ columns of $\mat{W}_j$.
The jackknife replicates are
\begin{equation} \label{eq:spectral_clustering_replicates}
    \mat{X}^{(j)} = \frac{\mat{D}^{-1/2}\mat{V}\mat{Q}_j^{\vphantom{*}}\mat{Q}_j^*\mat{V}^*\mat{D}^{-1/2}}{\norm{\mat{D}^{-1/2}\mat{V}\mat{Q}_j^{\vphantom{*}}\mat{Q}_j^*\mat{V}^*\mat{D}^{-1/2}}_{\rm F}} \quad \text{for $j = 1,\ldots,s$}.
\end{equation}
Introduce
\begin{equation*}
    \mat{G} \coloneqq \mat{V}^*\mat{D}^{-1}\mat{V}, \quad \mat{F}_j \coloneqq \mat{Q}_j/\norm{\mat{Q}_j^*\mat{G}\mat{Q}_j}_{\rm F}^{1/2}.
\end{equation*}
Then the replicates can be written as
\begin{equation*}
    \mat{X}^{(j)} = \mat{D}^{-1/2}\mat{V}\mat{F}_j^{\vphantom{*}}\mat{F}_j^*\mat{V}^*\mat{D}^{-1/2} \quad \text{for $j = 1,\ldots,s$}
\end{equation*}
and their mean is
\begin{equation*}
    \mat{X}^{(\cdot)} = \mat{D}^{-1/2}\mat{V}\mat{M}\mat{V}^*\mat{D}^{-1/2}\quad \text{where}\quad \mat{M} = \frac{1}{s}\sum_{j=1}^s \mat{F}_j^{\vphantom{*}}\mat{F}_j^*.
\end{equation*}
The jackknife variance estimate is
\begin{align*}
    \Jack^2(\mat{X})
    &= \sum_{j=1}^s \norm{\mat{D}^{-1/2}\mat{V}\mleft(\mat{F}_j^{\vphantom{*}}\mat{F}_j^* - \mat{M}\mright)\mat{V}^*\mat{D}^{-1/2}}_{\rm F}^2 \\
    &= \sum_{j=1}^s \tr \mleft(\mat{G}\mleft(\mat{F}_j^{\vphantom{*}}\mat{F}_j^* - \mat{M}\mright)\mat{G}\mleft(\mat{F}_j^{\vphantom{*}}\mat{F}_j^* - \mat{M}\mright)\mright) \\
    &= s \tr(\mat{G}\mat{M}\mat{G}\mat{M}) - 2\sum_{j=1}^s  \tr\mleft(\mat{F}_j^*\mat{G}\mat{M}\mat{G}\mat{F}_j^{\vphantom{*}}\mright) + \sum_{j=1}^s \norm{\mat{F}_j^*\mat{G}\mat{F}_j^{\vphantom{*}}}^2_{\rm F}.
\end{align*}
Using this formula, $\Jack(\mat{X})$ can be formed in $\order(ds^2)$ operations.
(In fact, the most expensive part of the calculation is the formation of $\mat{G}$; everything else requires $\order(s^3)$ operations.)

\begin{algorithm}[t]
  \caption{Nystr\"om-accelerated spectral clustering with variance estimate} \label{alg:spectral_clustering}
  \textbf{Input:} Data points $\vec{c}_1,\ldots,\vec{v}_d$, dimension $n_{\rm dim}$ of clustering space, number of clusters $n_{\rm cen}$, kernel function $\kappa : \real^d \times \real^d \to \real_+$, and Nystr\"om approximation rank $s$ and subspace iteration steps $q$ \\
  \textbf{Output:} Clusters $\mathsf{C}$ and jackknife estimate $\Jack$
  \begin{algorithmic}[1]
    \For{$i,j=1,\ldots,d$}
        \State $k_{ij} \leftarrow \kappa(\vec{c}_i,\vec{c}_j)$
    \EndFor
    \State $\mat{D} \leftarrow \diag\big( \sum_{j=1}^d k_{ij} : i=1,\ldots,d \big)$
    \State $\mat{\Phi} \leftarrow \mathtt{randn}(d,s)$
    \For{$i = 1,\ldots,q$}
        \State $\mat{\Phi} \leftarrow \mat{D}^{-1/2}(\mat{K}(\mat{D}^{-1/2}\mat{\Phi}))$
    \EndFor 
    \State $\mat{Y} \leftarrow \mat{D}^{-1/2}(\mat{K}(\mat{D}^{-1/2}\mat{\Phi}))$
    \State $\nu \leftarrow \epsilon_{\rm mach} \norm{\mat{Y}}$ and  $\mat{Y} \leftarrow \mat{Y} + \nu \mat{\Phi}$ \Comment{Shift for numerical stability}
    \State $(\mat{Q},\mat{R}) \leftarrow \mathtt{qr}(\mat{Y}, \texttt{'econ'})$ \Comment{Economy \QR Factorization}
    \State $\mat{H} \leftarrow \mat{\Omega}^*\mat{Y}$
    \State $\mat{C} \leftarrow \mathtt{chol}((\mat{H}+\mat{H}^*)/2)$ \Comment{Upper triangular Cholesky decomposition $\mat{H} = \mat{C}^*\mat{C}$}
    \State $(\mat{U},\mat{\Sigma},\sim) \leftarrow \mathtt{svd}(\mat{R}\mat{C}^{-1})$ \Comment{Triangular solve}
    \State $\mat{\Lambda} \leftarrow \max(\mat{\Sigma}^2-\nu\mathbf{I},0)$ \Comment{Entrywise maximum, shift back for numerical stability} 
    \State $\mat{V} \leftarrow \mat{Q}\mat{U}$, $\mat{W} \leftarrow \mat{D}^{-1/2}\mat{V}(:,1:n_{\rm cen})$
    \State $\mathsf{C} \leftarrow \mathtt{kmeans}(\{\mat{W}(:,i) : i = 1,\ldots,d\},n_{\rm cen})$
    \State $\mat{G} \leftarrow \mat{V}^*\mat{D}^{-1}\mat{V}$
    \State $\mat{T} \leftarrow ((\mat{U}^*\mat{R})\mat{C}^{-1})\mat{C}^{-*} \cdot \diag \mleft( (\mat{H}^{-1})_{ii}^{-1/2} : i=1,2,\ldots,s \mright)$
    \For{$j = 1,\ldots,s$}
        \State $(\mat{W}_j,\sim) \leftarrow \mathtt{eig}(\mat{\Lambda} - \mat{T}(:,j)\mat{T}(:,j)^*)$ \Comment{Use $\order(s^2)$ algorithm}
        \State $\mat{F}_j \leftarrow \mat{W}_j(:,1:n_{\rm cen}) / \norm{\mat{W}_j(:,1:n_{\rm cen})^*\mat{G}\mat{W}_j(:,1:n_{\rm cen})}_{\rm F}^{1/2}$
    \EndFor
    \State $\mat{M} \leftarrow s^{-1} \sum_{j=1}^s \mat{F}_j^{\vphantom{*}}\mat{F}_j^*$
    \State $\mat{E} \leftarrow \mat{G}\mat{M}\mat{G}$
    \State $\Jack \leftarrow \mleft( s \tr(\mat{G}\mat{E}) - 2\sum_{j=1}^s  \tr(\mat{F}_j^*\mat{E}\mat{F}_j^{\vphantom{*}} + \sum_{j=1}^s \norm{\mat{F}_j^*\mat{G}\mat{F}_j^{\vphantom{*}}}_{\rm F}^2\mright)^{1/2}$
  \end{algorithmic}
\end{algorithm}

As this example demonstrates, devising efficient algorithms for matrix jackknife variance estimation can require some symbollic manipulations, the process is entirely \emph{systematic}.
First, use the update formula \cref{eq:nys_update} and, if necessary, use $\order(s^2)$ algorithms to solve the sequence of diagonal plus rank-one eigenproblems \cref{eq:rank_one_modification}.
After that, the hard work is done and an efficient procedure can often be derived by symbollic manipulation.

\subsection{Singular values and vectors from the randomized singular value decomposition} \label{sec:rsvd_extra}
Using a modification of the techniques from \cref{sec:spectral_nystrom}, we can derive efficient algorithms for jackknife variance estimation for various quantities derived from the randomized SVD such as singular values, singular vectors, projectors onto singular subspaces, and truncations of the matrix to smaller rank $r < s$.
We briefly illustrate by sketching an $\order(s^3)$ algorithm for estimating the variance of the largest singular value reported by the randomized SVD.

By the update formula \cref{eq:rsvd_update}, the replicates take the form
\begin{equation*}
    S^{(j)} = \sigma_{\rm max}\mleft( \mat{U}(\Id - \mat{W}^*\vec{t}_j^{\vphantom{*}}\vec{t}_j^*\mat{W})\mat{\Sigma}\mat{V}^* \mright) = \sigma_{\rm max}\mleft( (\Id - \mat{W}^*\vec{t}_j^{\vphantom{*}}\vec{t}_j^*\mat{W})\mat{\Sigma} \mright) \quad \text{for }j =1,\ldots,s.
\end{equation*}
Using the algorithm of \cite[\S5]{GE95}, an SVD of the rank-one modified diagonal matrix
\begin{equation} \label{eq:rank_one_rsvd}
    (\Id - \mat{W}^*\vec{t}_j^{\vphantom{*}}\vec{t}_j^*\mat{W})\mat{\Sigma} = \mat{U}_j^{\vphantom{*}}\mat{\Sigma}_j^{\vphantom{*}}\mat{V}_j^*
\end{equation}
can be computed in $\order(s^2)$ operations, which immediately yields a jackknife variance estimate of the largest singular value in $\order(s^3)$ operations.
Efficient jackknife variance procedures for singular vectors, projectors onto singular subspaces, etc.\ can be developed along similar lines.

\section{Failure of the Efron bootstrap variance estimate for the randomized SVD} \label{sec:bootstrap-failure}
The matrix jackknife produces an estimate of the standard deviation that is positively biased (\cref{thm:jackknife})---by up to an order of magnitude in our experiments.
It is natural to wonder: Could the bootstrap do better?
In this section, we demonstrate that a staightforward application of the bootstrap variance estimate \cite[\S5.1]{Efr82} to the Halko--Martinsson--Tropp randomized SVD can go badly wrong.

To focus the discussion, consider the task of estimating the standard deviation of the largest singular value $S \coloneqq \sigma_{\rm max}(\mat{X})$ estimated from the randomized SVD approximation $\mat{X} = \mat{X}(\mat{\Omega})$ (with no subspace iteration, $q=0$).
A straightforward application of the classical bootstrap variance estimate of Efron \cite[\S5.1]{Efr82} would proceed as follows:
Fix a number of bootstrap iterations $B \ge 1$.
For $j = 1,\ldots,B$, do the following:
\begin{enumerate}
    \item Sample $\vec{\omega}_{1\star},\ldots,\vec{\omega}_{s\star}$ uniformly and \emph{with replacement} from $\{ \vec{\omega}_1,\ldots, \vec{\omega}_s \}$.
    \item Compute the randomized SVD $\mat{X}^{(j)} = \mat{X}(\begin{bmatrix} \vec{\omega}_{1\star} & \cdots & \vec{\omega}_{s\star} \end{bmatrix})$ using the bootstrap sample and define the \emph{bootstrap replicate} $S^{(j)} \coloneqq \sigma_{\rm max}(\mat{X}^{(j)})$.
\end{enumerate}
Finally, we define the Efron bootstrap variance estimate as
\begin{equation*}
    \mathrm{Boot}^2(S) \coloneqq \frac{1}{B-1} \sum_{j=1}^B \left( S^{(j)} - S^{(\cdot)} \right)^2 \quad \text{where} \quad S^{(\cdot)} = \frac{1}{B} \sum_{j=1}^B S^{(j)}.
\end{equation*}

The Efron bootstrap variance estimate has serious deficiencies for the randomized SVD algorithm. 
The issue is that the output of randomized SVD does not depend on the number of times a column $\vec{\omega}_{j\star}$ occurs repeated in the input matrix.
On average, a bootstrap sample $\{\vec{\omega}_{1\star},\ldots,\vec{\omega}_{s\star}\}$ contains only $(1-1/\mathrm{e})s$ unique vectors and thus $\mat{X}^{(j)}$ will typically be a rank-$(1-1/\mathrm{e})s$ approximation.
Since the matrix bootstrap replicates $\mat{X}^{(j)}$ have much smaller rank than $\mat{X}$, their variability often grossly overestimates the variance of $\mat{X}$.

To demonstrate this, consider the diagonal matrix
\begin{equation*}
    \mat{A} = \diag \left( 1,0.99,0.98,\ldots,0.26,\frac{0.25}{1^2}, \frac{0.25}{2^2}, \frac{0.25}{3^2},\ldots,\frac{0.25}{925^2}\right) \in \real^{1000\times 1000}.
\end{equation*}
The first 75 singular values of $\mat{A}$ decay quite slowly, after which they decay at a polynomial rate.
We consider the estimate of the largest singular value $S = \sigma_{\rm max}(\mat{X})$ produced by the randomized SVD with approximation rank $s = 100$.
We use $B = 1000$ bootstrap iterations and use $1000$ trials to estimate the true standard deviation $\sqrt{\Var(S)}$ of $S$.
For this example, the jackknife and bootstrap estimates and true standard deviation were as follows:
\begin{equation*}
    \sqrt{\Var(S))} = 8.2\times 10^{-8}, \quad \Jack(S) = 3.2\times 10^{-7}, \quad \mathrm{Boot}(S) = 5.0\times 10^{-3}.
\end{equation*}
The jackknife overestimates the standard deviation by a factor of 3.9; the bootstrap overestimates the standard deviation by a factor of 60,000.

Existing applications of bootstrap methodology to randomized matrix computations has focused on Monte Carlo-type matrix estimators, a setting in which the bootstrap is more natural.
Unfortunately, these Monte Carlo estimators have poor accuracy compared to the randomized SVD.
Based on the failure mode presented here, we believe it is unlikely that a generally effective and computationally efficient (i.e., requiring no additional matrix--vector products with $\mat{A}$) variance estimate based on bootstrap resampling can be developed for the randomized SVD.

\section{Additional numerical experiments} \label{sec:addit-numer-exper}
In addition to the exponential decay \cref{eq:exp_decay} and noisy low-rank \cref{eq:low_rank_plus_noise} examples, we consider a third synthetic matrix example with polynomial decay:
\begin{equation}
    \mat{A} = \diag( \underbrace{1,\ldots,1}_{R \textrm{ times}},2^{-p},3^{-p},\ldots,(d-R+1)^{-p}), \tag{PolyDecay} \label{eq:poly_decay}
\end{equation}
where $p \in \real$ is a real parameter.
We again set $d \coloneqq 1000$.

We also consider more application matrices:
\begin{itemize}
    \item \textbf{Kernel.} The QM9 kernel matrix from \cref{sec:loo_intro}.
    \item \textbf{Inverse-Poisson.} The inverse of the tridiagonal Poisson matrix corresponding to a discretization of the ODE boundary value problem $-u'' = f$, $u(0) = u(1) = 0$ with a three-point centered difference scheme.
    \item \textbf{Network.} The matrix exponential $\mat{A} = \exp(\mat{M})$ of the adjacency matrix $\mat{M}$ of the \texttt{yeast} network \cite{BZC+03} from the SuiteSparse matrix collection \cite{DH11}.
    The diagonal entries of this matrix are known as \emph{subgraph centralities}, which can be estimated with the help of low-rank approximations \cite[\S4.4]{ETW24}.
\end{itemize}

\begin{figure}[t]
  \centering
  \begin{subfigure}{0.32\textwidth}
      \includegraphics[width=0.99\textwidth]{figures/loo_randsvd_expfast.png}
      \caption*{\ref{eq:exp_decay}$(q \!=\! 0.1,R\!=\!5)$}
  \end{subfigure}
  \begin{subfigure}{0.32\textwidth}
      \includegraphics[width=0.99\textwidth]{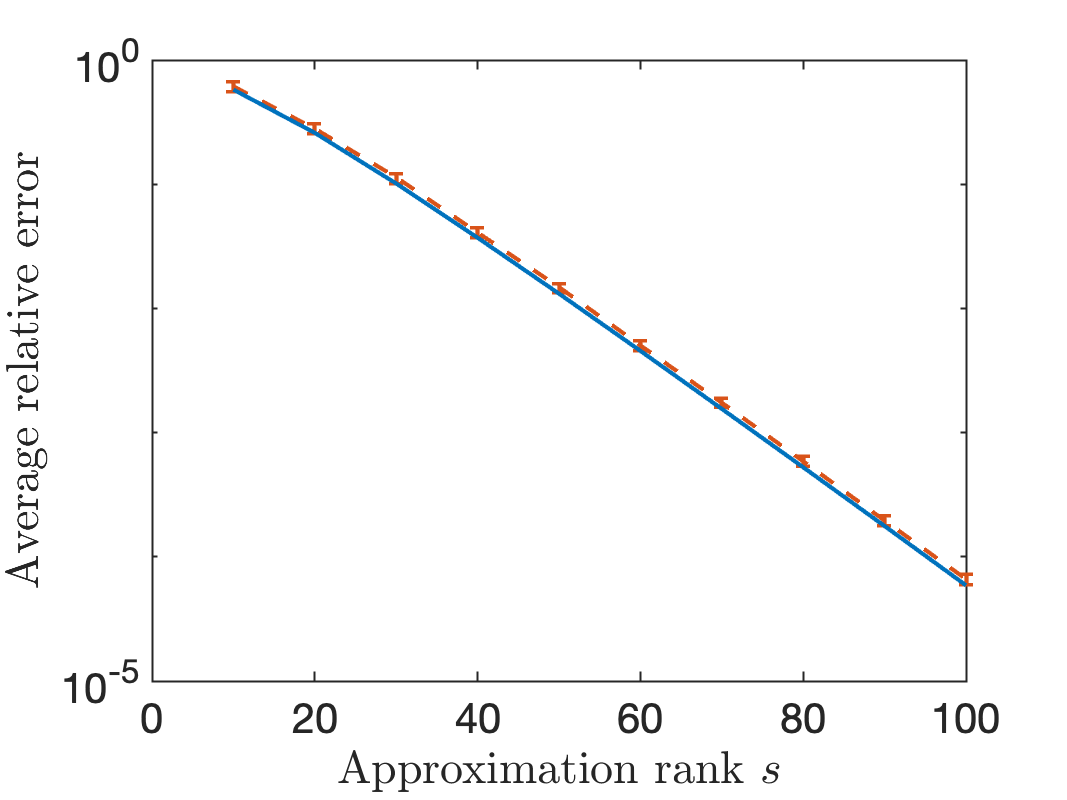}
      \caption*{\ref{eq:poly_decay}$(p \!=\! 2,R\!=\!5)$}
  \end{subfigure}
  \begin{subfigure}{0.32\textwidth}
      \includegraphics[width=0.99\textwidth]{figures/loo_randsvd_nlrfast.png}
      \caption*{\ref{eq:low_rank_plus_noise}$(\xi \!=\! 10^{-4},R\!=\!5)$}
  \end{subfigure}
  
  \begin{subfigure}{0.32\textwidth}
      \includegraphics[width=0.99\textwidth]{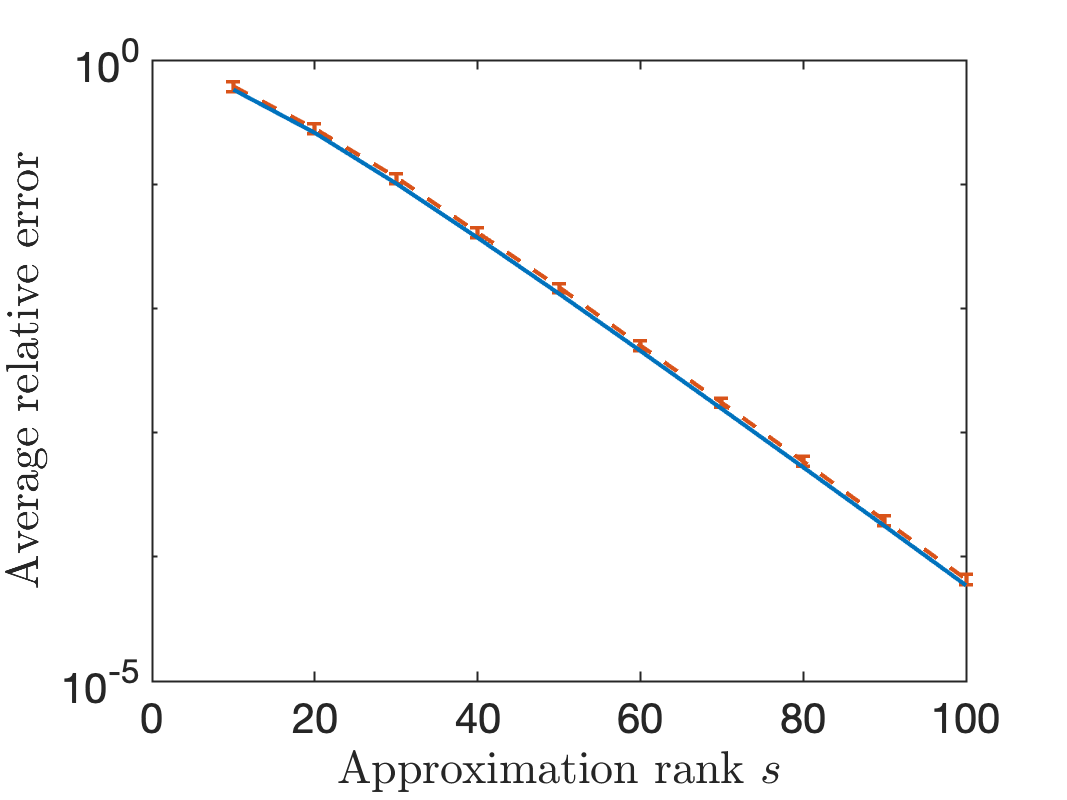}
      \caption*{\ref{eq:exp_decay}$(q \!=\! 0.05,R\!=\!5)$}
  \end{subfigure}
  \begin{subfigure}{0.32\textwidth}
      \includegraphics[width=0.99\textwidth]{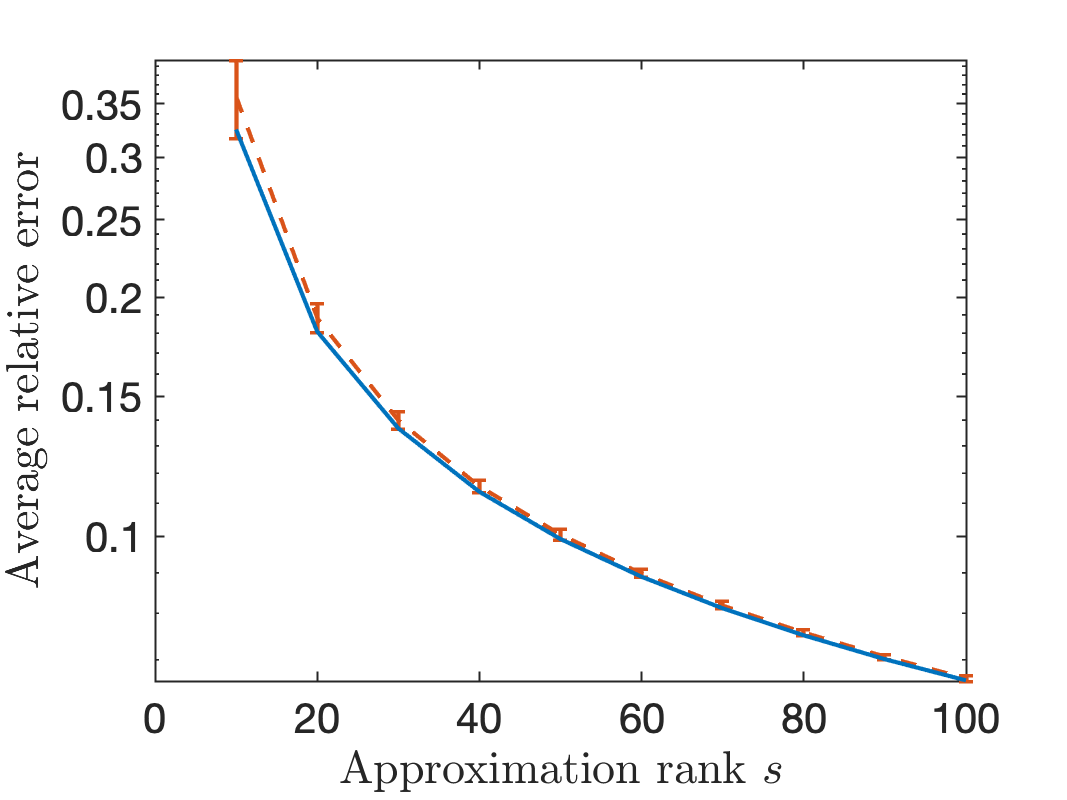}
      \caption*{\ref{eq:poly_decay}$(p \!=\! 1,R\!=\!5)$}
  \end{subfigure}
  \begin{subfigure}{0.32\textwidth}
      \includegraphics[width=0.99\textwidth]{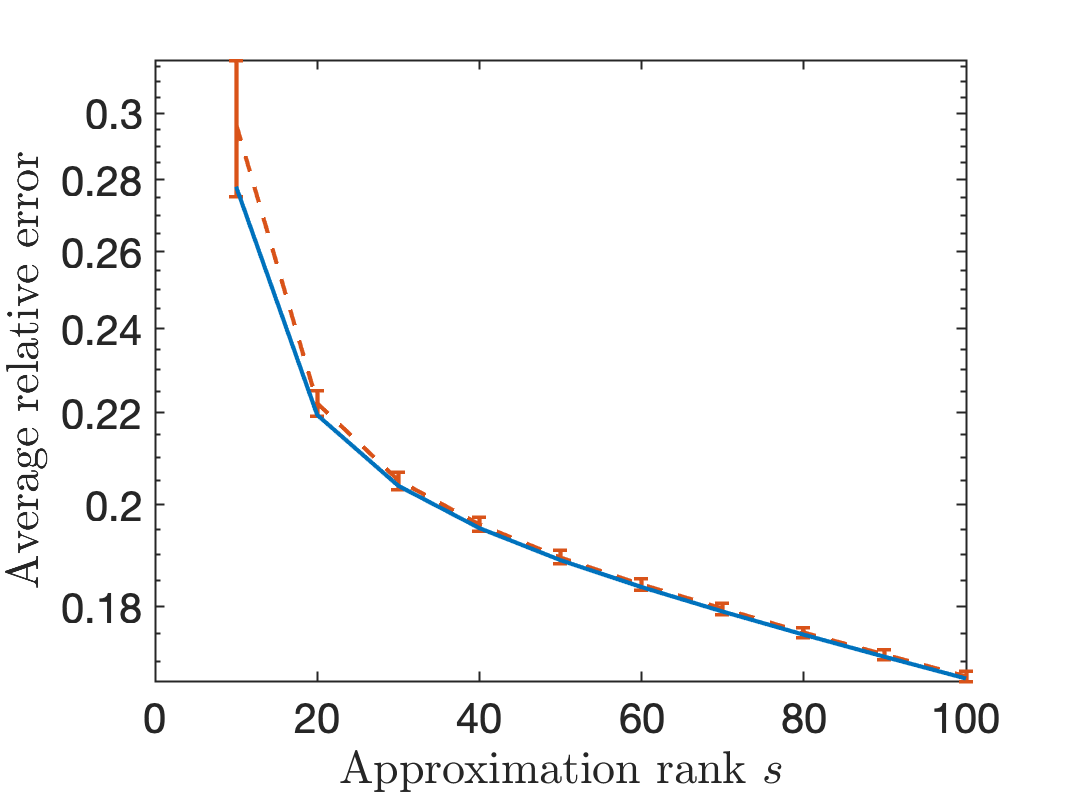}
      \caption*{\ref{eq:low_rank_plus_noise}$(\xi \!=\! 10^{-2},R\!=\!5)$}
  \end{subfigure}
  
  \begin{subfigure}{0.32\textwidth}
      \includegraphics[width=0.99\textwidth]{figures/loo_randsvd_vel.png}
      \caption*{Velocity}
  \end{subfigure}
  \begin{subfigure}{0.32\textwidth}
      \includegraphics[width=0.99\textwidth]{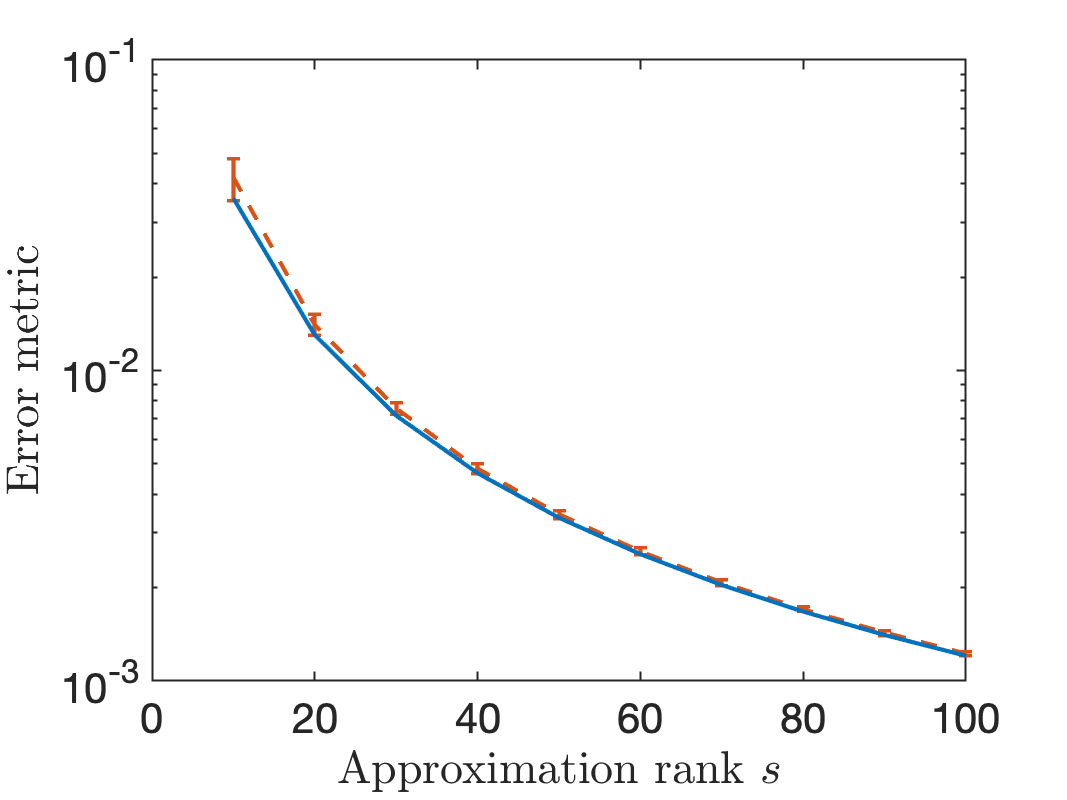}
      \caption*{Inverse-Poisson}
  \end{subfigure}
  \begin{subfigure}{0.32\textwidth}
      \includegraphics[width=0.99\textwidth]{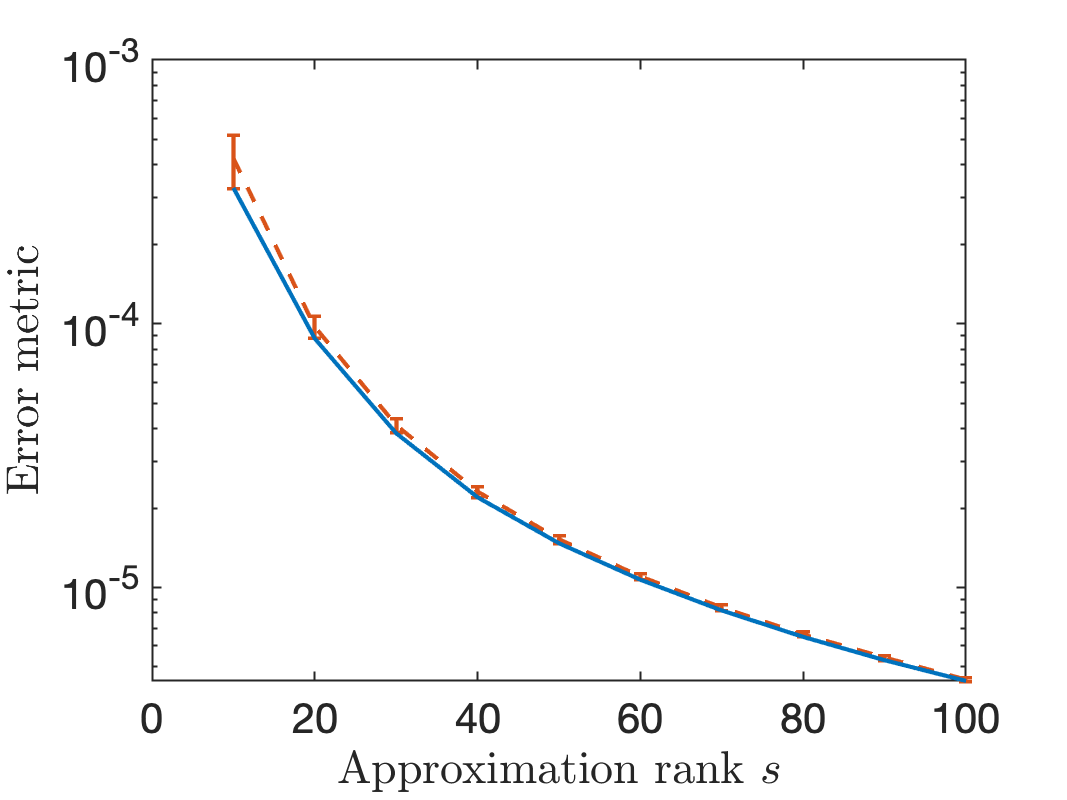}
      \caption*{Network}
  \end{subfigure}
  
  \mycaption{Leave-one-out error estimator for randomized SVD}{Error and error estimate for randomized SVD ($q=0$) approximation. Includes more test matrices than \cref{fig:rsvd_loo}.}
  \label{fig:rsvd_loo_extra}
\end{figure}

\begin{figure}[t]
  \centering
  \begin{subfigure}{0.32\textwidth}
      \includegraphics[width=0.99\textwidth]{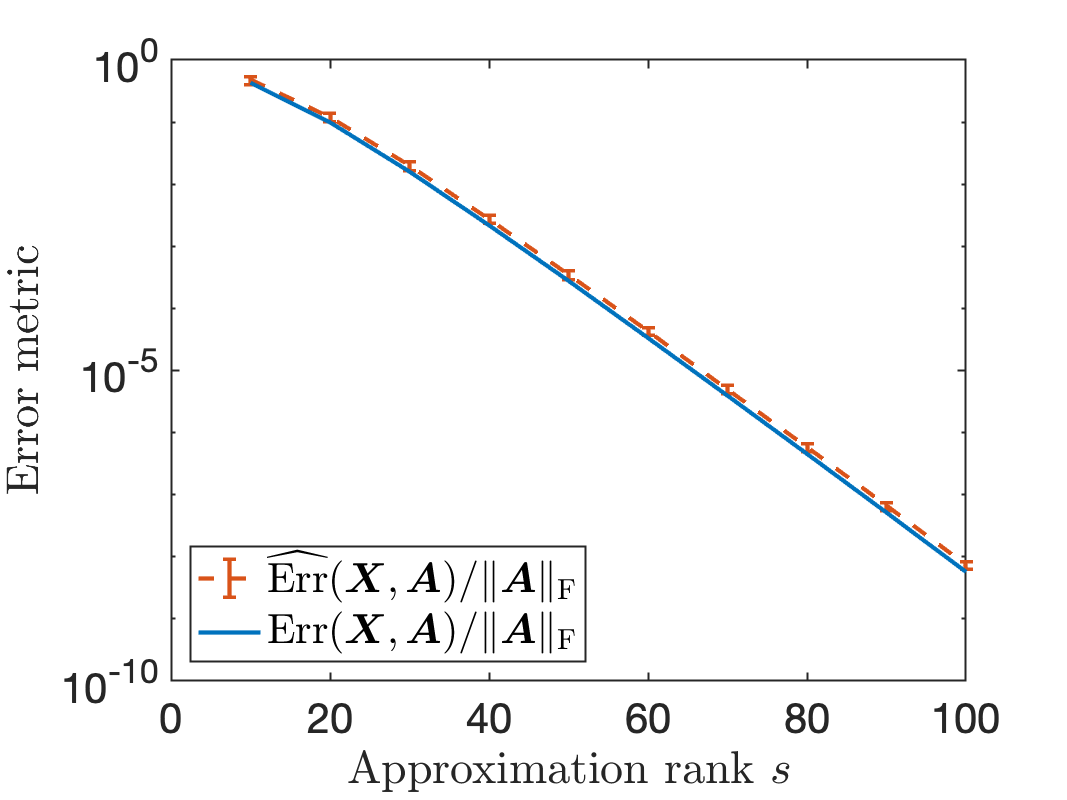}
      \caption*{\ref{eq:exp_decay}$(q \!=\! 0.1,R\!=\!5)$}
  \end{subfigure}
  \begin{subfigure}{0.32\textwidth}
      \includegraphics[width=0.99\textwidth]{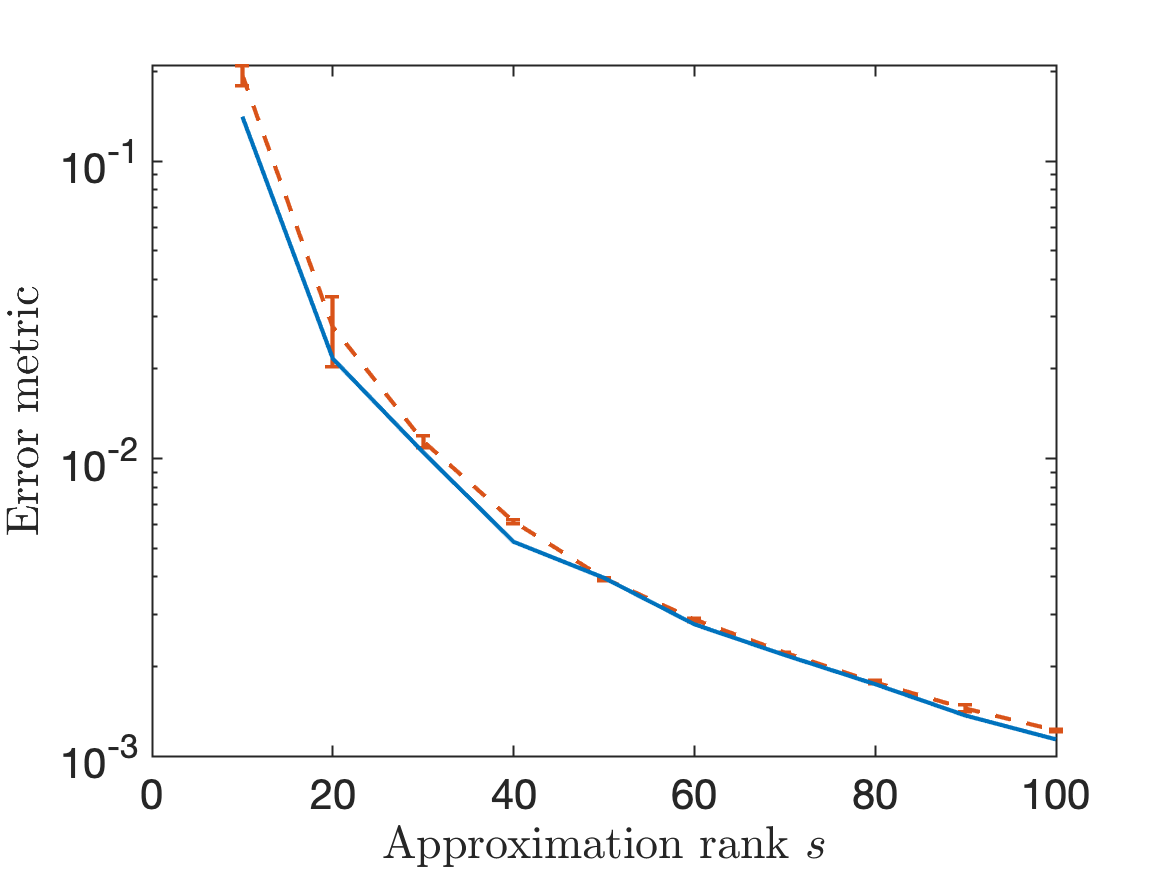}
      \caption*{\ref{eq:poly_decay}$(p \!=\! 2,R\!=\!5)$}
  \end{subfigure}
  \begin{subfigure}{0.32\textwidth}
      \includegraphics[width=0.99\textwidth]{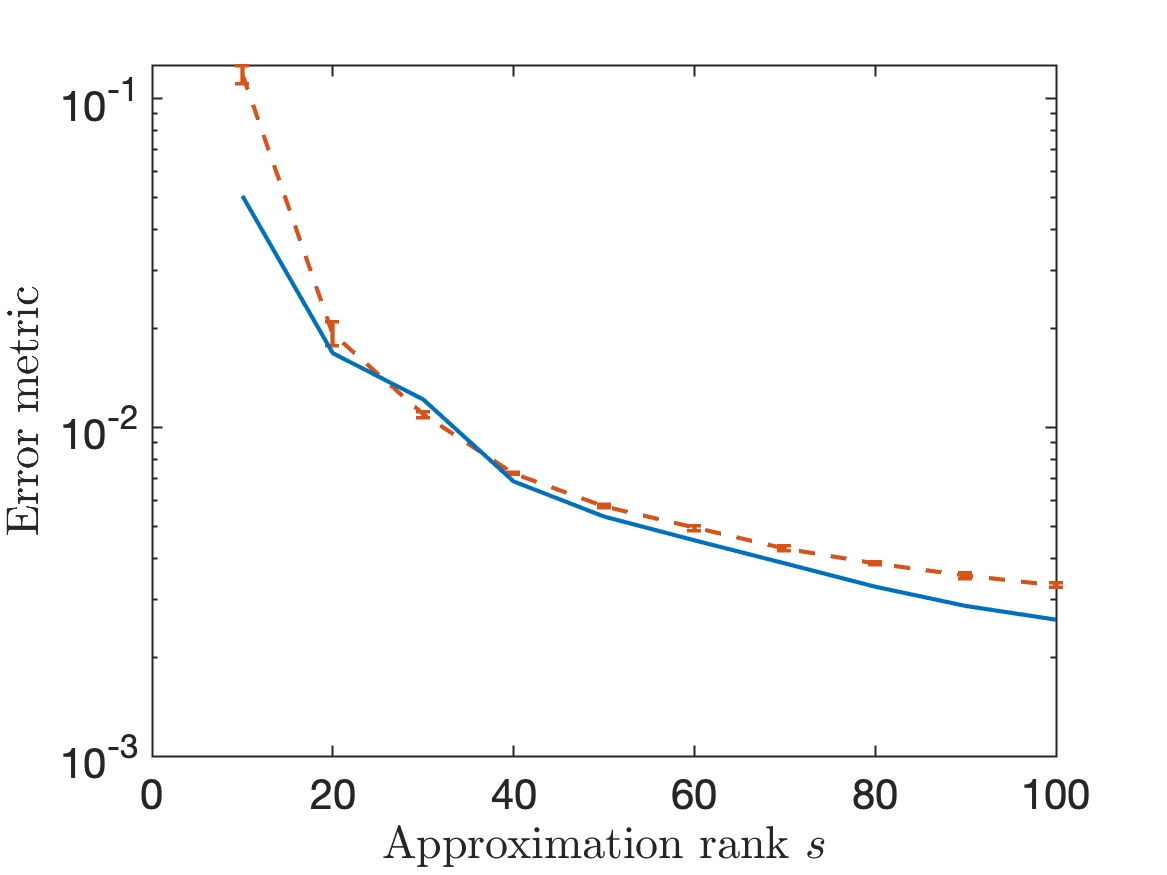}
      \caption*{\ref{eq:low_rank_plus_noise}$(\xi \!=\! 10^{-4},R\!=\!5)$}
  \end{subfigure}
  
  \begin{subfigure}{0.32\textwidth}
      \includegraphics[width=0.99\textwidth]{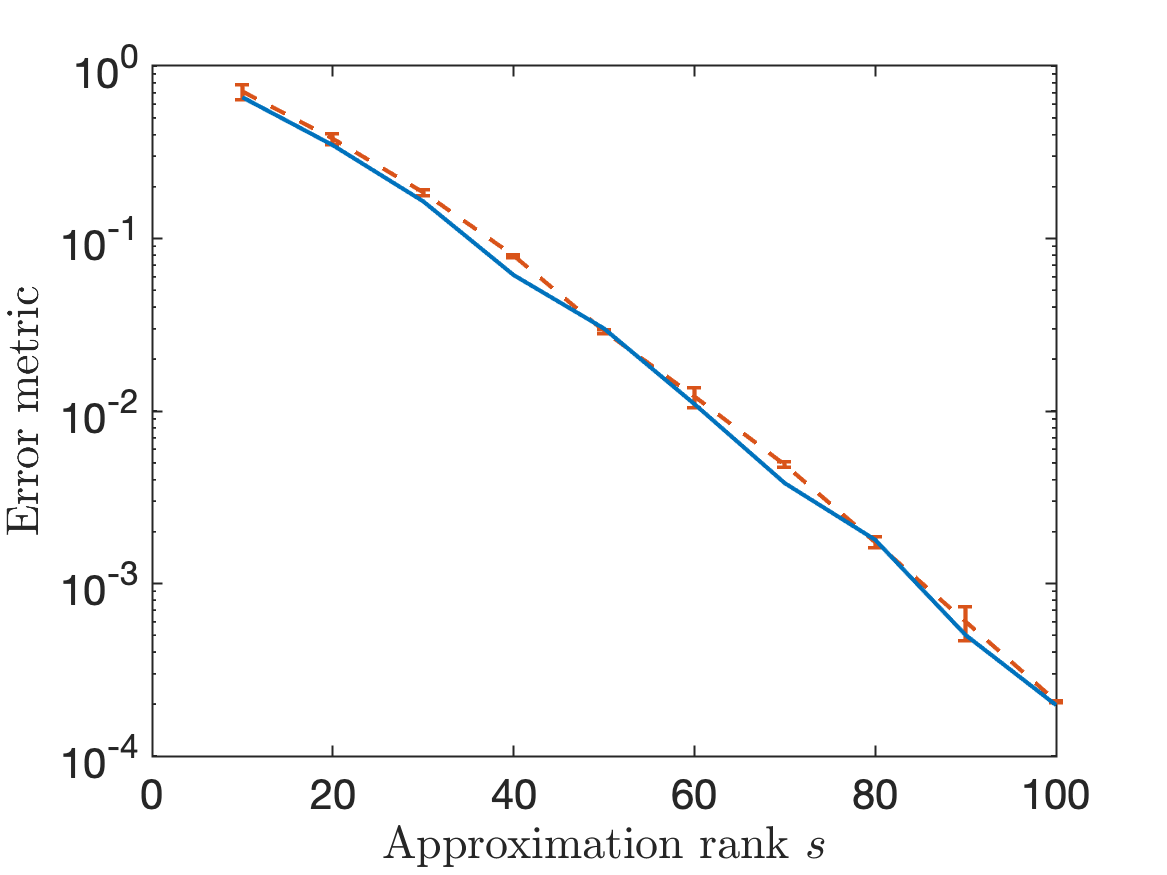}
      \caption*{\ref{eq:exp_decay}$(q \!=\! 0.05,R\!=\!5)$}
  \end{subfigure}
  \begin{subfigure}{0.32\textwidth}
      \includegraphics[width=0.99\textwidth]{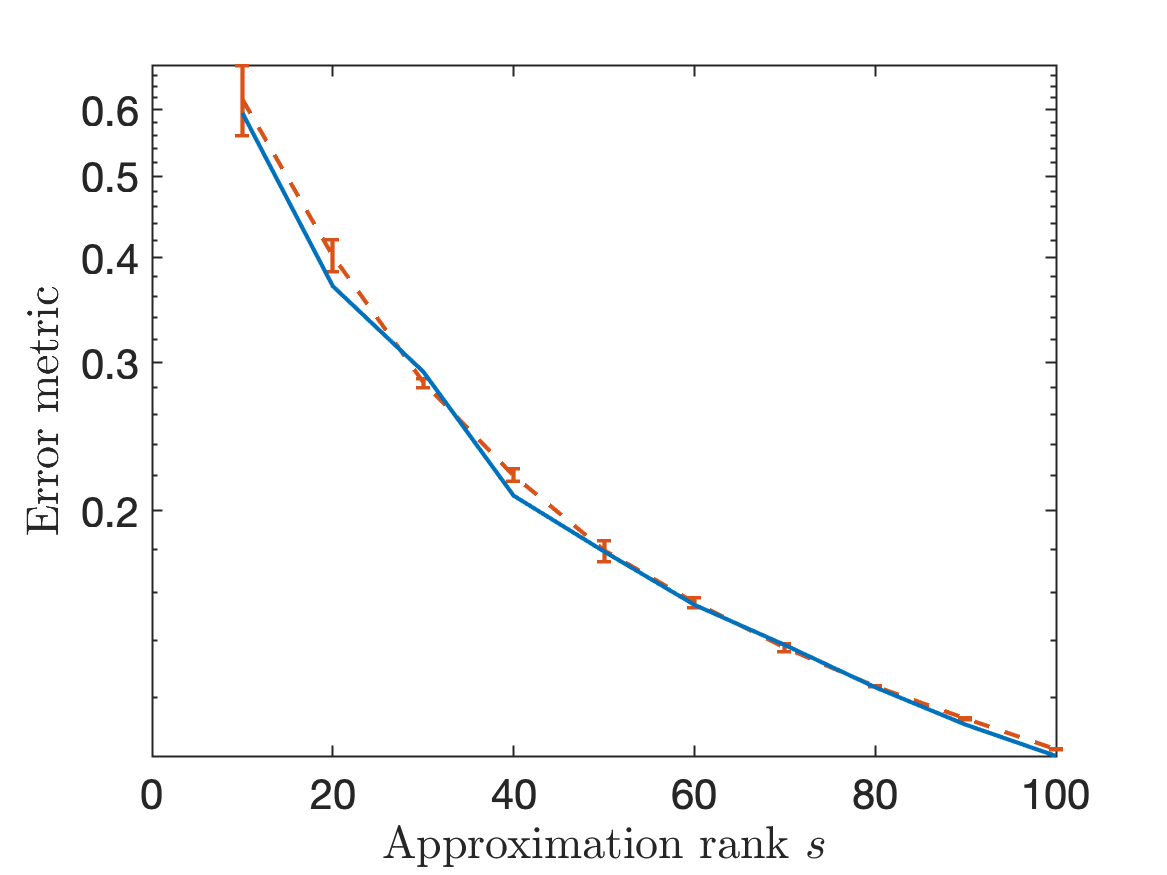}
      \caption*{\ref{eq:poly_decay}$(p \!=\! 1,R\!=\!5)$}
  \end{subfigure}
  \begin{subfigure}{0.32\textwidth}
      \includegraphics[width=0.99\textwidth]{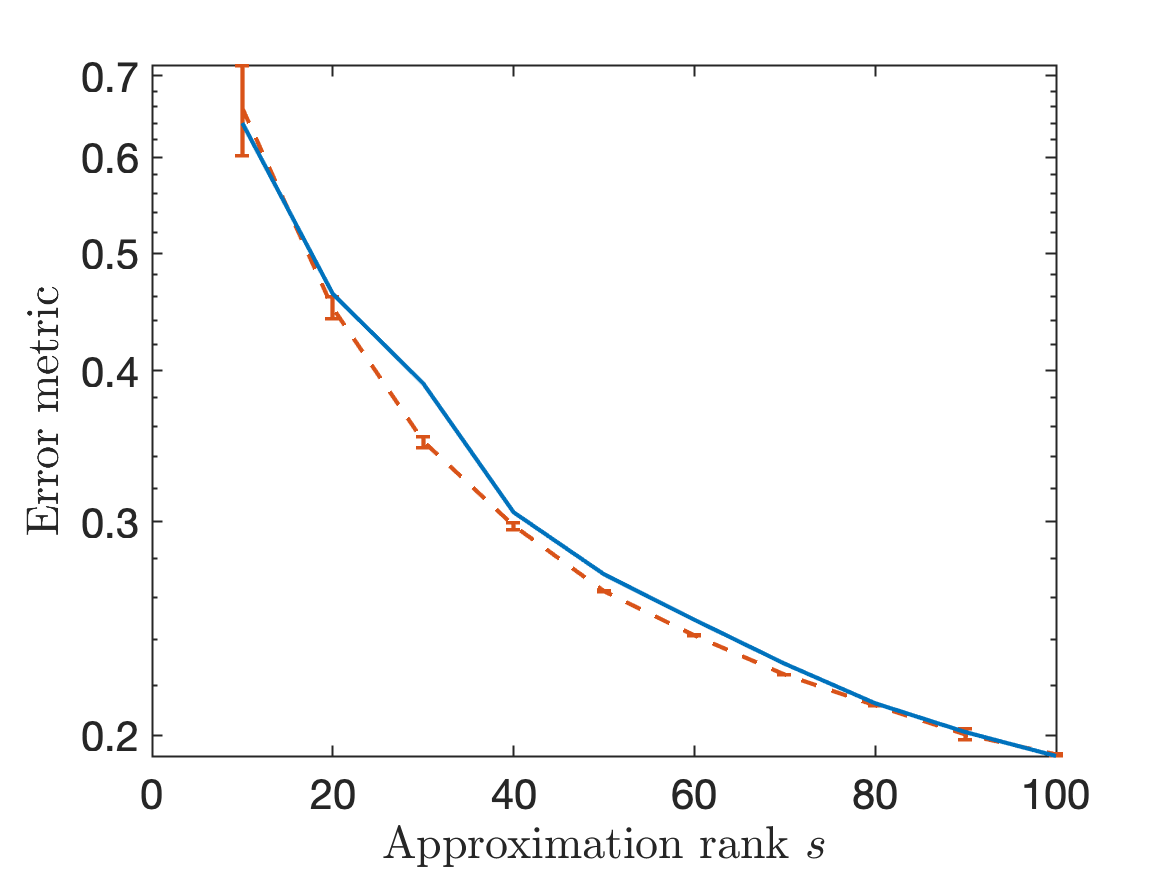}
      \caption*{\ref{eq:low_rank_plus_noise}$(\xi \!=\! 10^{-2},R\!=\!5)$}
  \end{subfigure}
  
  \begin{subfigure}{0.32\textwidth}
      \includegraphics[width=0.99\textwidth]{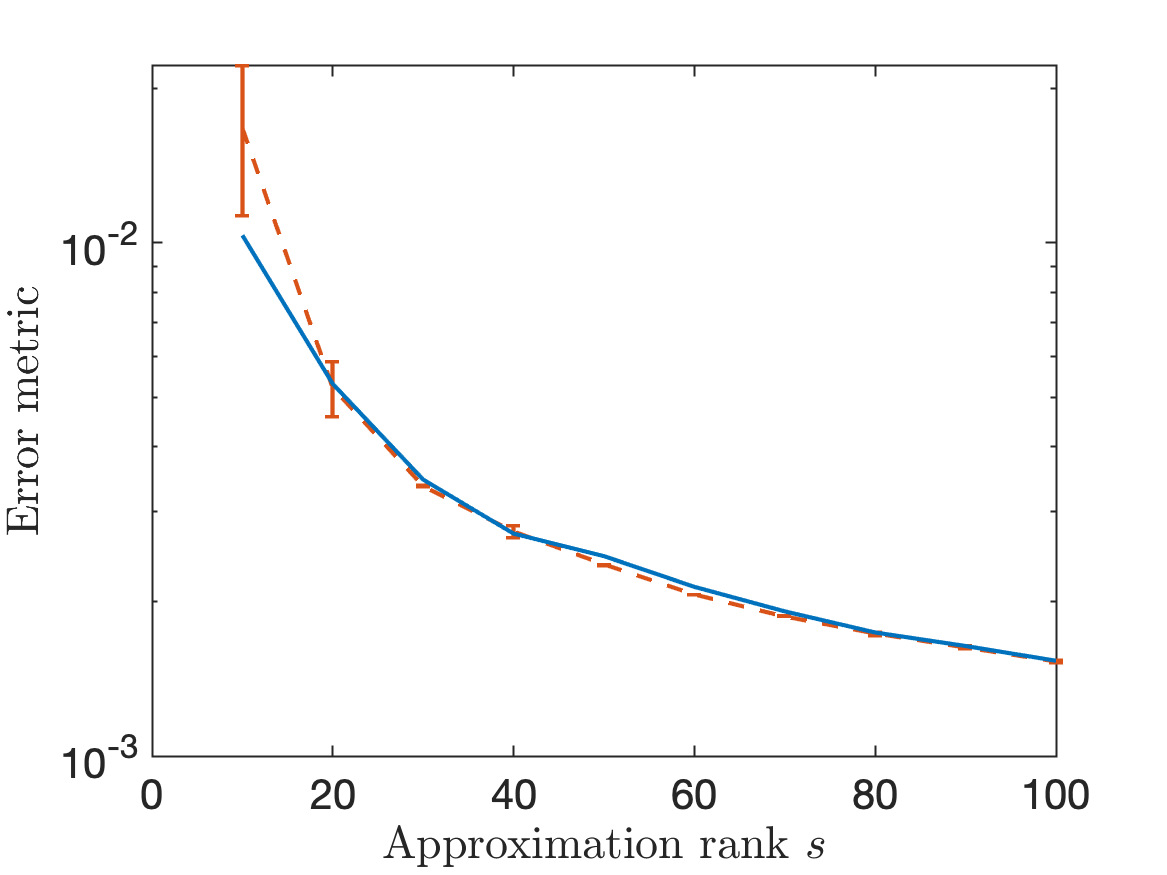}
      \caption*{Kernel}
  \end{subfigure}
  \begin{subfigure}{0.32\textwidth}
      \includegraphics[width=0.99\textwidth]{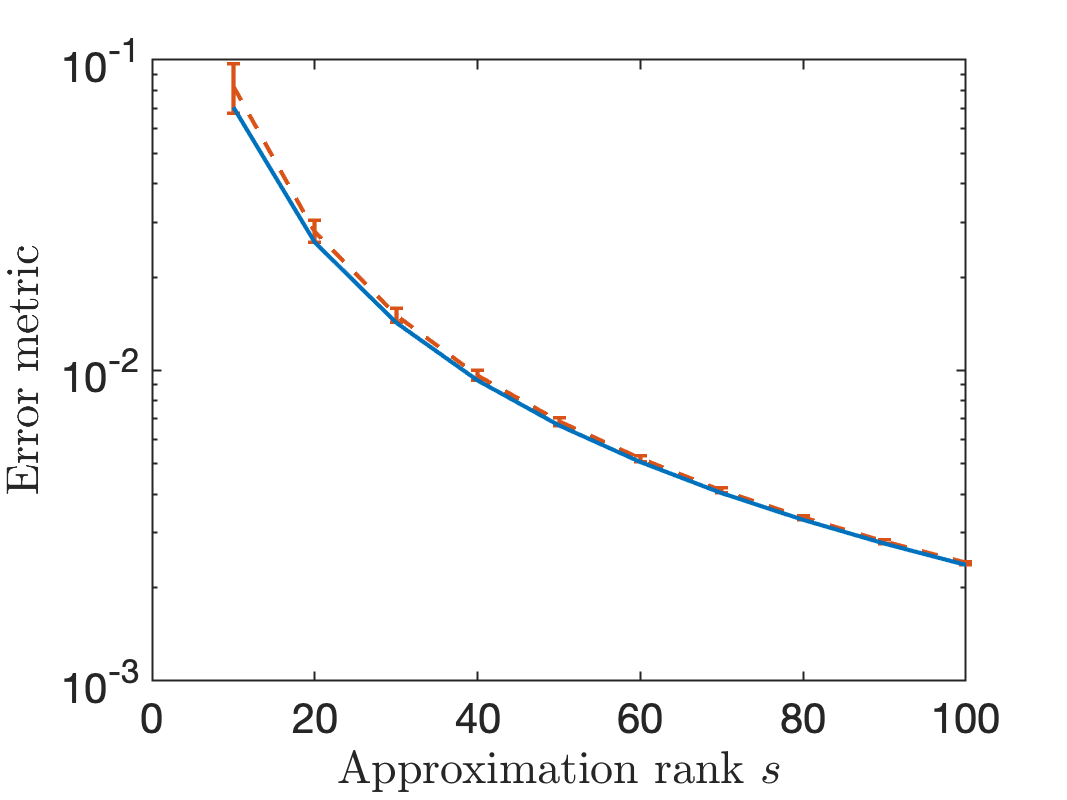}
      \caption*{Inverse-Poisson}
  \end{subfigure}
  \begin{subfigure}{0.32\textwidth}
      \includegraphics[width=0.99\textwidth]{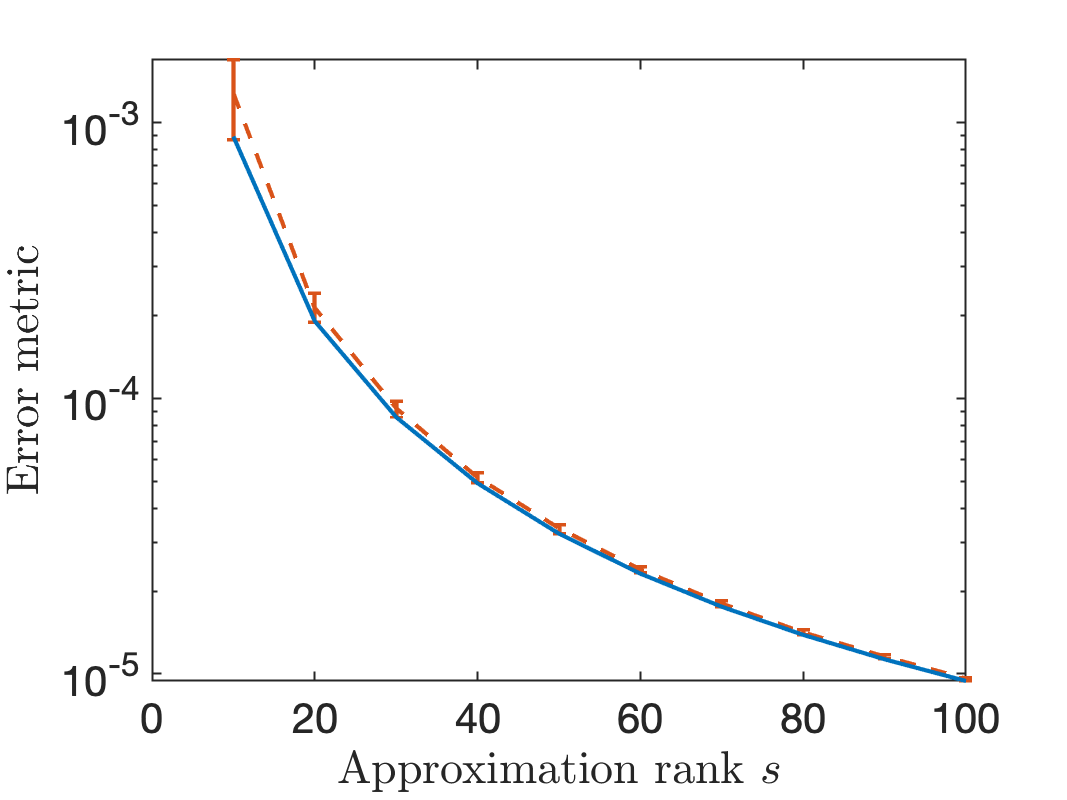}
      \caption*{Network}
  \end{subfigure}
  
  \mycaption{Leave-one-out error estimator for randomized Nystr\"om}{Error and error estimate for randomized SVD ($q=0$) approximation. Includes more test matrices than \cref{fig:rsvd_loo}.}
  \label{fig:nystrom_loo}
\end{figure}

\begin{figure}[t]
  \centering
  \begin{subfigure}{0.32\textwidth}
      \includegraphics[width=0.99\textwidth]{figures/jack_randsvd_expfast.png}
      \caption*{\ref{eq:exp_decay}$(q \!=\! 0.1,R\!=\!5)$}
  \end{subfigure}
  \begin{subfigure}{0.32\textwidth}
      \includegraphics[width=0.99\textwidth]{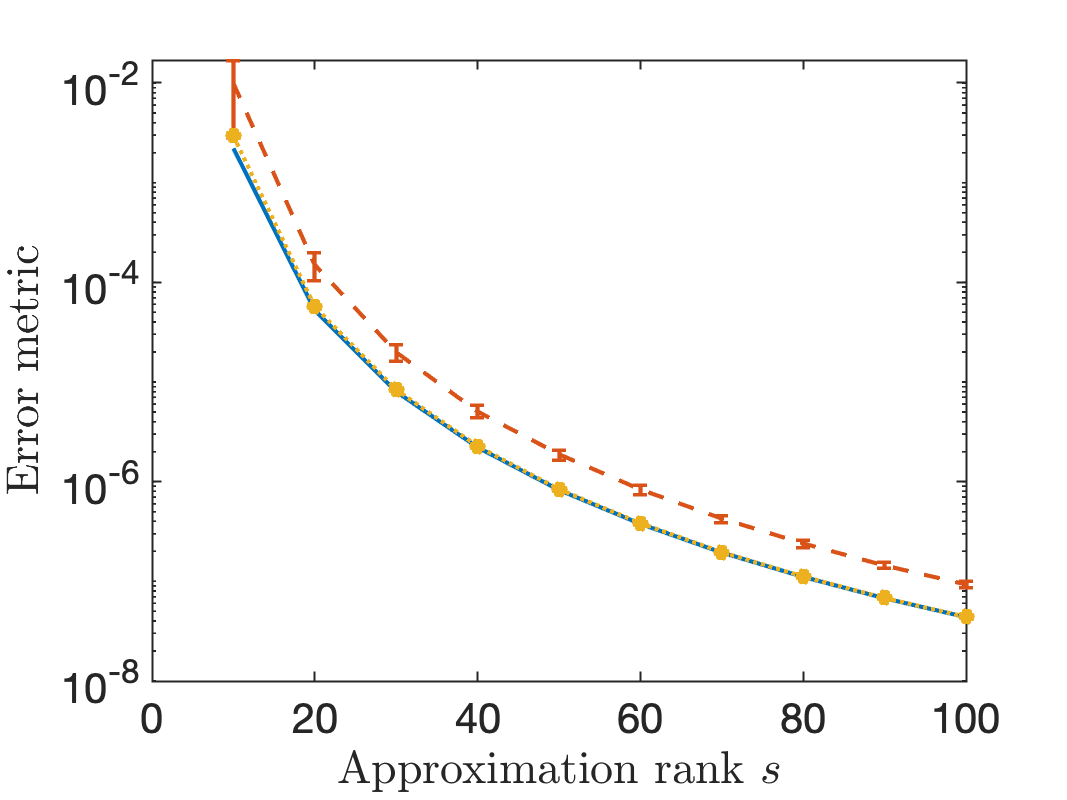}
      \caption*{\ref{eq:poly_decay}$(p \!=\! 2,R\!=\!5)$}
  \end{subfigure}
  \begin{subfigure}{0.32\textwidth}
      \includegraphics[width=0.99\textwidth]{figures/jack_randsvd_nlrfast.png}
      \caption*{\ref{eq:low_rank_plus_noise}$(\xi \!=\! 10^{-4},R\!=\!5)$}
  \end{subfigure}
  
  \begin{subfigure}{0.32\textwidth}
      \includegraphics[width=0.99\textwidth]{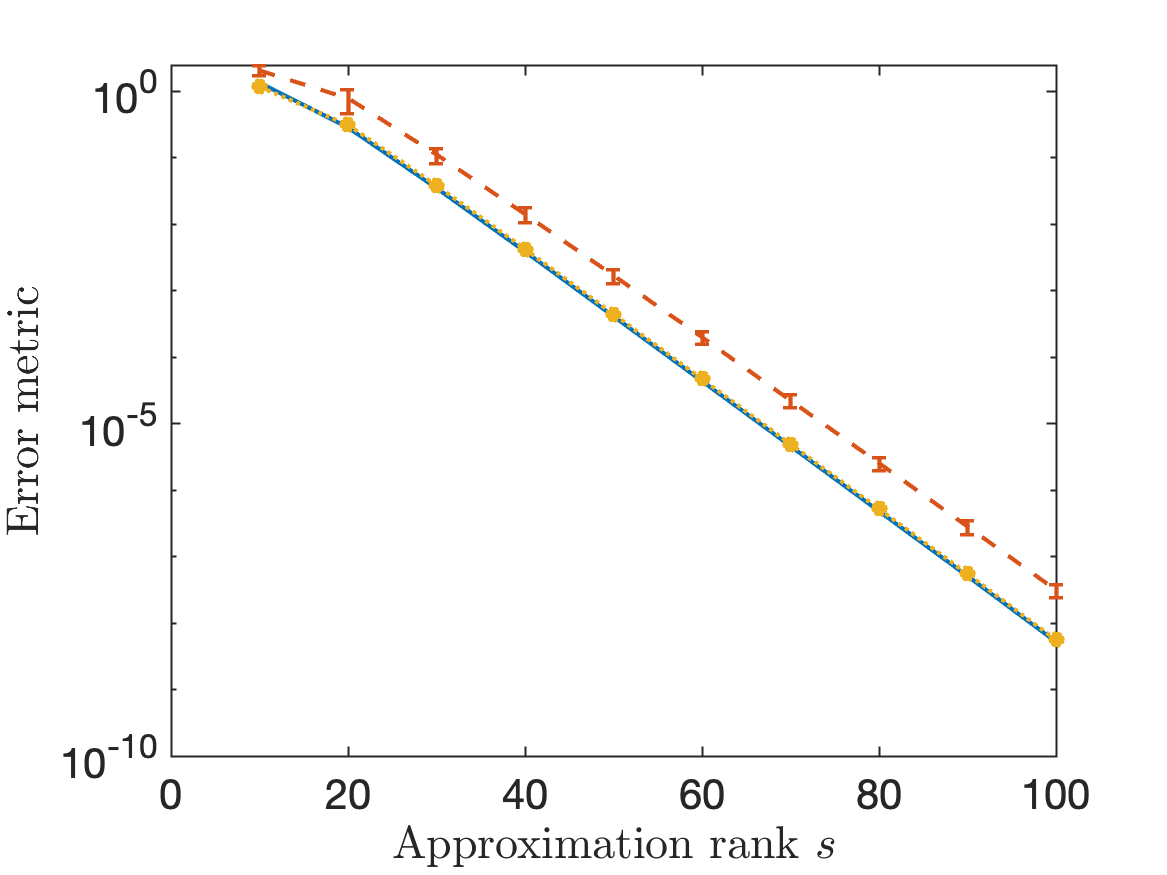}
      \caption*{\ref{eq:exp_decay}$(q \!=\! 0.05,R\!=\!5)$}
  \end{subfigure}
  \begin{subfigure}{0.32\textwidth}
      \includegraphics[width=0.99\textwidth]{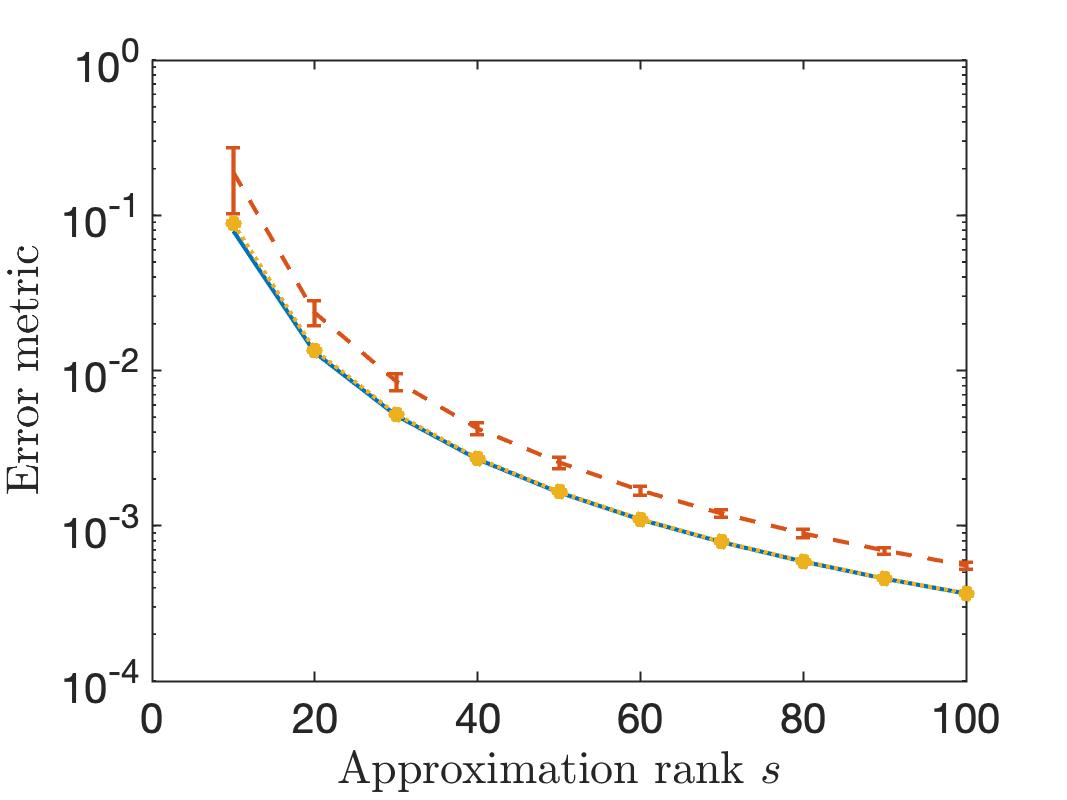}
      \caption*{\ref{eq:poly_decay}$(p \!=\! 1,R\!=\!5)$}
  \end{subfigure}
  \begin{subfigure}{0.32\textwidth}
      \includegraphics[width=0.99\textwidth]{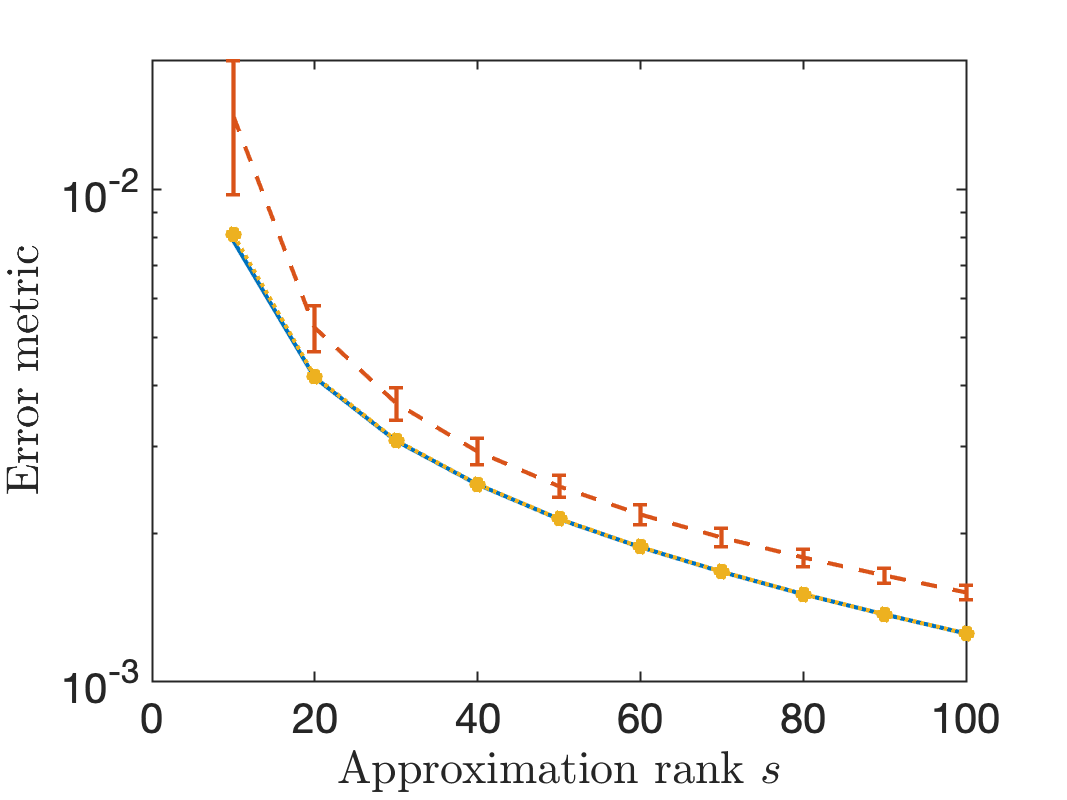}
      \caption*{\ref{eq:low_rank_plus_noise}$(\xi \!=\! 10^{-2},R\!=\!5)$}
  \end{subfigure}
  
  \begin{subfigure}{0.32\textwidth}
      \includegraphics[width=0.99\textwidth]{figures/jack_randsvd_vel.png}
      \caption*{Velocity}
  \end{subfigure}
  \begin{subfigure}{0.32\textwidth}
      \includegraphics[width=0.99\textwidth]{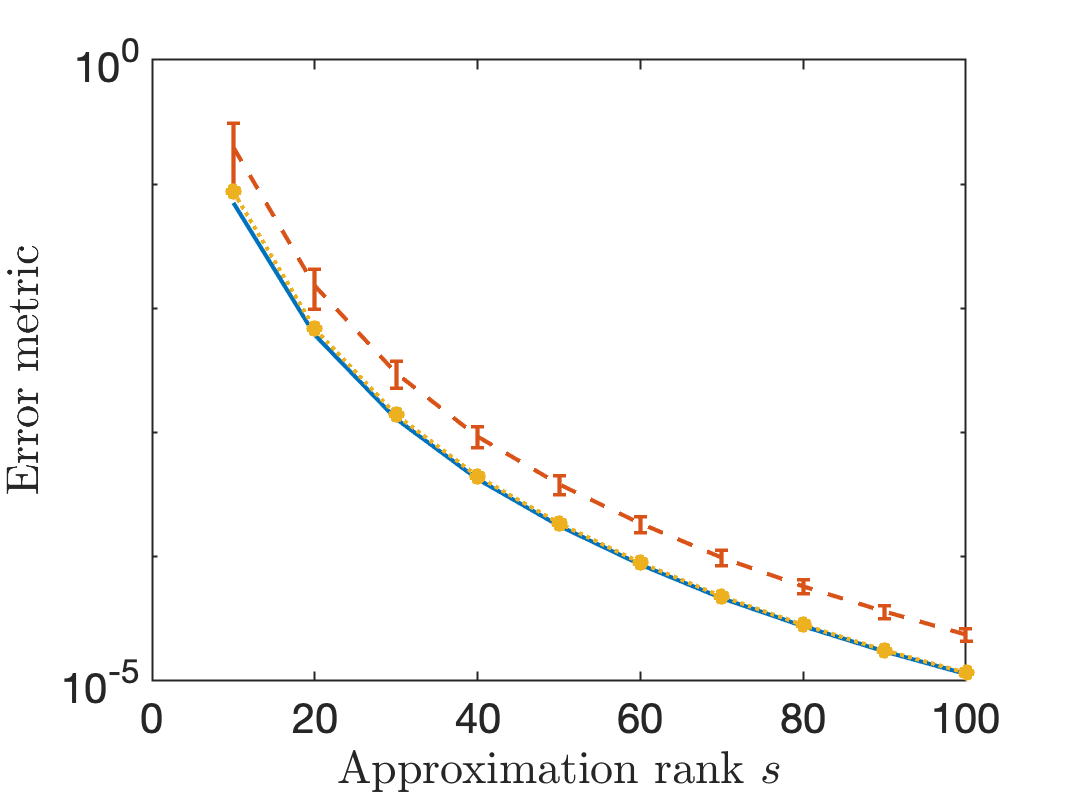}
      \caption*{Inverse-Poisson}
  \end{subfigure}
  \begin{subfigure}{0.32\textwidth}
      \includegraphics[width=0.99\textwidth]{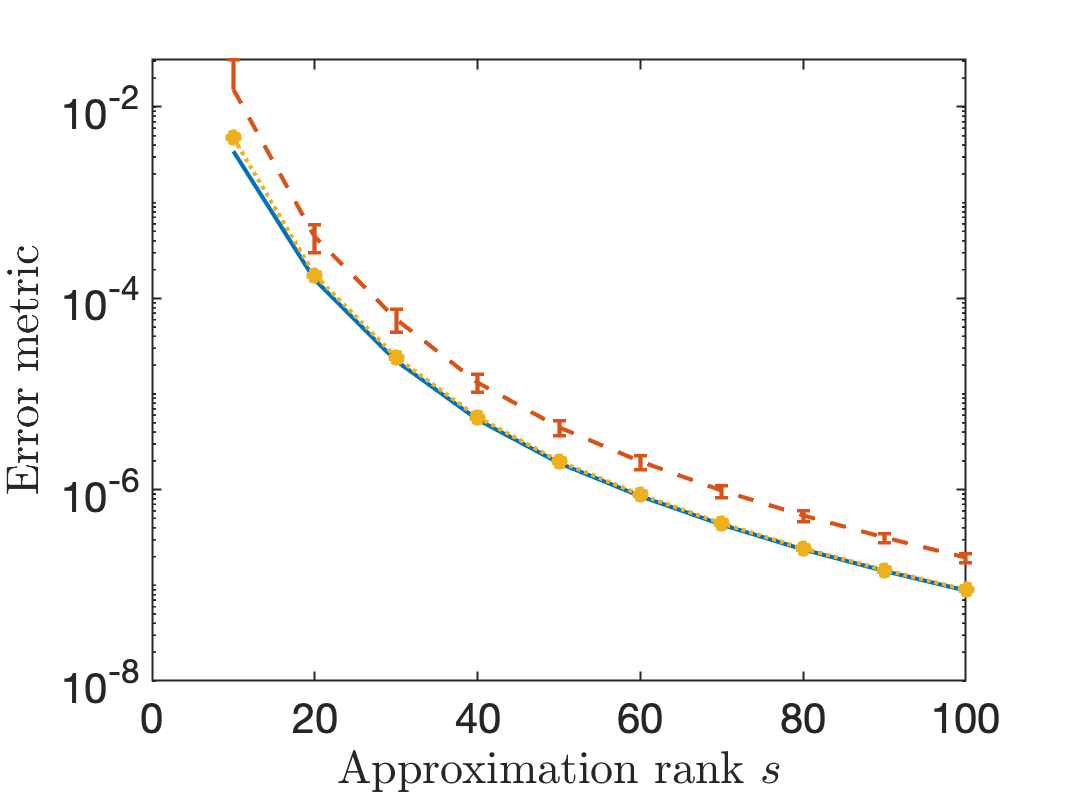}
      \caption*{Network}
  \end{subfigure}
  
  \mycaption{Matrix jackknife for projectors onto singular subspaces}{Error, standard deviation, and jackknife estimate for randomized SVD ($q=0$) approximation $\mat{X}$ \cref{eq:singular_projector} to the projector $\mat{\Pi}$ onto the span of the five dominant right singular vectors.
  Includes more test matrices than \cref{fig:rsvd_jack}.}
  \label{fig:rsvd_jack_extra}
\end{figure}

\begin{figure}[t]
  \centering
  \begin{subfigure}{0.32\textwidth}
      \includegraphics[width=0.99\textwidth]{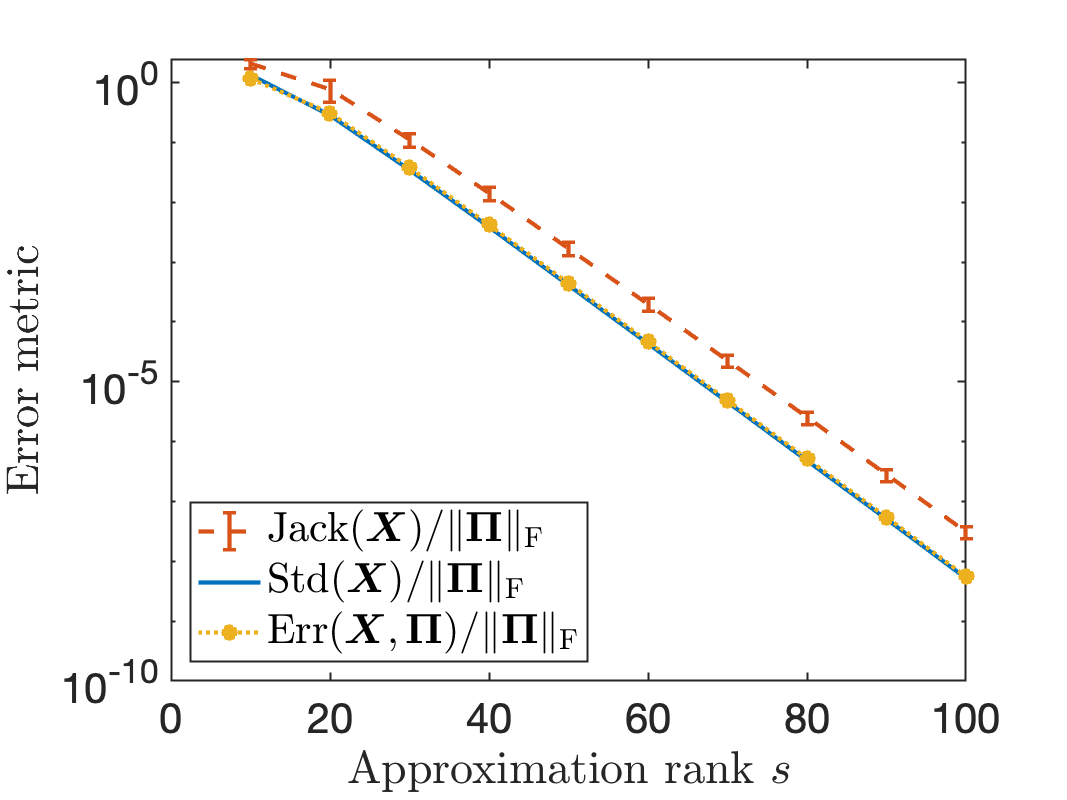}
      \caption*{\ref{eq:exp_decay}$(q \!=\! 0.1,R\!=\!5)$}
  \end{subfigure}
  \begin{subfigure}{0.32\textwidth}
      \includegraphics[width=0.99\textwidth]{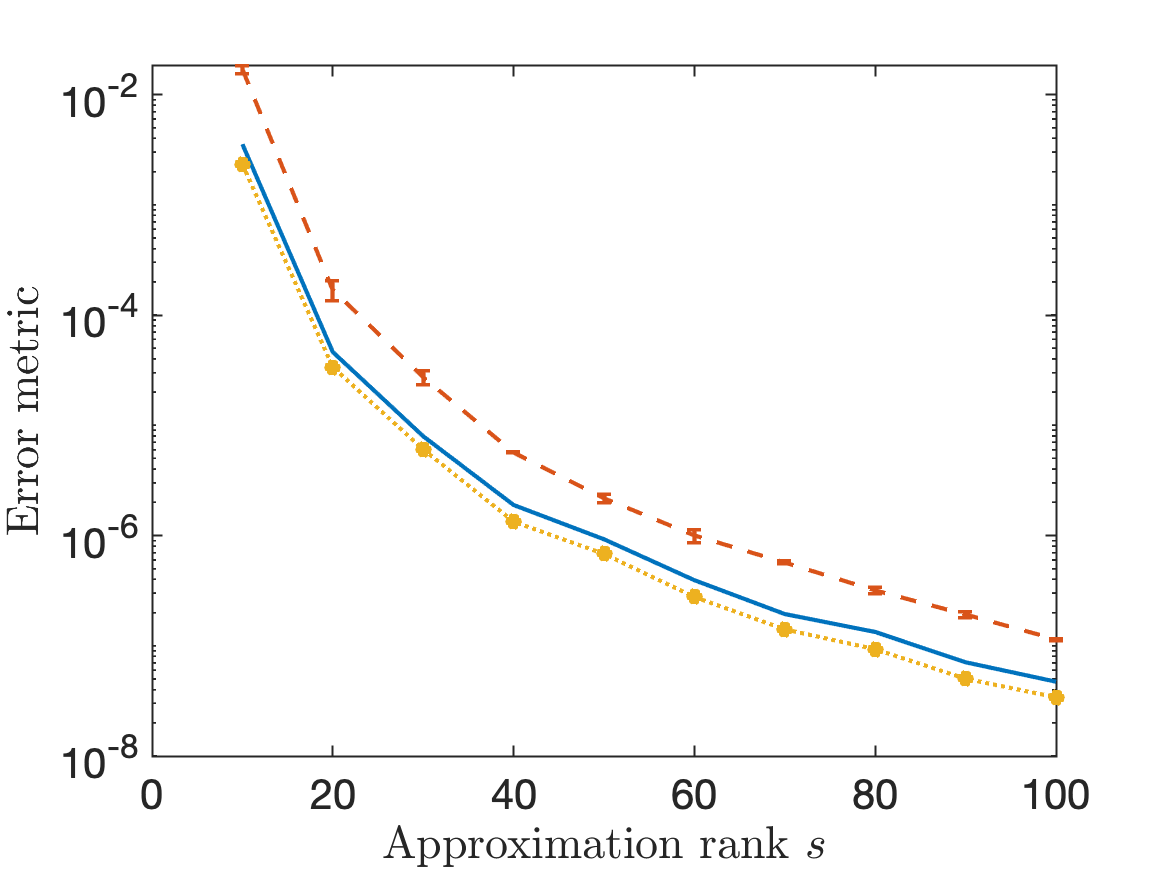}
      \caption*{\ref{eq:poly_decay}$(p \!=\! 2,R\!=\!5)$}
  \end{subfigure}
  \begin{subfigure}{0.32\textwidth}
      \includegraphics[width=0.99\textwidth]{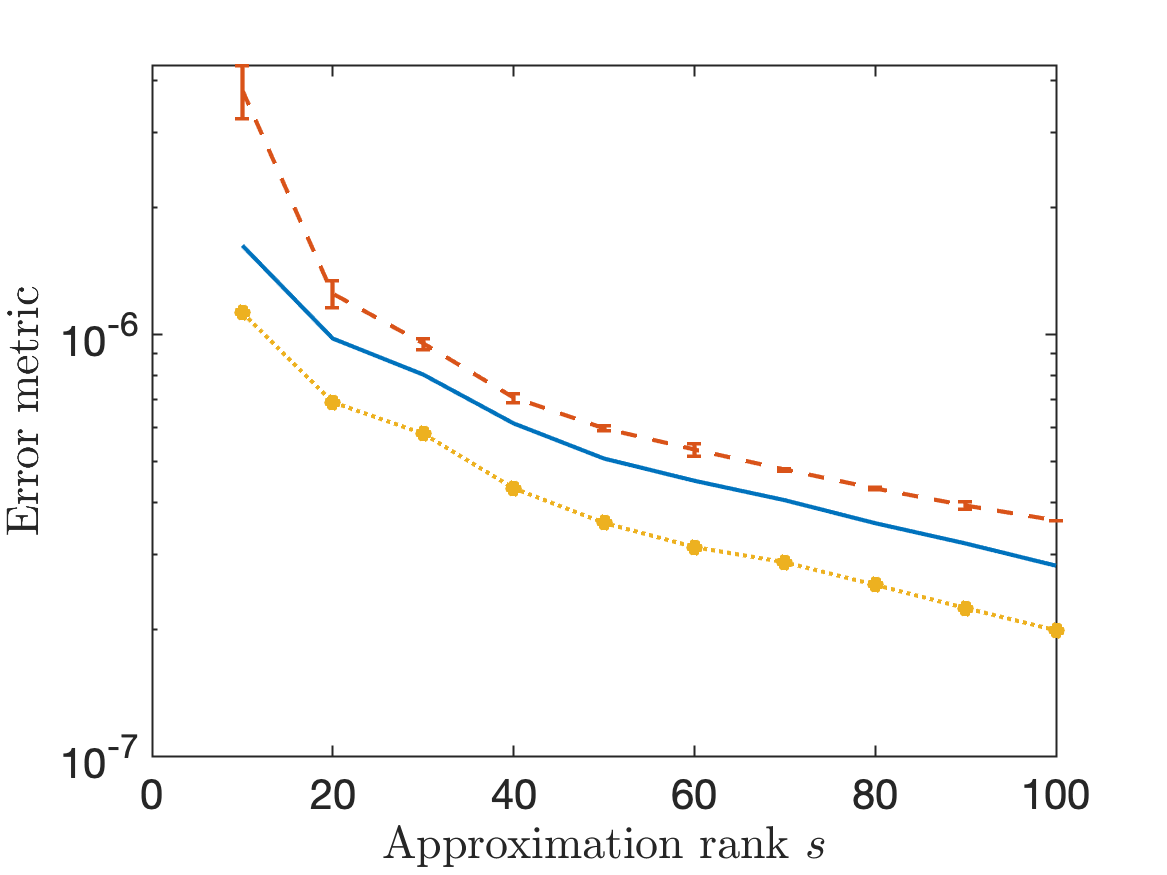}
      \caption*{\ref{eq:low_rank_plus_noise}$(\xi \!=\! 10^{-4},R\!=\!5)$}
  \end{subfigure}
  
  \begin{subfigure}{0.32\textwidth}
      \includegraphics[width=0.99\textwidth]{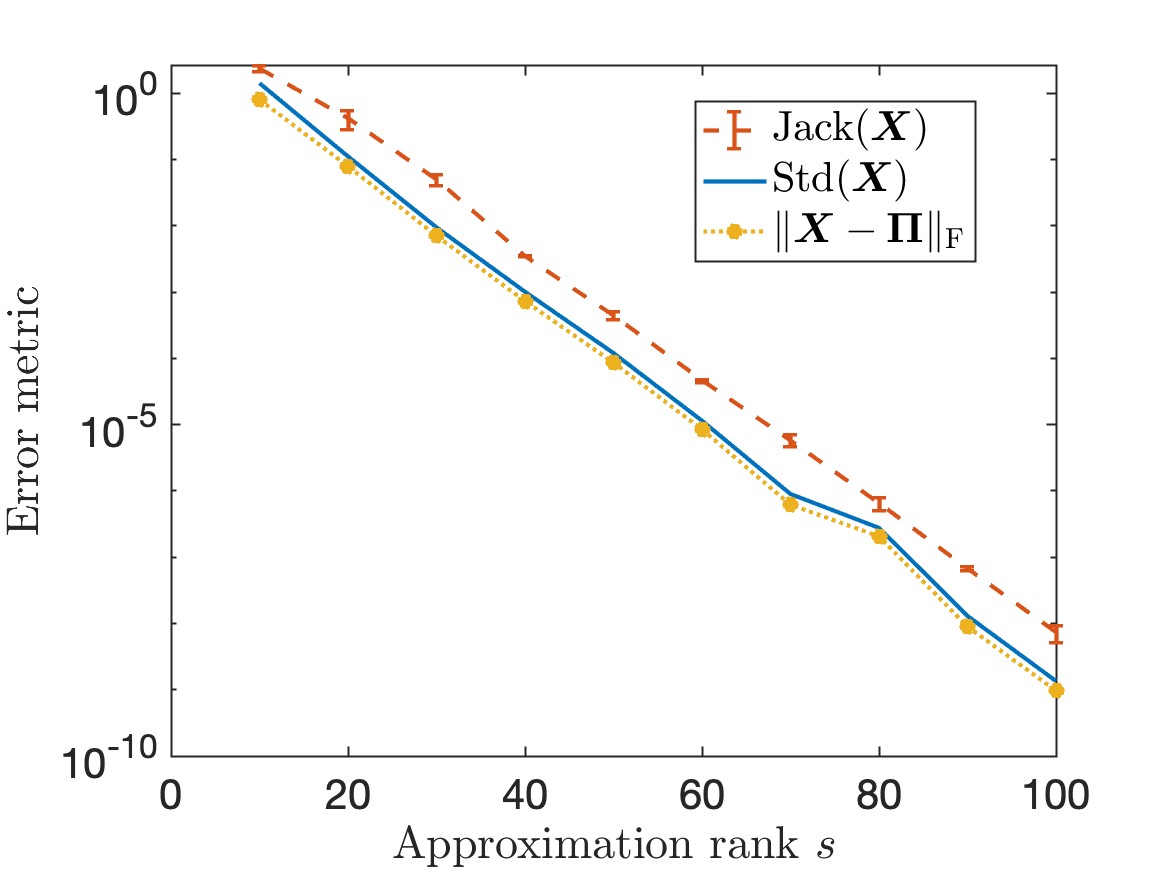}
      \caption*{\ref{eq:exp_decay}$(q \!=\! 0.05,R\!=\!5)$}
  \end{subfigure}
  \begin{subfigure}{0.32\textwidth}
      \includegraphics[width=0.99\textwidth]{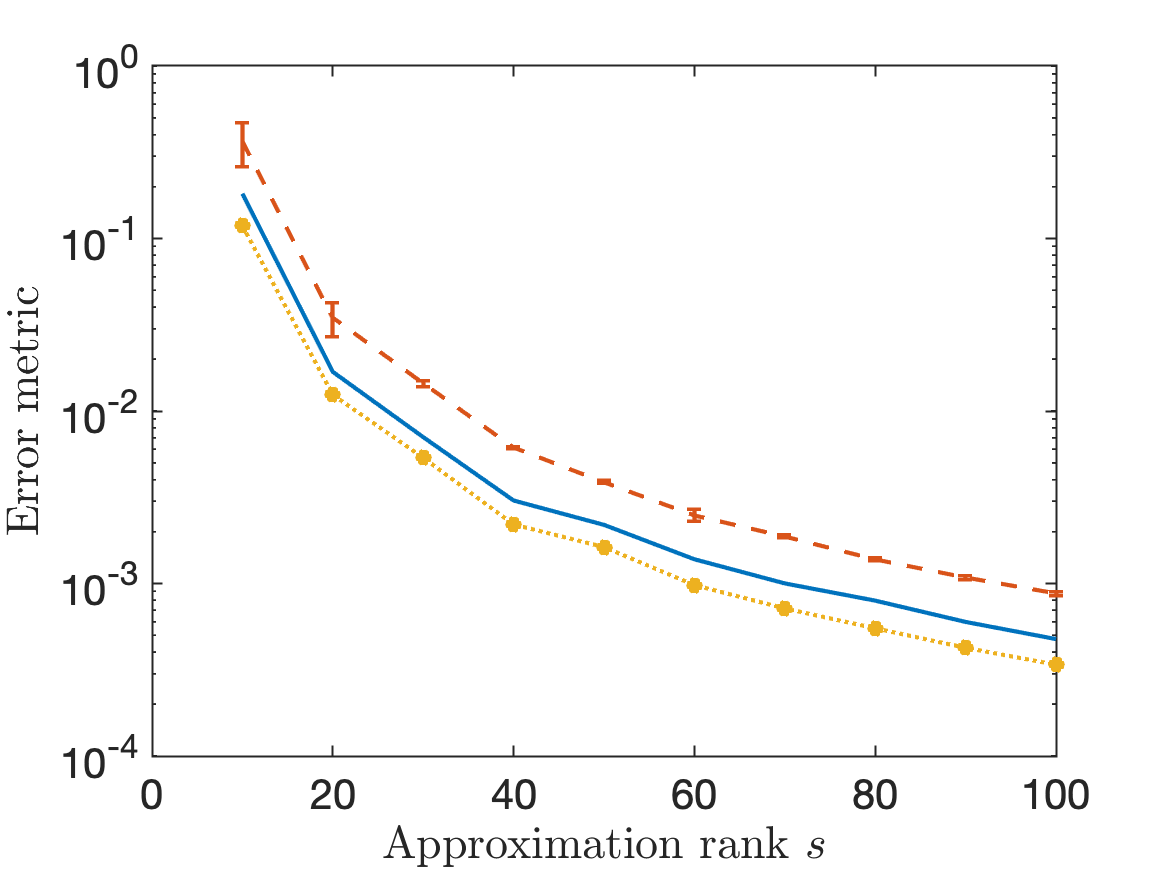}
      \caption*{\ref{eq:poly_decay}$(p \!=\! 1,R\!=\!5)$}
  \end{subfigure}
  \begin{subfigure}{0.32\textwidth}
      \includegraphics[width=0.99\textwidth]{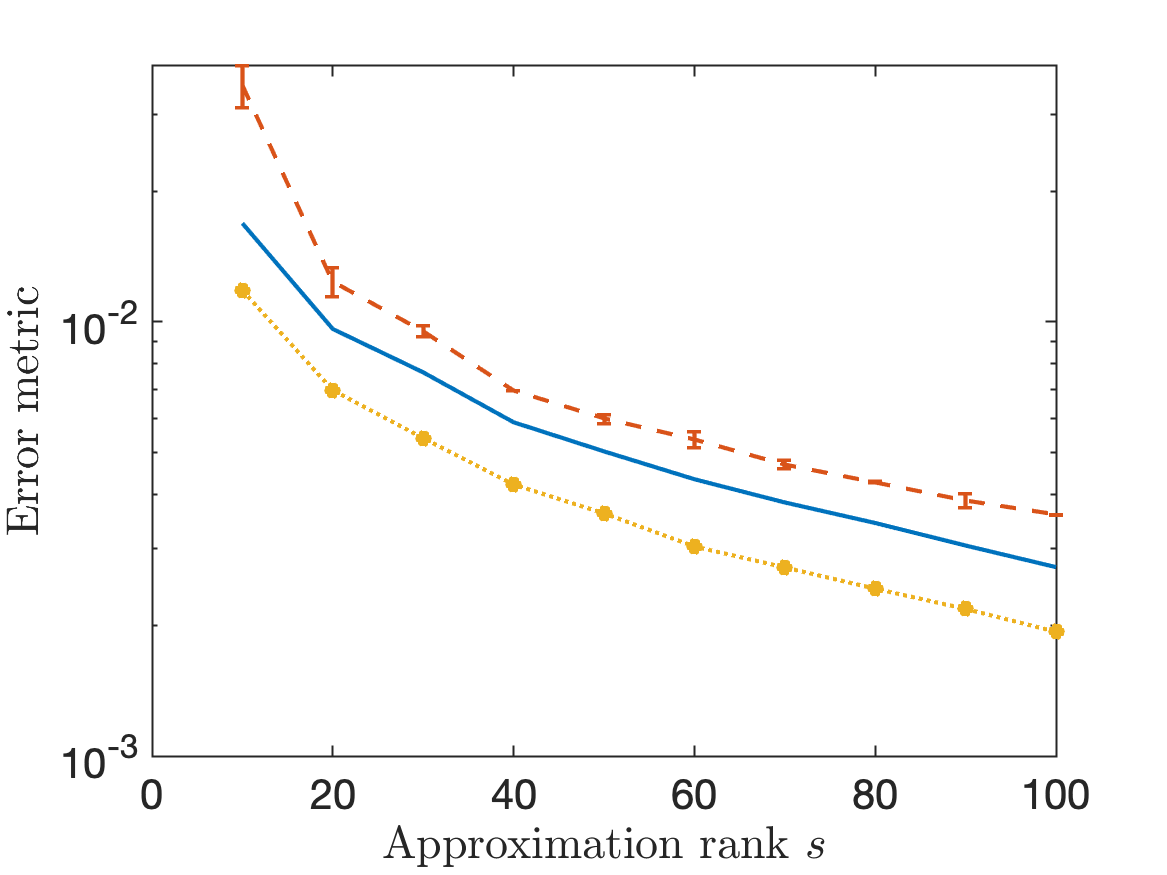}
      \caption*{\ref{eq:low_rank_plus_noise}$(\xi \!=\! 10^{-2},R\!=\!5)$}
  \end{subfigure}
  
  \begin{subfigure}{0.32\textwidth}
      \includegraphics[width=0.99\textwidth]{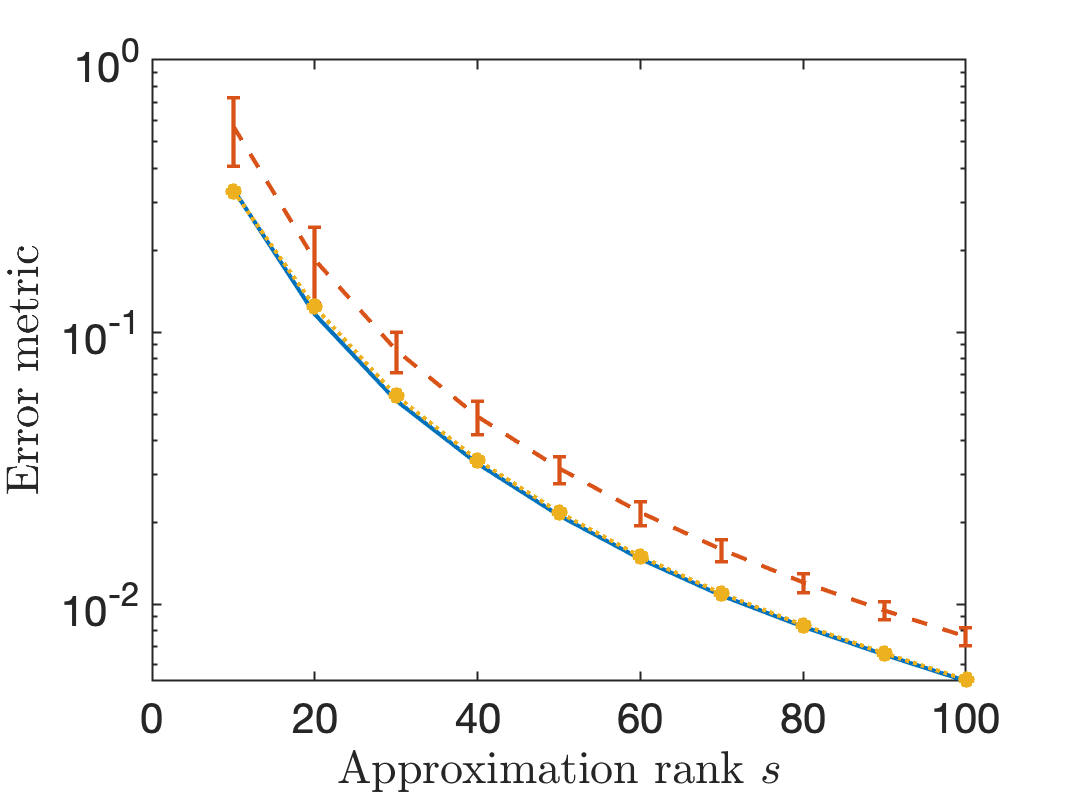}
      \caption*{Inverse-Poisson}
  \end{subfigure}
  \begin{subfigure}{0.32\textwidth}
      \includegraphics[width=0.99\textwidth]{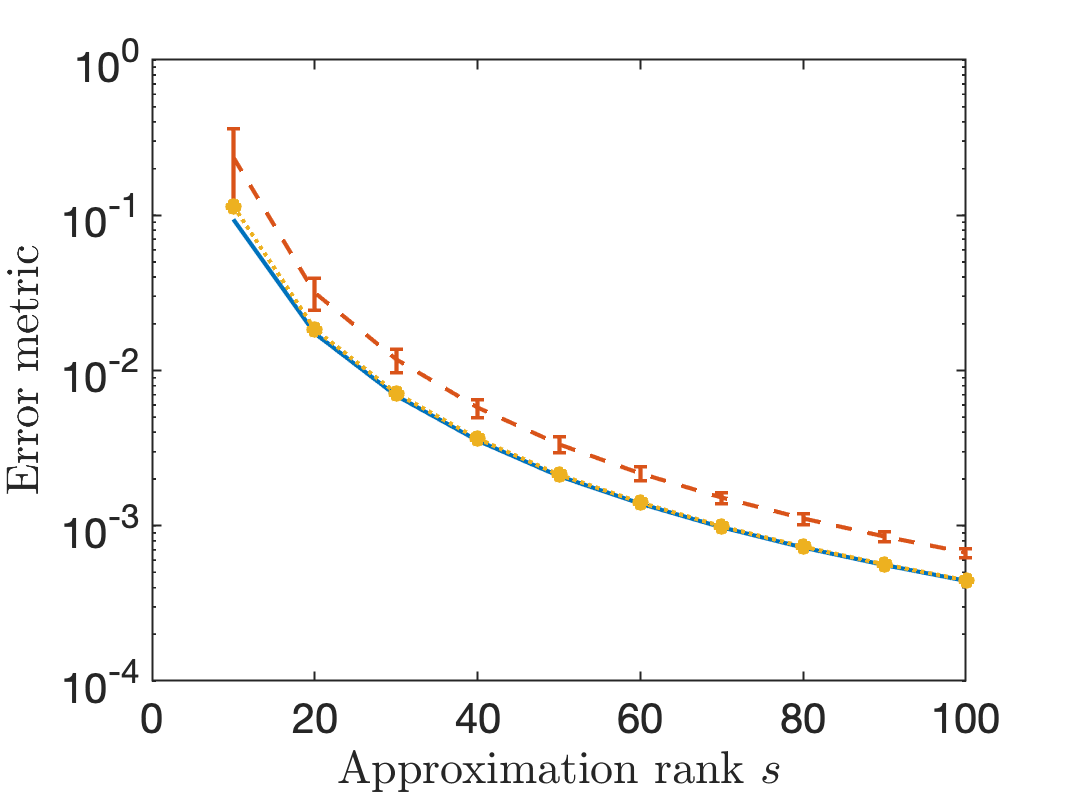}
      \caption*{Network}
  \end{subfigure}
  
  \mycaption{Matrix jackknife for spectral projectors}{Error, standard deviation, and jackknife estimate for randomized Nystr\"om ($q=0$) approximation $\mat{X}$ to the projector $\mat{\Pi}$ onto the span of the five dominant eigenvectors.}
  \label{fig:nystrom_jack}
\end{figure}

\section{MATLAB implementations} \label{sec:matlab}
In this section, we provide MATLAB R2022b implementations of the leave-one-out error estimate and jackknife variance estimate for the Nystr\"om approximation (\cref{sec:nystrom_code}). and randomized SVD (\cref{sec:rsvd_code}).

\subsection{Nystr\"om approximation} \label{sec:nystrom_code}
\Cref{list:nystrom} contains an implementation of Nystr\"om approximation with subspace iteration.
This code computes both the leave-one-out error estimate $\hat{\Err}(\mat{X},\mat{A})$ (outputted as \texttt{err}) and the jackknife estimate $\Jack(\mat{F}(\mat{X}))$ (outputted as \texttt{jack}) of a transformation $\mat{F}(\mat{X})$ of the Nystr\"om approximation. 
We assume the transformation $\mat{F}$ takes the form
\begin{equation*}
    \mat{F}(\mat{X}^{(j)}) = \mat{Q} \cdot \mathtt{transform}(\mat{W}_j,\mat{D}_j) \cdot \mat{Q}^* \quad \text{for every $j=1,\ldots,s$},
\end{equation*}
where $\mat{Q}$ is a fixed unitary (usually $\mat{V}$) independent of $j$ and $(\mat{W}_j,\mat{\Lambda}_j)$ are the eigendecomposition defined in \cref{eq:rank_one_modification}. 
Natural examples include spectral projectors $\mat{X} = \mat{V}\mat{\Lambda}\mat{V}^* \mapsto \mat{V}(:,\set{S})\mat{V}(:,\set{S})^*$ for $\set{S} \subseteq \{1,\ldots,s-1\}$ and truncation to rank $r$, $\mat{X} \mapsto \lowrank{\mat{X}}_r$.
We solve \cref{eq:rank_one_modification} using the LAPACK routine \texttt{dlaed9}, called in MATLAB (\cref{list:update_eig}) via a MEX file (\cref{list:update_eig_mex}).
An implementation of the Nystr\"om-accelerated spectral clustering procedure with jackknife variance estimation (\cref{alg:spectral_clustering}) is also provided in \cref{list:spectral_clustering}.

\begin{lstlisting}[float=t,frame=single, caption={\texttt{nystrom.m}: Nystr\"om approximation with leave-one-out error estimator and jackknife variance estimate},label={list:nystrom}]
function [V,Lambda,err,jack] = nystrom(A,s,q,transform)
%% Nystrom approximation
Omega = randn(size(A,1),s);
if q > 0
    Omega0 = Omega; Z = A*Omega; Omega = Z;
end
for i = 2:q
    Omega = A*Omega;
end
Y = A*Omega;
nu = eps()*norm(Y);
Y = Y + nu*Omega;
[Q,R] = qr(Y,0);
H = Omega'*Y;
C = chol((H+H')/2);
[U,Sigma,~] = svd(R/C,'econ');
Lambda = max(Sigma^2 - nu*eye(s), 0);
V = Q*U;

%% Update formula
diagHinv = diag(inv(H))';
T = ((U'*R/C)/(C')) .* diagHinv.^(-0.5);

%% Leave-one-out error estimator
if q ~= 0
  VO = V'*Omega0;
  err = norm(Z-V*(Lambda*VO)+V*(T.*diag(T'*VO)'),'fro')/sqrt(s);
else
  err = norm(((R/C)/(C')) ./ diagHinv, 'fro')/sqrt(s);
end

%% Jackknife variance estimate
replicates = zeros(s,s,s);
lambda = diag(Lambda); 
for i = 1:s
  [Q,D] = update_eig(lambda, T(:,i), -1);
  replicates(:,:,i) = transform(Q,D);
end
jack = norm(replicates - mean(replicates, 3), 'fro');
end
\end{lstlisting}
\begin{lstlisting}[float=t,frame=single, caption={\texttt{update\_eig.m}: Solve the updated eigenvalue problem \cref{eq:rank_one_modification}},label={list:update_eig}]
function [Q,D] = update_eig(d,v,c)
%% Negate, sort, and normalize
negative = c < 0;
if negative; d = -d; c = -c; end
[d,p] = sort(d,'ascend');
v = v(p);
c = c * norm(v)^2;
v = v / norm(v);

%% Compute update
try 
    [Q,D] = update_eig_mex(d,v,c);
catch ME
    if strcmp(ME.identifier, "MATLAB:UndefinedFunction")
        error(['Need to build MEX function '...
            '"mex -v update_eig_mex.c -lmwlapack"'])
    else
        rethrow(ME)
    end
end

%% Undo negation and sorting
Q(p,:) = Q;
if negative; D = -D; end
D = diag(D);
end
\end{lstlisting}
\begin{lstlisting}[float=t,frame=single, caption={\texttt{update\_eig\_mex.c}: MEX file to solve the updated eigenvalue problem \cref{eq:rank_one_modification}},label={list:update_eig_mex}]
// update_eig_mex.c

#include "mex.h"
#include "lapack.h"

void mexFunction(int nlhs, mxArray *plhs[], int nrhs, const mxArray *prhs[])
{
#if MX_HAS_INTERLEAVED_COMPLEX
    mxDouble *d, *v, *c, *Q, *l;
#else
    double *d, *v, *c, *Q, *l;
#endif
    ptrdiff_t n;      /* matrix dimensions */

#if MX_HAS_INTERLEAVED_COMPLEX
    d = mxGetDoubles(prhs[0]); v = mxGetDoubles(prhs[1]);
    c = mxGetDoubles(prhs[2]);
#else
    d = mxGetPr(prhs[0]); v = mxGetPr(prhs[1]);
    c = mxGetPr(prhs[2]);
#endif
    n = mxGetM(prhs[0]);  

    /* create output matrix C */
    plhs[0] = mxCreateDoubleMatrix(n, n, mxREAL);
    plhs[1] = mxCreateDoubleMatrix(n, 1, mxREAL);;
#if MX_HAS_INTERLEAVED_COMPLEX
    Q = mxGetDoubles(plhs[0]); l = mxGetDoubles(plhs[1]);
#else
    Q = mxGetPr(plhs[0]); l = mxGetPr(plhs[1]);
#endif

    double *work = malloc(sizeof(double) * n*n);

    ptrdiff_t info; ptrdiff_t one = 1;

    /* Pass arguments to Fortran by reference */
    dlaed9(&n, &one, &n, &n, l, work, &n, c, d, v, Q, &n, &info);

    free(work);
}
\end{lstlisting}
\begin{lstlisting}[float=t,frame=single, caption={\texttt{spectral\_clustering.m}: Spectral clustering with jackknife variance estimation, implementing algorithm \cref{alg:spectral_clustering}},label={list:spectral_clustering}]
function [clusters, jack, coords] = spectral_clustering(pts,sigma,l,s,q)
%% Set up
K = kernelexp(-pdist2(pts,pts,"euclidean").^2 / (2*sigma^2));
d = sum(K, 2);
A = d.^(-0.5) .* (K .* (d').^(-0.5));

%% Nystrom approximation
Omega = randn(size(A,1),s);
for i = 1:q
    Omega = A*Omega;
end
Y = A*Omega;
nu = eps()*norm(Y);
Y = Y + nu*Omega;
[Q,R] = qr(Y,0);
H = Omega'*Y;
C = chol((H+H')/2);
[U,Sigma,~] = svd(R/C,'econ');
lambda = max(diag(Sigma).^2 - nu, 0);
V = Q*U;

%% Spectral clustering
coords = d.^(-0.5).*V(:,1:l);
clusters = kmeans(coords, l);

%% Jackknife variance estiamte
G = V' * (V ./ d);
diagHinv = diag(inv(H))';
T = (U'*R/H) .* diagHinv.^(-0.5);
replicates = zeros(s,l,s); avg = zeros(s);
for i = 1:s
    Q = update_eig(lambda,T(:,i),-1);
    Q = Q(:,1:l);
    Q = Q / sqrt(norm(Q'*G*Q,'fro'));
    replicates(:,:,i) = Q;
    avg = avg + Q*Q'/s;
end
GavgG = G*avg*G;
jack = s * trace(GavgG*avg);
for i = 1:s
    Q = replicates(:,:,i);
    jack = jack - 2*trace(Q'*GavgG*Q) + norm(Q'*G*Q, 'fro')^2;
end
jack = sqrt(jack);
end
\end{lstlisting}

\subsection{Randomized SVD} \label{sec:rsvd_code}
\Cref{list:randsvd} contains an implementation of the randomized SVD with subspace iteration.
This code computes both the leave-one-out error estimate $\hat{\Err}(\mat{X},\mat{A})$ (outputted as \texttt{err}) and the jackknife estimate $\Jack(\mat{F}(\mat{X}))$ (outputted as \texttt{jack}) of a transformation $\mat{F}(\mat{X})$ of the randomized SVD approximation. 
We assume the transformation $\mat{F}$ takes the form
\begin{equation*}
    \mat{F}(\mat{X}^{(j)}) = \mat{Q}_1 \cdot \mathtt{transform}(\mat{U}_j,\mat{\Sigma}_j,\mat{V}_j) \cdot \mat{Q}_2^* \quad \text{for every $j=1,\ldots,s$},
\end{equation*}
where $\mat{Q}_1$ and $\mat{Q}_2$ are fixed unitaries independent of $j$ and $\mat{U}_j,\mat{\Sigma}_j,\mat{V}_j$ are defined in \cref{eq:rank_one_rsvd}. 
Natural examples include projectors onto left singular subspaces ($\mat{Q}_1 = \mat{Q}_2 = \mat{U}$), right singular subspaces ($\mat{Q}_1 = \mat{Q}_2 = \mat{V}$), and the truncation of $\mat{X}$ to rank $r < s$ ($\mat{Q}_1 = \mat{U}$, $\mat{Q}_2 = \mat{V}$).
We have not implemented the algorithm of \cite{GE95} to solve \cref{eq:rank_one_rsvd} in $\order(s^2)$ operations.

\begin{lstlisting}[float=t,frame=single, caption={\texttt{randsvd.m}: Randomized SVD with leave-one-out error estimator and jackknife variance estimate},label={list:randsvd}]
function [U,S,V,err,jack] = randsvd(A,s,q,transform)
%% Randomized SVD
Omega = randn(size(A,2),s);
Z = A*Omega; Y = Z;
for iter = 1:q
    Y = A*(A'*Y);
end
[Q,R] = qr(Y,0);
C = Q'*A;
[W,S,V] = svd(C,'econ');
U = Q * W;

%% Leave-one-out error estimator
T = cnormc(inv(R')); 
if q ~= 0
  QZ = Q'*Z;
  err = norm(Z-Q*QZ+Q*T.*(diag(T'*QZ).'),'fro')/sqrt(s);
else
  err = norm(1./vecnorm(inv(R')))/sqrt(s);
end

%% Jackknife variance estimate
WT = W'*T;
replicates = zeros(s,s,s);
for i = 1:s
  [UU,SS,VV] = svd(S - WT(:,i)*(WT(:,i)'*S));
  replicates(:,:,i) = transform(UU,SS,VV);
end
jack = norm(replicates - mean(replicates, 3), 'fro');
end

%% Normalize columns of a possibly complex matrix
function M = cnormc(M)
M = M ./ vecnorm(M,2,1);
end
\end{lstlisting}

\section{Derivation of leave-one-out error estimator for randomized SVD} \label{sec:loo-derivation}
In this section, we provide a derivation for the leave-one-out error estimator for the randomized SVD without subspace iteration $q = 0$ given in \cref{alg:rsvd} and for $q > 0$ given in \cref{list:randsvd}.

\subsection{The case $q=0$}
First suppose $q=0$, and denote $\mat{G} = (\mat{R}^*)^{-1}$ with $j$th column $\vec{g}_j$.
By the update formula \cref{eq:rsvd_Q_update}, we have
\begin{equation*}
  \mat{Q}^{(j)} \mleft( \mat{Q}^{(j)} \mright)^* = \mat{Q} \mleft( \Id - \frac{\vec{g}_j^{\vphantom{*}} \vec{g}_j^*}{\norm{\vec{g}_j}^2} \mright)\mat{Q}^*.
\end{equation*}
The leave-one-out error estimator is
\begin{align*}
  \hat{\Err}^2
  &= \frac{1}{s} \sum_{j=1}^s \norm{ \mleft( \Id - \mat{Q}^{(j)} \mleft( \mat{Q}^{(j)} \mright)^*\mright) \mat{A} \vec{\omega}_j }^2.
\end{align*}
Combining the two previous displays gives
\begin{equation} \label{eq:loo-rsvd-all-out}
  \hat{\Err}^2 = \frac{1}{s} \sum_{j=1}^s \norm{ \mleft( \Id - \mat{Q}\mat{Q}^*\mright) \mat{A} \vec{\omega}_j + \frac{\mat{Q} \vec{g}_j (\mat{Q}\vec{g}_j)^*\mat{A}\vec{\omega}_j}{\norm{\vec{g}_j}^2} }^2.
\end{equation}
The matrix $\mat{Q}\mat{Q}^*$ is the orthoprojector onto the column span of $\mat{A}\mat{\Omega}$, so $(\Id - \mat{Q}\mat{Q}^*)\mat{A}\vec{\omega}_j = \vec{0}$.
In addition, the Euclidean norm is unitarily invariant, so the leave-one-out estimator simplies as follows:
\begin{align*}
  \hat{\Err}^2 = \frac{1}{s} \sum_{j=1}^s \norm{\vec{g}_j}^{-2} \cdot |\vec{g}_j^*\mat{Q}^*\mat{A}\vec{\omega}_j^{\vphantom{*}}|^2.
\end{align*}
Since $\mat{A}\mat{\Omega} = \mat{Q}\mat{R}$, $\mat{Q}^*\mat{A}\vec{\omega}_j$ is the $j$th column of $\mat{R}$.
By the identity $\mat{G} = (\mat{R}^*)^{-1} = (\mat{R}^{-1})^*$, $\vec{g}_j^*$ is the $j$th row of $\mat{R}^{-1}$.
Ergo, $\vec{g}_j^*\mat{Q}^*\mat{A}\vec{\omega}_j^{\vphantom{*}} = 1$ and
\begin{equation*}
  \hat{\Err}^2 = \frac{1}{s} \sum_{j=1}^s \norm{\vec{g}_j}^{-2}.
\end{equation*}
This is the formula for the leave-one-out error estimator used in \cref{alg:rsvd}.

\subsection{The case $q > 0$}
Now suppose $q > 0$.
Following the notation used in \cref{list:randsvd}, we let $\mat{Z} = \mat{A}\mat{\Omega}$ and $\mat{Y} = (\mat{A}\mat{A}^*)^q \mat{Z} = \mat{Q}\mat{R}$.
Additionally, let $\mat{T}$ denote $(\mat{R}^*)^{-1}$ with its columns scaled to unit norm.
Denote the $j$th column of $\mat{T}$ and $\mat{Z}$ as $\vec{t}_j$ and $\vec{z}_j$.
Translating \cref{eq:loo-rsvd-all-out} to the present notation, we obtain
\begin{align*}
  \hat{\Err}^2
  &= \frac{1}{s} \sum_{j=1}^s \norm{ \mleft( \Id - \mat{Q}\mat{Q}^*\mright) \vec{z}_j + \mat{Q} \vec{t}_j \vec{t}_j^*(\mat{Q}^*\vec{z}_j) }^2 \\
  &= \frac{1}{s} \norm{\mat{Z} - \mat{Q}(\mat{Q}^*\mat{Z}) + \mat{Q}\mat{T} \cdot \diag \mleft( \vec{t}_j^*(\mat{Q}\vec{z}_j) : j=1,\ldots,s  \mright)}_{\rm F}^2
\end{align*}
In the second line, we package the sum of squared Euclidean norms of vectors into a squared matrix Frobenius norm.
We have obtained the formula that appears in \cref{list:randsvd}.

\end{document}